\documentclass[dvipdfmx]{article}
\usepackage{amsmath, amssymb}
\usepackage{amsthm}
\usepackage{multirow,bigdelim}
\usepackage{ascmac}
\usepackage{fancybox}
\usepackage{color}
\usepackage{nccmath} 

\usepackage{multicol}

\usepackage[top=30truemm,bottom=30truemm,left=25truemm,right=25truemm]{geometry}
\usepackage{graphicx}
\usepackage{xcolor}
\usepackage{here}

\theoremstyle{definition}
\newtheorem{thm}{Theorem}[section]
\newtheorem{defi}[thm]{Definition}
\newtheorem{lemm}[thm]{Lemma}
\newtheorem{prop}[thm]{Proposition}
\newtheorem{remark}[thm]{Remark}
\newtheorem{coro}[thm]{Corollary}
\newtheorem{problem}[thm]{Problem}

\begin{document}

\title{Sasaki-Einstein orbits in compact Hermitian symmetric spaces}
\author{Yuuki Sasaki}
\date{}
\maketitle

\begin{abstract}

\rm{
The aim of the present papar is to study the orbits of the isotropy gourp action on an irreducible Hermitian symmetric space of compact type.
Specifically, we examine the properties of these orbits as {\it CR} submanifolds of a K\"{a}hler manifold.
Our focus is on the leaves of the totally real distribution, and we investigate the properties of leaves as a Riemannian submanifold.
In particular, we prove that any leaf is a totally geodesic submanifold of the orbit.
Additionally, we explore the conditions under which each leaf becomes a totally geodesic submanifold of the ambient space.
The integrability of the complex distribution is also studied.
Moreover, we analyze a contact structure of orbits where the rank of the totally real distribution is 1.
We obtain a classification of the orbits that possess either a contact structure or a Sasakian structure compatible with the complex structure on the ambient space.
Furthermore, we classify those Sasaki orbits that are Einstein with respect to the induced metric.
Specifically, we completely detemine Sasaki-Einstein orbits.
}

\end{abstract}

%%%%%%%%%%%%%%%%%%%%%%%%%%%%%%%%%%%%%%%%%%%%%%%%%%%%%%%%%%%%%%%%%%%%%%%%%%%%%%%%%%%%%%%%%%%%%%%%%%%%%%%%%%%%%%%%%%%%%%%%%%%%%%%%%%%%%%%%%%%%%%%%%%%%%%%

\section{Introduction}\label{sec1}

Complex submanifolds and totally real submanifolds are important submanifolds in K\"{a}hler geometry.
A complex submanifold is a submanifold such that the tangent space at any point is invariant under the complex structure.
In contrast, a totally real submanifold is characterized by the property that the image of the tangent space under the complex structure is contained in the normal space.
In this sense, complex and totally real submanifolds are opposite submanifolds.
In particular, a totally real submanifold whose dimension is half that of the ambient space is called a Lagrangian submanifold.
Both complex and totally real submanifolds have been extensively studied by many mathematicians.
In the present paper, we consider $CR$ submanifolds, which are an analogue of complex and totally real submanifolds.
A submanifold of a K\"{a}hler manifold is called a $CR$ submanifold if its tangent bundle is decomposed into a complex subbundle and a totally real subbundle.
By definition, complex submanifolds and totally real submanifolds are special cases of $CR$ submanifolds.
We call a $CR$ submanifold that is neither complex nor totally real a proper $CR$ submanifold.
$CR$ submanifolds were introduced by Bejancu \cite{Bejancu}, and the fundamental theory has been developed by Bejancu, Chen, Kon, Yano, and several other geometers \cite{Bejancu2}, \cite{Bejancu-Kon-Yano}, \cite{Chen}, \cite{Chen1}.
One of the most interesting properties of $CR$ submanifolds is the integrability of the totally real distribution.
Due to this property, $CR$ submanifolds have a foliation, and understanding the natures of the leaves as a Riemannian submanifold is an interesting question.
An example of a $CR$ submanifold is a real hypersurface of a K\"{a}hler manifold.
In a real hypersurface, the image of the normal bundle under the complex structure becomes a totally real distribution.
Obviously, the rank of this totally real distribution is $1$.
A real hypersurface is called a Hopf real hypersurface if any integral curve of the totally real distribution is a geodesic.
Hopf real hypersurfaces are fundamental objects in submanifold theory, and have been actively studied by many mathematicians \cite{Borisenko}, \cite{Cecil-Ryan},  \cite{Kimura2} .
For example, it is known that any homogeneous real hypersurface is a Hopf real hypersurface in a complex projective space  \cite{Kimura}, \cite{Takagi}.

The second example of $CR$ submanifolds is given by polar and coisotropic actions.
An action of a compact Lie group of isometries on a Riemannian manifold $M$ is called polar if there exists a submanifold $N$ of $M$ such that each $G$-orbit intersects $N$ and $N$ is orthogonal to each $G$-orbit at all common points \cite{Palais-Terng}.
The submanifold $N$ is called a section.
If the section is flat, the polar action is called a hyperpolar action.
For example, the isotropy group action and Hermann actions on a symmetric space are hyperpolar actions (note that the isotropy group action is a special case of Hermann actions).
An action on a compact K\"{a}hler manifold is called coisotropic if any principal orbit is a coisotropic submanifold, that is, the K\"{a}hler form vanishes on the normal space of the tangent space at any point of any principal orbit.
By definition, any principal orbit of a coisotropic action is a $CR$ submanifold.
It was proved by Podest\`{a} and Thorbergsson that any polar action on an irreducible compact homogeneous K\"{a}hler manifold is coisotropic \cite{Podesta-Thorbergsson}.
Hence, any polar action produces many $CR$ submanifolds.
Recently, D\'{i}az-Ramos, Dom\'{i}nguez-V\'{a}zquez, and P\'{e}rez-Barral classified homogeneous $CR$ submanifolds arising as orbits of a subgroup of the solvable part of the Iwasawa decomposition in the complex hyperbolic space \cite{Diaz}.

Sasakian geometry is an analogue of K\"{a}hler geometry.
A Sasakian manifold was introduced by Sasaki in the 1960s \cite{Sasaki} and is characterized by having a K\"{a}hler structure on the metric cone.
In a Sasakian manifold, there exists a one-dimensional foliation, called the Reeb foliation, which has a transverse K\"{a}hler structure.
Hence, a Sasakian manifold is related to these two K\"{a}hler manifolds.
For example, the unit sphere $S^{2n-1}$ of $\mathbb{C}^{n}$ with the standard inner product is a typical example of a Sasakian manifold.
In this case, the metric cone is $\mathbb{C}^{n} - \{ 0 \}$, and the Reeb foliation is given by the Hopf fibration.
Moreover, the leaf space of the Reeb foliation is the complex projective space $\mathbb{C}P^{n-1}$.
The unit sphere $S^{2n-1}$ is Einstein with the induced metric and is an example of a Sasaki-Einstein manifold.
Influenced by the connection between Sasaki-Einstein manifolds and the AdS/CFT correspondence, Sasaki-Einstein manifolds have been actively studied.
Additionally, the metric cone of a Sasaki-Einstein manifold is a Ricci-flat K\"{a}hler manifold and an example of a Calabi-Yau manifold.
Special Lagrangian submanifolds of Calabi-Yau manifolds play an important role in string theory and have been extensively studied by many mathematicians.

In the complex projective space, a principal orbit of the isotropy group action is also called a geodesic sphere, and every geodesic sphere is a real hypersurface.
It is known that a geodesic sphere is a Hopf real hypersurface.
Moreover, Berndt and Adachi, Kameda, and Maeda studied the condition under which a homogeneous real hypersurface has a Sasakian structure with the Reeb vector field lying in the totally real foliation \cite{Adachi-Kameda-Maeda}, \cite{Berndt}.
According to their results, in the complex projective space, the only real hypersurface satisfying this condition is a geodesic sphere.
Thus, we naturally consider the following problem.

\begin{problem} \label{1}
Let $N$ be any orbit of the isotropy group action on an irreducible Hermitian symmetric space $M$ of compact type.

\begin{itemize}

\item[(i)]
Is $N$ a $CR$ submanifold ? (see Theorem \ref{2-4-8})
Are there $CR$ orbits such that any leaf of a totally real foliation is a totally geodesic submanifold of the orbit $N$ or the ambient space $M$ ? (see Theorem \ref{3-3-7})

\item[(ii)]
For $CR$ orbits where the rank of the totally real distribution is $1$, are there a contact form that is identically zero on the complex distribution ? (see Theorem \ref{4-1-1})

\item[(iii)]
For such contact orbits, are there a Sasakian structure by the induced metric such that the Reeb vector field lies in the totally real distribution ? (see Theorem \ref{4-3-2})

\item[(iv)]
For such Sasakian orbits, are there orbits that are Einstein with respect to the induced metric ? (see Theorem \ref{5-3-1})

\end{itemize}
\end{problem}

In the present paper, we consider Problem \ref{1}.
We review the concept of a contact manifold.
Let $M$ be a $2n+1 \ (n \geq 1)$ dimensional manifold.
If there exists a 1-form $\eta$ on $M$ such that $\eta \wedge (d\eta)^{n} \not= 0$ at every point of $M$, we call $(M, \eta)$ a contact manifold and $\eta$ a contact form on $M$.

The present paper is organized as follows.
In Section \ref{s2}, we provide the preliminaries used in this paper.
In Subsection \ref{s2-1}, we recall the definition and some fundamental properties of $CR$ submanifolds and Sasakian manifolds.
In Subsection \ref{s2-2}, we discuss the construction and properties of irreducible Hermitian symmetric spaces of compact type, including their classification.
In Subsection \ref{s2-3}, we summarize the fundamental theory of the restricted root system.
Much of these properties is referred to \cite{Helgason}, \cite{Takeuchi}, and \cite{Takeuchi2}.
In Subsection \ref{s2-4}, we address the restricted root system of an irreducible Hermitian symmetric space and show that any orbit other than the principal orbits of the isotropy group action is a $CR$ submanifold.
In Section \ref{s3}, we consider the remaining part of Problem \ref{1} (i).
We show that for any $CR$ orbit, each leaf of the totally real foliation is a totally geodesic submanifold of the orbit.
Moreover, we provide the condition under which such leaves are also totally geodesic submanifolds in the ambient space $M$.
We also study the integrability of the complex distribution.
In Section \ref{s4}, we address Problem \ref{1} (ii) and (iii), and demonstrate that for any irreducible Hermitian symmetric space, there exists only one orbit that is a Sasaki $CR$ submanifold for any invariant metric.
In Section \ref{s5}, we study Problem \ref{1}(v) and prove that, for any irreducible Hermitian symmetric space, except for the complex projective space and some complex Grassmannians, there exist only one invariant metric for which the Sasaki $CR$ orbit is Einstein with respect to the induced metric.

We introduce some notations used in the present paper.
Let $G_{k}(\mathbb{C}^{n})$ denote the set of all $k$-dimensional complex subspaces of $\mathbb{C}^{n} \ (k \leq n)$.
Hence, $\mathbb{C}P^{n} = G_{1}(\mathbb{C}^{n+1})$.
Moreover, let $\tilde{G}_{2}(\mathbb{R}^{n})$ be the set of all $2$-dimensional oriented real subspaces of $\mathbb{R}^{n}\ (n \geq 2)$.
Then, 
\[
G_{k}(\mathbb{C}^{n}) = SU(n)/S(U(k) \times U(n-k)), \quad \tilde{G}_{2}(\mathbb{R}^{n}) = SO(n)/SO(2) \times SO(n-2).
\]
Moreover, we set
\[
EIII = E_{6}/\big( U(1) \times Spin(10)/\mathbb{Z}_{4} \big), \quad EVII = E_{7}/\big( U(1) \times E_{6}/ \mathbb{Z}_{3} \big).
\]

%%%%%%%%%%%%%%%%%%%%%%%%%%%%%%%%%%%%%%%%%%%%%%%%%%%%%%%%%%%%%%%%%%%%%%%%%%%%%%%%%%%%%%%%%%%%%%%%%%%%%%%%%%%%%%%%%%%%%%%%%%%%%%%%%%%%%%%%%%%%%%%%%%%%%%%%

\section{Preliminaries} \label{s2}

\subsection{$CR$ submanifolds} \label{s2-1}

In this section, we recall fundamental properties of $CR$ submanifolds in a K\"{a}hler manifold.
Let $(M, J, g)$ be a K\"{a}hler manifold, where $J$ is a complex structure and $g$ is a Riemannian metric.

\begin{defi} \cite{Bejancu}
Let $N$ be a submanifold of $M$.
We say that $N$ is a {\it $CR$ submanifold} if there exists a smooth distribution $\mathcal{D}$ on $N$ satisfying the following conditions:

(i) $\mathcal{D}$ is $J$-invariant, that is, $J(\mathcal{D}_{x}) \subset \mathcal{D}_{x}$ for any $x \in N$.

(ii) Let $\mathcal{D}^{\perp}$ be the orthogonal complement distribution of $\mathcal{D}$. 
Then, $J(\mathcal{D}^{\perp}_{x}) \subset T_{x}^{\perp}N$ for any $x \in N$, where $T_{x}^{\perp}N$ is the orthogonal complement of $T_{x}N$ in $T_{x}M$.
\end{defi}

In this setting, a $CR$ submanifold $N$ is a complex submanifold if and only if $\dim \mathcal{D}^{\perp} = 0$, and $N$ is a totally real submanifold if and only if $\dim \mathcal{D} = 0$.
If $\dim \mathcal{D} = 0$ and $2\dim N = \dim M$, then $N$ is a Lagrangian submanifold.
We say that a $CR$ submanifold $N$ is {\it proper} if $N$ is neither a complex submanifold nor a totally real submanifold.
For a $CR$ submanifold, $\mathcal{D}$ (resp.\ $\mathcal{D}^{\perp}$) is called a {\it complex distribution} (resp.\ a {\it totally real distribution}).
Let $h$ be the second fundamental form of $N$.

\begin{lemm}\cite{Chen} \label{2-1}
For any $CR$ submanifold, the totally real distribution $\mathcal{D}^{\perp}$ is integrable.
Moreover, the complex distribution $\mathcal{D}$ is integrable if and only if 
\[
g \big( h(X, JY), JZ \big) = g \big( h(JX, Y), JZ \big)
\]
for any $x \in N, X,Y \in \mathcal{D}, Z \in \mathcal{D}^{\perp}$.
\end{lemm}

By Proposition \ref{2-1}, we can define the foliation $\mathcal{F}$ of $N$ by the totally real distribution $\mathcal{D}^{\perp}$.
The foliation $\mathcal{F}$ is called the {\it totally real foliation} of $N$.
Lemma \ref{2-2} easily follows from Proposition \ref{2-1}.

\begin{lemm} \label{2-2}
For any proper $CR$ submanifold $N$, if there exist $X \in \mathcal{D}$ and $Z \in \mathcal{D}^{\perp}$ such that 
\[
g \big( h(X, X) + h(JX, JX), JZ \big) \not= 0,
\]
then the complex distribution $\mathcal{D}$ is not integrable.
\end{lemm}

One example of $CR$ submanifolds is a real hypersurface $H$ of a K\"{a}hler manifold as mentioned above.
If we denote the normal bundle of $H$ by $T^{\perp}H$, then $J(T^{\perp}H)$ becomes a totally real distribution of $N$.
Consider a Hopf real hypersurface of a K\"{a}hler manifold.
One of the most typical examples of a Hopf real hypersurface is the round sphere $S^{2n-1}$ in $\mathbb{C}^{n}$.
In this case, each leaf of the totally real foliation is a great circle defined by the Hopf fibration $S^{2n-1} \rightarrow \mathbb{C}P^{n-1}$.
Thus, the leaves are totally geodesic and $S^{2n-1} \subset \mathbb{C}^{n}$ is a Hopf real hypersurface.
Let $N$ be a $CR$ submanifold that is not necessarily a real hypersurface.
We say that $N$ is a {\it Hopf $CR$ submanifold} if the leaves of the totally real foliation are totally geodesic submanifolds of $N$.
Totally real submanifolds and complex submanifolds are special cases of Hopf $CR$ submanifolds.
Moreover, we say that $N$ is a {\it ruled $CR$ submanifold} if the leaves of the totally real foliation are totally geodesic submanifolds of $M$.
Any ruled $CR$ submanifold is a Hopf $CR$ submanifold.
For example, a hyperplane of $\mathbb{C}^{n}$ is a ruled $CR$ submanifold.
By definition, totally geodesic totally real submanifolds and complex submanifolds are special cases of ruled $CR$ submanifolds.

In the present paper, we consider a $CR$ submanifold that is also a contact manifold.
Let $N$ be a $CR$ submanifold with $\dim \mathcal{D}^{\perp} = 1$.
If there exists a 1-form $\eta$ on $N$ such that $\eta|_{\mathcal{D}} \equiv 0$ and $\eta \wedge (d \eta)^{n} \not= 0$ at every point $x \in N$, then $N$ is called a {\it contact $CR$ submanifold}.
Next, we consider Sasakian manifolds.

\begin{defi}
Let $(S,h)$ be a $2n+1$-dimensional compact Riemannian manifold.
The Levi-Civita connection and the curvature tensor are denoted by $\nabla$ and $R$, respectively. 
If $(S,h)$ satisfies one of the following conditions, we call $(S,h)$ a {\it Sasakian manifold}.

\begin{itemize}

\item[(i)]
There exists a Killing vector field $\xi$ of unit length on $S$ such that the $(1,1)$-tensor $\Phi_{x}(A) := \nabla_{X}\xi \ (x \in S, A \in T_{x}S)$ satisfies 
\[
(\nabla_{X}\Phi)(Y) = g(\xi, Y)X - g(X,Y)\xi
\]
for any vector field $X$ and $Y$ on $S$.

\item[(ii)]
There exists a Killing vector field $\xi$ of unit length on $S$ such that 
\[
R(X,\xi)Y = g(\xi, Y)X - g(X,Y)\xi
\]
for any vector field $X$ and $Y$ on $S$.

\item[(iii)]
The metric cone $(M \times \mathbb{R}_{+}, r^{2}g + dr^{2})$ is a K\"{a}hler manifold.

\end{itemize}
\end{defi}

It is known that any Sasakian manifold is a contact manifold with a contact form $\eta (\cdot ) = h(\xi, \cdot)$.
If a contact $CR$ submanifold $N$ is a Sasakian manifold with respect to the induced metric on $M$, then we say that $N$ is a Sasaki $CR$ submanifold.
For example, the unit sphere $S^{2n-1}$ in $\mathbb{C}^{n}$ with the standard inner product is a Sasaki $CR$ submanifold.
We recall fundamental properties of Sasakian manifolds from \cite{Boyer-Galicki}.

For a Sasakian manifold $(S,h)$, the Killing vector field $\xi$ is called the {\it Reeb vector field}.
The $(1,1)$-tensor $\Phi$ is called the {\it characteristic tensor} and the $1$-form $\eta(\cdot) = h(\xi, \cdot)$ is called the {\it characteristic 1-form}.
For $\xi, \eta, \Phi$, 
\[
\begin{array}{llll}
\Phi \circ \Phi(X) = -X + \eta (X)\xi, 
& \Phi(\xi) = 0, \quad \eta(\Phi X) = 0, 
\\ 
g(\Phi(X), Y) + g(X,. \Phi(Y)) = 0, 
& g(\Phi(X), \Phi(Y)) = g(X,Y) - \eta(Y)\eta(X), 
\\ 
d\eta(X,Y) = 2g(\Phi(X), Y).
& 
\end{array}
\]
for any $x \in S$ and $X,Y \in T_{x}S$.
Since $\xi$ is a Killing vector field of unit length, the foliation of $S$ is defined by $\xi$.
This foliation is called the {\it Reeb foliation} of $S$.
Each leaf of the Reeb foliation is a geodesic of $S$.
Therefore, a Sasaki $CR$ submanifold is a Hopf $CR$ submanifold.
If each leaf of the Reeb foliation is a closed geodesic and all leaves have the same length, we say that the Sasakian manifold is regular.
In the present paper, we assume that a Sasakian manifold is regular.
Let $T$ be the space of leaves of the Reeb foliation.
Then, $T$ is a manifold by the assumption.
Denote the natural projection by $\pi : S \rightarrow T$.
A $2$-form $\omega^{T}$ is defined by $\pi^{*}\omega^{T} = (1/2)d \eta$ and $(T, \omega^{T})$ is a symplectic manifold.
Moreover, a $(1,1)$-tensor $J^{T}$ on $T$ is defined by $\pi^{*}J^{T} = \Phi$.
Since $\Phi^{2} = - \mathrm{id} + \eta \otimes \xi$, the $(1,1)$-tensor $J^{T}$ is an almost complex structure on $T$.
Since $\xi$ is a Killing vector field on $S$, a Riemannian metric $h^{T}$ on $T$ can be defined such that $\pi:(S,h) \rightarrow (T,h^{T})$ is a Riemannian submersion.
Then, $(T, J^{T}, h^{T})$ is a compact K\"{a}hler manifold, and $\omega^{T}$ is the K\"{a}hler form.
This K\"{a}hler structure is called a {\it transverse K\"{a}hler structure} of $S$.
Let $r$ (resp.\ $r^{T}$) be the Ricci tensor of $(S, h)$ (resp.\ $(L, h^{T})$).

\begin{defi}
A Sasakian manifold $(S,h)$ is said to be $\eta$-Einstein if there exist constants $\lambda$ and $\mu$ such that
\[
r = \lambda g + \mu \ \eta \otimes \eta.
\]
\end{defi}

In this case, $\lambda + \mu = 2n$ since $r(\xi, \xi) = 2n$.
In particular, if $(S,h)$ is an Einstein manifold, then $\lambda = 2n$ and $\mu = 0$.
Thus, the Einstein constant of $(S,h)$ is $2n$.

\begin{defi}
A Sasakian manifold $(S,h)$ is said to be transversely K\"{a}hler Einstein if $(T,h^{T})$ is a K\"{a}hler Einstein manifold, that is, there exists a constant $\tau$ such that $r^{T} = \tau h^{T}$.
\end{defi}

It is well known that a Sasakian manifold $(S,h)$ is $\eta$-Einstein if and only if $(S,h)$ is transversely K\"{a}hler Einstein.
In fact, if $r^{T} = \tau g^{T}$, then
\[
r = (\tau - 2) g + (2n+2 - \tau)\eta \otimes \eta.
\]
Conversely, if $r = \lambda g + \nu \eta \otimes \eta$, then
\[
r^{T} = (\lambda + 2)g^{T}.
\]
Thus, $(S,h)$ is a Sasaki-Einstein manifold if and only if $(S,h)$ is transversely K\"{a}hler Einstein and $\tau = 2n+2$.

%%%%%%%%%%%%%%%%%%%%%%%%%%%%%%%%%%%%%%%%%%%%%%%%%%%%%%%%%%%%%%%%%%%%%%%%%%%%%%%%%%%%%%%%%%%%%%%%%%%%%%%%%%%%%%%%%%%%%%%%%%%%%%%%%%%%%%%%%%%%%%%%%%%%%%%%

\subsection{Compact Hermitian symmetric spaces} \label{s2-2}

In this subsection, we recall the construction and properties of irreducible Hermitian symmetric spaces of compact type.
Let $\frak{g}$ be a compact simple Lie algebra and $G$ be the identity component of the automorphism group of $\frak{g}$.
We denote the complexification of $\frak{g}$ by $\tilde{\frak{g}}$.
Let $(\ ,\ )$ be an invariant non-degenerate symmetric bilinear form on $\tilde{\frak{g}}$.
Define $\langle \ , \ \rangle = -(\ ,\ )|_{\frak{g} \times \frak{g}}$.
Then, $\langle \ ,\ \rangle$ is an invariant inner product on $\frak{g}$.
We define $J \in \frak{g}\ (J \not= 0)$ such that $(\mathrm{ad}J)^{3} = -\mathrm{ad}J$.
The $\mathrm{Ad}(G)$-orbit through $J$ is denoted by $M$.
Then, $g(J)\ (g \in G)$ implies $\mathrm{Ad}(g)(J)$.
The induced metric on $M$ by $\langle \ ,\ \rangle$ is denoted by the same symbol $\langle \ , \ \rangle$.
Define $K = \{ g \in G\ ;\ \mathrm{Ad}(g)J = J \}$.
Then, $M = G/K$ and the Lie algebra of $K$ is $\frak{k} := \{ X \in \frak{g}\ ;\ [J, X] = 0 \}$.
In particular, $J \in \frak{k}$ and the center of $\frak{k}$ is $\mathbb{R}J$.
It is known that $K$ is connected \cite{Besse}.
Set $\theta = e^{\pi \mathrm{ad}J}$.
The inner automorphism by $\theta$ is an involutive automorphism of $G$ since $(\mathrm{ad}J)^{3} = -\mathrm{ad}J$.
This involutive automorphism of $G$ is denoted by the same symbol $\theta$.
Moreover, we denote by the same symbol the induced automorphism of $\frak{g}$ by $\theta$.
Then, $K$ is the identity component of $F(\theta, G) := \{ g \in G \ ;\ \theta(g) = g \}$, and $\frak{k} = \{ X \in \frak{g} \ ;\ \theta(X) = X \}$.
Set $\frak{m} = \{ X \in \frak{g} \ ;\ \theta(X) = -X \}$.
When consider $J$ as a point of $M$, we denote $J$ by $o$.
Then, $T_{o}M = \frak{m}$.
The $K$-invariant inner product on $\frak{m}$ is induced by $\langle \ ,\ \rangle|_{\frak{m} \times \frak{m}}$.
This inner product on $\frak{m}$ is denoted by the same symbol $\langle \ ,\ \rangle$.
Obviously, $\mathrm{ad}J(\frak{m}) \subset \frak{m}$ and $(\mathrm{ad}J|_{\frak{m}})^{2} = -\mathrm{id}_{\frak{m}}$.
Since $J$ is an element of the center of $\frak{k}$, we can define an almost complex structure on $M$ by $J$.
This almost complex structure is denoted by the same symbol $J$.
It is known that the almost complex structure $J$ is integrable, so $J$ is a complex structure on $M$, and $M$ is a complex manifold.
Moreover, $(M, J, \langle \ ,\ \rangle)$ is an irreducible Hermitian symmetric space of compact type.
Conversely, any irreducible Hermitian symmetric space of compact type can be constructed by this method.
Irreducible Hermitian symmetric spaces of compact type are classified as shown in Table 1.
The following diffeomorphism is known:
\[
\begin{array}{llllllll}
\tilde{G}_{2}(\mathbb{R}^{3}) \cong \mathbb{C}P^{1} \cong SO(4)/U(2) \cong Sp(1)/U(1) \\ 
\tilde{G}_{2}(\mathbb{R}^{4}) \cong \mathbb{C}P^{1} \times \mathbb{C}P^{1}, \quad \tilde{G}_{2}(\mathbb{R}^{5}) \cong Sp(2)/U(2), \\
\tilde{G}_{2}(\mathbb{R}^{6}) \cong G_{2}(\mathbb{C}^{4}), \quad  \mathbb{C}P^{2} \cong SO(6)/U(3).
\end{array}
\]
In the present paper, we consider irreducible Hermitian symmetric spaces of compact type excluding the $2$-dimensional sphere.
Thus, we consider the following six classes
\[
\begin{array}{llll}  
& G_{k}(\mathbb{C}^{n}) \ (n \geq 3, k < n), & \tilde{G}_{2}(\mathbb{R}^{n})\ (n \geq 5), & SO(2n)/U(n) \ (n \geq 3), \\
& Sp(n)/U(n) \ (n \geq 2), & EIII, & EVII.
\end{array}
\]
We note that the restricted root system of these Hermitian symmetric spaces is either type $BC_{n} \ (n \geq 1)$ or $C_{n} \ (n \geq 2)$ \cite{Helgason}.

\begin{table}[h]
\caption{Irreducible Hermitian symmetric spaces of compact type}
\begin{tabular}{c|ccccccccccccccccc} \hline
$M$ & $G_{k}(\mathbb{C}^{n})$ & $\tilde{G}_{2}(\mathbb{R}^{n}) \ (n \not= 4)$ & $SO(2n)/U(n)$ & $Sp(n)/U(n)$ & $EIII$ & $EVII$ \\ \hline
$\dim M$ & $2(n-k)$ & $2(n-2)$ & $n(n-1)$ & $n(n+1)$ & 32 & 54 \\ \hline
$\mathrm{rank} M$ & $\mathrm{min}(k,n-k)$ & 2 & $[(1/2)n]$ & $n$ & 2 & 3 \\ \hline
restricted & $BC\ (n \not= 2k)$ & $C$ & $C$\ ($n$ is even) & $C$ & $BC$ & $C$ \\
root system  & $C \ (n = 2k)$ & & $BC$\ ($n$ is odd) & & & & \\ \hline
\end{tabular}
\end{table}

We recall the fundamental properties of the Levi-Civita connection of symmetric spaces.
For any $X \in \frak{g}$, the fundamental vector field corresponding to $X$ is denoted by $X^{*}$, that is, for any $p \in M$,
\[
X^{*}_{p} = \left. \frac{d}{dt}\mathrm{exp}tX(p) \right|_{t=0}.
\]
Let $V$ be any vector space and $W$ be a subspace of $V$.
Then, $(v)_{W}$ denotes the $W$-part of $v$ for any $v \in V$.

\begin{lemm} \cite{Helgason} \label{2-2-1}
Let $\nabla^{M}$ be the Levi-Civita connection of $M$.

\begin{itemize}

\item[(i)]
For any $g \in G$ and $X,Y \in \frak{g}$, 
\[
g_{*} \Big( \nabla^{M}_{X^{*}}Y^{*} \Big) = \nabla^{M}_{(\mathrm{Ad}(g)X)^{*}} \big( \mathrm{Ad}(g)Y \big)^{*}.
\]

\item[(ii)]
For any $X,Y \in \frak{g}$,
\[
(\nabla^{M}_{X^{*}}Y^{*})_{o} =
\left\{
\begin{array}{ll}
-[X,Y]_{\frak{m}} & (X \in \frak{m}), \\ 
0 & (X \in \frak{k}).
\end{array}
\right.
\]

\item[(iii)]
For any $p = g(o)\ (g \in G)$ and $X,Y \in \frak{g}$,
\[
(\nabla^{M}_{X^{*}}Y^{*})_{p} = -g_{*} \Big[ \big( \mathrm{Ad}(g^{-1})X \big)_{\frak{m}}, \mathrm{Ad}(g^{-1})Y \Big]_{\frak{m}}.
\]

\end{itemize}
\end{lemm}

We will use the symbol $R^{M}$ to denote the curvature tensor of $M$.

\begin{lemm} \cite{Helgason} \label{2-2-2}
For any $X,Y,Z \in \frak{m}$, 
\[
R^{M}(X,Y)Z = -\big[ [X,Y], Z \big].
\]
\end{lemm}

We will close this subsection by recalling the Gauss equation.

\begin{lemm}[The Gauss equation]
Let $(M,g)$ be a Riemannian manifold and $N$ be a submanifold of $M$.
Consider the induced metric on $N$ by $g$.
Denote by $R^{M}$ (resp.\ $R^{N}$) the curvature tensor of $M$ (resp.\ $N$).
Let $h$ be the second fundamental form of $N$.
Then, for any $x \in N$ and $X,Y,Z,W \in T_{x}M$,
\[
\begin{split}
g \Big( R^{N}(X,Y)Z,W \Big) = g \Big( & R^{M}(X,Y)Z,W \Big) \\ 
&- g \Big( h(X,Z), h(Y,W) \Big) + g \Big( h(X, W), h(Y,Z) \Big).
\end{split}
\]
\end{lemm}

%%%%%%%%%%%%%%%%%%%%%%%%%%%%%%%%%%%%%%%%%%%%%%%%%%%%%%%%%%%%%%%%%%%%%%%%%%%%%%%%%%%%%%%%%%%%%%%%%%%%%%%%%%%%%%%%%%%%%%%%%%%%%%%%%%%%%%%%%%%%%%%%%%%%%%%%

\subsection{The restricted root system} \label{s2-3}

In this subsection, we recall fundamental properties of the restricted root system of $M$ from \cite{Helgason}, \cite{Takeuchi}, and \cite{Takeuchi2}.
Let $\tau$ be the complex conjugation corresponding to $\frak{g}$.
We extend $\theta$ to $\tilde{\frak{g}}$ and denote it by the same symbol.
Set $\sigma = \theta \circ \tau = \tau \circ \theta$.
The real form corresponding to $\sigma$ is denoted by $\frak{g}_{\sigma}$, that is, $\frak{g}_{\sigma} = \{ X \in \tilde{\frak{g}} \ ;\ \sigma(X) = X \}$.
Then, $\frak{g}_{\sigma}$ is a real simple Lie algebra, and $\frak{g}_{\sigma} = \frak{k} + i\frak{m}$.
Let $\frak{a}$ be a maximal abelian subspace of $\frak{m}$ and $\frak{h}$ be a maximal abelian subspace of $\frak{g}$ containing $\frak{a}$ such that $\theta(\frak{h}) \subset \frak{h}$.
Set $\frak{b} = \frak{h} \cap \frak{k}$.
Then, $\frak{h} = \frak{b} + \frak{a}$.
Let $\tilde{\frak{h}}$ and $\tilde{\frak{a}}$ be the complexification of $\frak{h}$ and $\frak{a}$, respectively.
Then, $\tilde{\frak{h}}$ is a Cartan subalgebra of $\tilde{\frak{g}}$.
Define $\frak{h}_{0} = i\frak{h}$ and $\frak{a}_{0} = i\frak{a}$.
For any vector space $V$, we denote the dual of $V$ by $V^{*}$.
For each $\alpha \in \tilde{\frak{h}}^{*}$, set
\[
\tilde{\frak{g}}_{\alpha} = \{ X \in \tilde{\frak{g}}\ ;\ [H, X] = \alpha(H)X\ (H \in \tilde{\frak{h}}) \}.
\]
Thus, we define the root system $\tilde{\Sigma}$ of $\tilde{\frak{g}}$ with respect to $\tilde{\frak{h}}$ by $\tilde{\Sigma} = \{ \alpha \in \frak{h}_{0}^{*} - \{ 0 \} \ ;\ \dim \tilde{\frak{g}_{\alpha}} \not= 0 \}$.
Then, the following root space decomposition follows:
\[
\tilde{\frak{g}} = \tilde{\frak{h}} + \sum_{\alpha \in \tilde{\Sigma}}\tilde{\frak{g}}_{\alpha}.
\]
Note that each $\alpha \in \tilde{\Sigma}$ is $\mathbb{R}$-valued on $\frak{h}_{0}$.
Because the restriction $(\ ,\ )|_{\frak{h}_{0} \times \frak{h}_{0}}$ of $(\ ,\ )$ to $\frak{h}_{0} \times \frak{h}_{0}$ is an inner product on $\frak{h}_{0}$, we set $H_{\alpha} \in \frak{h}_{0}$ for any $\alpha \in \tilde{\Sigma}$ such that $(H_{\alpha}, H) = \alpha(H)$ for any $H \in \frak{h}_{0}$.
We denote $(H_{\alpha}, H_{\beta})$ by $(\alpha, \beta)$ for any $\alpha, \beta \in \tilde{\Sigma}$, 
The coroot of $\alpha$ is denoted by $\alpha^{*}$ and the element of $\frak{h}_{0}$ corresponding to $\alpha^{*}$ is denoted by $A_{\alpha}$, that is, $\alpha^{*} = (2/(\alpha, \alpha))\alpha$ and $A_{\alpha} = (2/(\alpha, \alpha))H_{\alpha}$.
For any $\alpha \in \tilde{\Sigma}, H \in \tilde{\frak{h}},$ and $X \in \tilde{\frak{g}}_{\alpha}$,
\[
\begin{split}
[H, \tau(X)] = \tau\Big( [\tau(H), X] \Big) = - \tau \big( \alpha \big( \tau(H) \big)X \big) = -\overline{\alpha \big(\tau(H) \big)}\tau(X).
\end{split}
\]
Thus, $\tau(\alpha) = \overline{\alpha \circ \tau (\cdot)} \in \tilde{\Sigma}$.
Similarly, $\sigma(\alpha) = \overline{\alpha \circ \sigma (\cdot)} \in \tilde{\Sigma}$.
Then, $H_{\tau(\alpha)} = \tau(H_{\alpha}) =  -H_{\alpha}$ and $H_{\sigma(\alpha)} = \sigma(H_{\alpha})$.
Note that $\tau|_{\frak{h}_{0}} = -\mathrm{id}$ and $\sigma|_{\frak{h}_{0}} = -\mathrm{id}_{i\frak{b}} + \mathrm{id}_{\frak{a}_{0}}$.

For any $\lambda \in \tilde{\frak{a}}^{*}$, set
\[
\frak{g}_{\lambda}^{\mathbb{C}} = \{ X \in \tilde{\frak{g}}\ ;\ [H, X] = \alpha(H)X \ (H \in \frak{a})\}
\]
and define the restricted root system of $(\frak{g}, \frak{k})$ with respect to $\frak{a}$ by $\Sigma = \{ \lambda \in \frak{a}_{0}^{*} - \{ 0 \}\ ;\ \dim \frak{g}_{\lambda}^{\mathbb{C}} \not= 0 \}$.
Note that any restricted root is $\mathbb{R}$-valued on $\frak{a}_{0}$. 
We denote by $\overline{(\, \cdot \, )}$ the orthogonal projection from $\frak{h}_{0}$ onto $\frak{a}_{0}$ as well as the restriction of $\tilde{\frak{h}}$ to $\tilde{\frak{a}}$.
Then, 
\[
\Sigma = \{ \bar{\alpha} \ ;\ \alpha \in \tilde{\Sigma} \}.
\]
For each $\lambda \in \Sigma$, we define $H_{\lambda} \in \frak{a}_{0}$ such that $(H_{\lambda}, H) = \lambda(H)$ for any $H \in \frak{a}_{0}$.
If $\alpha \in \tilde{\Sigma}$ satisfies $\bar{\alpha} = \lambda$, then $H_{\lambda} = \overline{H}_{\alpha}$.
Moreover, $\lambda = (1/2)(\alpha + \sigma(\alpha))$ and $H_{\lambda} = (1/2)(H_{\alpha} + \sigma(H_{\alpha}))$.
Denote $(H_{\lambda}, H_{\mu})$ by $(\lambda, \mu)$ for any $\lambda, \mu \in \Sigma$ and set $A_{\lambda} = (2/(\lambda, \lambda))H_{\lambda}$.
For any $\lambda \in \Sigma$,
\[
\frak{g}^{\mathbb{C}}_{\lambda} = \sum_{\alpha \in \tilde{\Sigma}, \bar{\alpha} = \lambda}\tilde{\frak{g}}_{\alpha}
\]
Let $m(\lambda)$ be the multiplicity of $\lambda \in \Sigma$, that is, $m(\lambda) = \#\{ \alpha \in \tilde{\Sigma} \ ;\ \bar{\alpha} = \lambda \} = \dim_{\mathbb{C}}\frak{g}_{\lambda}^{\mathbb{C}}$.
If the centralizer of $\tilde{\frak{a}}$ in $\tilde{\frak{g}}$ is denoted by $\frak{g}_{0}^{\mathbb{C}}$, then
\[
\tilde{\frak{g}} = \frak{g}_{0}^{\mathbb{C}} + \sum_{\lambda \in \Sigma}\frak{g}_{\lambda}^{\mathbb{C}}.
\]
Take a linear order on $\frak{h}_{0}$ such that $\overline{H} > 0$ if $H > 0$ for any $H \in \frak{h}_{0}$.
If $\alpha \in \tilde{\Sigma}$ satisfies $\alpha > 0$, then $\sigma(\alpha) > 0$.
The set of all positive roots (resp.\ restricted roots) is denoted by $\tilde{\Sigma}^{+}$ (resp.\ $\Sigma^{+}$).
Moreover, the fundamental roots (resp.\ restricted roots) is denoted by $\tilde{F}$ (resp.\ $F$).
Then,
\[
\Sigma^{+} = \{ \bar{\alpha} \ ;\ \alpha \in \tilde{\Sigma}^{+} \}, \quad F = \{ \bar{\alpha} \ ;\ \alpha \in \tilde{F} \}.
\]

\begin{lemm} \cite{Takeuchi} \label{2-3-1}
For each $\alpha \in \tilde{\Sigma}$, there exists $X_{\alpha} \in \tilde{\frak{g}_{\alpha}}$ satisfying the following properties.

\begin{itemize}

\item[(i)]
$[X_{\alpha}, X_{-\alpha}] = -A_{\alpha}$.

\item[(ii)]
$\tau(X_{\alpha}) = X_{-\alpha}$.

\item[(iii)]
Set $U_{\alpha} = X_{\alpha} + X_{-\alpha}$ and $W_{\alpha} = i(X_{\alpha} - X_{-\alpha})$.
Then, $U_{-\alpha} = U_{\alpha}$ and $W_{-\alpha} = -W_{\alpha}$.
Moreover,
\[
\frak{g} = \frak{h} + \sum_{\alpha \in \tilde{\Sigma}^{+}}(\mathbb{R}U_{\alpha} + \mathbb{R}W_{\alpha}).
\]

\item[(iv)]
$\sigma(X_{\alpha}) = X_{\sigma(\alpha)}$.

\item[(v)]
If $\alpha, \beta \in \tilde{\Sigma}$ satisfy $\alpha + \beta \in \tilde{\Sigma}$, we set $N_{\alpha, \beta} \in \mathbb{C}$ such that $[X_{\alpha}, X_{\beta}] = N_{\alpha, \beta}X_{\alpha + \beta}$.
Then, $N_{-\alpha, -\beta} = \epsilon_{\alpha, \beta}N_{\alpha, \beta}$, where $\epsilon_{\alpha, \beta} = \pm 1$ and $N_{\alpha, \beta}^{2} = \epsilon_{\alpha, \beta}(p+1)^{2}$, where $p$ is the largest integer such that $\alpha - p\beta \in \tilde{\Sigma}$.

\end{itemize}
\end{lemm}

\begin{lemm} \label{2-3-2}
For any $\alpha, \beta, \gamma \in \tilde{\Sigma}$,

\begin{itemize}

\item[(i)]
$N_{\beta, \alpha} = - N_{\alpha, \beta}$ and $\epsilon_{\beta, \alpha} = \epsilon_{\alpha, \beta}$.

\item[(ii)]
$N_{-\alpha, -\beta} = \overline{N_{\alpha, \beta}}$ and $N_{\sigma(\alpha), \sigma(\beta)} =\overline{ N_{\alpha, \beta} }$.
Moreover, $\epsilon_{-\alpha, -\beta} = \epsilon_{\alpha, \beta}$ and $\epsilon_{\sigma(\alpha), \sigma(\beta)} = \epsilon_{\alpha, \beta}$.

\item[(iii)]
$N_{\alpha, \beta} = (1/k)N_{\beta, \gamma} = N_{\gamma, \alpha}$ and $\epsilon_{\alpha, \beta} = \epsilon_{\beta, \gamma} = \epsilon_{\gamma, \alpha}$ if $\alpha + \beta + \gamma = 0$ and $|\alpha| = \sqrt{k}|\beta| = \sqrt{k}|\gamma|\ (k \in \mathbb{Z})$.

\end{itemize}
\end{lemm}

\begin{proof}
Since $[X_{\alpha}, X_{\beta}] = - [ X_{\beta}, X_{\alpha} ] \ (X,Y \in \tilde{\frak{g}})$, we see (i) follows.
Moreover, $\tau([X_{\alpha}. X_{\beta}]) = [\tau(X), \tau(Y)]$ and $\sigma([X_{\alpha}. X_{\beta}]) = [\sigma(X), \sigma(Y)]$.
Thus, (ii) follows.
The Jacobi identity implies that (iii) follows.
\end{proof}

By simple computations, we have Lemma \ref{2-3-3}.

\begin{lemm} \label{2-3-3}
For $\alpha, \beta \in \tilde{\Sigma}^{+}$, set $U_{\alpha \pm \beta} = W_{\alpha \pm \beta} = 0$ if $\alpha \pm \beta \not\in \tilde{\Sigma}$.
If $(\epsilon_{\alpha, \beta}, \epsilon_{\alpha, -\beta}) = (1,1)$,
\[
\begin{array}{lllllllllllllllllllll}
\ [U_{\alpha}, U_{\beta}]  = N_{\alpha, \beta}U_{\alpha + \beta} + N_{\alpha, -\beta}U_{\alpha - \beta}, 
& [W_{\alpha}, W_{\beta}] = - N_{\alpha, \beta}U_{\alpha + \beta} + N_{\alpha, -\beta}U_{\alpha - \beta}, \\
\ [U_{\alpha}, W_{\beta}] = N_{\alpha, \beta}W_{\alpha + \beta} - N_{\alpha, -\beta}W_{\alpha - \beta}, 
& [W_{\alpha}, U_{\beta}] =N_{\alpha, \beta}W_{\alpha + \beta} + N_{\alpha, -\beta}W_{\alpha - \beta}. \\
\end{array}
\]
If $(\epsilon_{\alpha, \beta}, \epsilon_{\alpha, -\beta}) = (1,-1)$,
\[
\begin{array}{lllllllllllllllllll}
\ [U_{\alpha}, U_{\beta}] = N_{\alpha, \beta}U_{\alpha + \beta} - iN_{\alpha, -\beta}W_{\alpha - \beta}, 
& [W_{\alpha}, W_{\beta}] = -N_{\alpha, \beta}U_{\alpha + \beta} - iN_{\alpha, -\beta}W_{\alpha - \beta}, \\
\ [U_{\alpha}, W_{\beta}] = N_{\alpha, \beta}W_{\alpha + \beta} - iN_{\alpha, -\beta}U_{\alpha - \beta}, 
& [W_{\alpha}, U_{\beta}] = N_{\alpha, \beta}W_{\alpha + \beta} + iN_{\alpha, -\beta}U_{\alpha - \beta}. 
\end{array}
\]
If $(\epsilon_{\alpha, \beta}, \epsilon_{\alpha, -\beta}) = (-1,1)$,
\[
\begin{array}{lllllllllllllllllll}
\ [U_{\alpha}, U_{\beta}] = - iN_{\alpha, \beta}W_{\alpha + \beta} + N_{\alpha, -\beta}U_{\alpha - \beta}, 
& [W_{\alpha}, W_{\beta}] = iN_{\alpha, \beta}W_{\alpha + \beta} + N_{\alpha, -\beta}U_{\alpha - \beta}, \\
\ [U_{\alpha}, W_{\beta}] = iN_{\alpha, \beta}U_{\alpha + \beta} - N_{\alpha, -\beta}W_{\alpha - \beta}, 
& [W_{\alpha}, U_{\beta}] = iN_{\alpha, \beta}U_{\alpha + \beta} + N_{\alpha, -\beta}W_{\alpha - \beta}.
\end{array}
\]
If $(\epsilon_{\alpha, \beta}, \epsilon_{\alpha, -\beta}) = (-1,-1)$,
\[
\begin{array}{llllllllllllllll}
\ [U_{\alpha}, U_{\beta}] = -iN_{\alpha, \beta}W_{\alpha + \beta} - iN_{\alpha, -\beta}W_{\alpha - \beta}, 
& [W_{\alpha}, W_{\beta}] = iN_{\alpha, \beta}W_{\alpha + \beta} - iN_{\alpha, -\beta}W_{\alpha - \beta}, \\
\ [U_{\alpha}, W_{\beta}] = iN_{\alpha, \beta}U_{\alpha + \beta} - iN_{\alpha, -\beta}U_{\alpha - \beta}, 
& [W_{\alpha}, U_{\beta}] = iN_{\alpha, \beta}U_{\alpha + \beta} + iN_{\alpha, -\beta}U_{\alpha - \beta}. 
\end{array}
\]
\end{lemm}

Let $\lambda \in \Sigma$.
If $\alpha \in \tilde{\Sigma}$ and $\bar{\alpha} = \lambda$, then $\overline{\sigma(\alpha)} = \lambda$.
Thus, $\frak{g}_{\lambda}^{\mathbb{C}} + \frak{g}_{-\lambda}^{\mathbb{C}}$ is invariant under both $\tau$ and $\sigma$.
Define 
\[
\frak{k}_{\lambda} = (\frak{g}_{\lambda}^{\mathbb{C}} + \frak{g}_{-\lambda}^{\mathbb{C}}) \cap \frak{k}, \quad \frak{m}_{\lambda} = (\frak{g}_{\lambda}^{\mathbb{C}} + \frak{g}_{-\lambda}^{\mathbb{C}}) \cap \frak{m}.
\]
By definition, $\frak{k}_{-\lambda} = \frak{k}_{\lambda}$ and $\frak{m}_{-\lambda} = \frak{m}_{\lambda}$.
Let $\frak{k}_{0} = \frak{k} \cap \frak{g}_{0}^{\mathbb{C}}$.
Then, $\frak{k}_{0}$ is the centralizer of $\frak{a}$ in $\frak{k}$ and we have
\[
\frak{k} = \frak{k}_{0} + \sum_{\lambda \in \Sigma^{+}}\frak{k}_{\lambda}, \quad\quad \frak{m} = \frak{a} + \sum_{\lambda \in \Sigma^{+}}\frak{m}_{\lambda}.
\]

\begin{lemm} \cite{Takeuchi2} \label{2-3-4}
For $\lambda, \mu \in \Sigma^{+}\ (\lambda \not= \mu)$, set $\frak{k}_{\lambda \pm \mu} = \frak{m}_{\lambda \pm \mu} = \{ 0 \}$ if $\lambda \pm \mu \not\in \Sigma$.
Then, 
\[
\begin{array}{llllllll}
& [\frak{k}_{\lambda}, \frak{k}_{\mu}] \subset \frak{k}_{\lambda + \mu} + \frak{k}_{\lambda - \mu}, 
& [\frak{m}_{\lambda}, \frak{m}_{\mu}] \subset \frak{k}_{\lambda + \mu} + \frak{k}_{\lambda - \mu}, 
& [\frak{k}_{\lambda}, \frak{m}_{\mu}] \subset \frak{m}_{\lambda + \mu} + \frak{m}_{\lambda - \mu}, \\
& [\frak{k}_{\lambda}, \frak{k}_{\lambda}] \subset \frak{k}_{0} + \frak{k}_{2\lambda}, 
& [\frak{m}_{\lambda}, \frak{m}_{\lambda}] \subset \frak{k}_{0} + \frak{k}_{2\lambda}, 
& [\frak{k}_{\lambda}, \frak{m}_{\lambda}] \subset \frak{a} + \frak{m}_{2\lambda}, \\
& [\frak{k}_{0}, \frak{k}_{\lambda}] \subset \frak{k}_{\lambda},
& [\frak{k}_{0}, \frak{m}_{\lambda}] \subset \frak{m}_{\lambda}. \\
\end{array}
\]
\end{lemm}

\

Set subsets of $\tilde{\Sigma}$ as follows:
\[
\begin{array}{llll} 
\tilde{\Sigma}_{(0)} = \{ \alpha \in \Sigma \ ;\ \sigma( \alpha ) = - \alpha \}, 
& \tilde{\Sigma}_{(1)} = \{ \alpha \in \Sigma \ ;\ \sigma(\alpha) = \alpha \}, \\
\tilde{\Sigma}_{(2)} = \{ \alpha \in \Sigma \ ;\ \sigma(\alpha) \not= \pm \alpha, 2\bar{\alpha} \not\in \Sigma \}, 
& \tilde{\Sigma}_{(3)} = \{ \alpha \in \Sigma \ ;\ \sigma(\alpha) \not= \pm \alpha, 2\bar{\alpha} \in \Sigma \}. \\
\end{array}
\]
For each $0 \leq i \leq 3$, set $\tilde{\Sigma}^{+}_{(i)} = \tilde{\Sigma}^{+} \cap \tilde{\Sigma}_{(i)}$.
For $\alpha \in \tilde{\Sigma}_{(0)}$, set 
\[
\begin{array}{llllllllllllll}
S_{\alpha}^{(1)} = U_{\alpha}, \quad S_{\alpha}^{(2)} = W_{\alpha}.
\end{array}
\]
For $\alpha \in \tilde{\Sigma}_{(1)}$, set 
\[
\begin{array}{llllllllllllll}
S_{\alpha}^{(1)} = U_{\alpha}, \quad S_{\alpha}^{(2)} = 0, \quad T_{\alpha}^{(1)} = W_{\alpha}, \quad T_{\alpha}^{(2)} = 0. \\
\end{array}
\]
For $\alpha \in \tilde{\Sigma}_{(2)}$, set 
\[
\begin{split}
S_{\alpha}^{(1)} = U_{\alpha} + U_{\sigma(\alpha)}, \quad & S_{\alpha}^{(2)} = W_{\alpha} - W_{\sigma(\alpha)}, \\ 
T_{\alpha}^{(1)} = W_{\alpha} + W_{\sigma(\alpha)}, \quad & T_{\alpha}^{(2)} = - U_{\alpha} + U_{\sigma(\alpha)}. 
\end{split}
\]
For $\alpha \in \tilde{\Sigma}_{(3)}$, set 
\[
\begin{array}{llll}
S_{\alpha}^{(1)} = \sqrt{2} (U_{\alpha} + U_{\sigma(\alpha)}), & S_{\alpha}^{(2)} = \sqrt{2}(W_{\alpha} - W_{\sigma(\alpha)}), \\
T_{\alpha}^{(1)} = \sqrt{2}(W_{\alpha} + W_{\sigma(\alpha)}), & T_{\alpha}^{(2)} = \sqrt{2}(- U_{\alpha} + U_{\sigma(\alpha)}). 
\end{array}
\]
If $\alpha \in \cup_{i=1}^{3}\tilde{\Sigma}_{(i)}$, then
\[
\begin{array}{lllllllllllllll}
S_{-\alpha}^{(1)} = S_{\alpha}^{(1)}, & S_{\sigma(\alpha)}^{(1)} = S_{\alpha}^{(1)} & S_{-\sigma(\alpha)}^{(1)} = S_{\alpha}^{(1)}, \\
S_{-\alpha}^{(2)} = -S_{\alpha}^{(2)}, & S_{\sigma(\alpha)}^{(2)} = -S_{\alpha}^{(2)}, & S_{-\sigma(\alpha)}^{(2)} = S_{\alpha}^{(2)}, \\
T_{-\alpha}^{(1)} = -T_{\alpha}^{(1)}, & T_{\sigma(\alpha)}^{(1)} = T_{\alpha}^{(1)}, & T_{-\sigma(\alpha)}^{(1)} = -T_{\alpha}^{(1)}, \\
T_{-\alpha}^{(2)} = T_{\alpha}^{(2)}, & T_{\sigma(\alpha)}^{(2)} = -T_{\alpha}^{(2)}, & T_{-\sigma(\alpha)}^{(2)} = -T_{\alpha}^{(2)}.
\end{array}
\]

\begin{lemm} \cite{Takeuchi2} \label{2-3-5}
The above $S_{\alpha}^{(i)}$ and $T_{\alpha}^{(i)} \ (i=1,2, \alpha \in \tilde{\Sigma})$ satisfy the following properties.

(1)\ 
Let $\alpha \in \cup_{k=1}^{3}\tilde{\Sigma}_{(k)}$ and $H \in \frak{a}_{0}$.
Then, for any $i = 1,2$,
\[
\begin{split}
[iH, S_{\alpha}^{(i)}] = \alpha(H)T_{\alpha}^{(i)}, & \quad
\mathrm{Ad}(\mathrm{exp}(iH))S_{\alpha}^{(i)} = \cos \alpha(H) S^{(i)}_{\alpha} + \sin \alpha(H) T_{\alpha}^{(i)}, \\
[iH, T_{\alpha}^{(i)}] = -\alpha(H)S_{\alpha}^{(i)}, & \quad
 \mathrm{Ad}(\mathrm{exp}(iH))T_{\alpha}^{(i)} = \cos \alpha(H) T_{\alpha}^{(i)} - \sin \alpha(H) T_{\alpha}^{(i)}. \\
\end{split}
\]

(2)\ Let $\alpha \in \cup_{k=1}^{3}\tilde{\Sigma}_{(k)}$.
If $S^{(i)}_{\alpha}$ and $T^{(i)}_{\alpha} \ (i=1,2)$ are nonzero, then $[S_{\alpha}^{(i)}, T_{\alpha}^{(i)}] = 2iA_{\bar{\alpha}}$.

(3)\ 
Let $\lambda \in \Sigma^{+}$.
Define $\tilde{\Sigma}_{\lambda} = \{ \alpha \in \tilde{\Sigma} \ ;\ \bar{\alpha} = \lambda \}$ and $\tilde{\Sigma}^{\sigma}_{\lambda} = \{ \alpha \in \tilde{\Sigma}^{+}_{\lambda}\ ;\ \bar{\alpha} = \lambda, \alpha \geq \sigma(\alpha) \}$.
Then,
\[
\begin{split}
& \left\{ S^{(i)}_{\alpha} \ ;\ \alpha \in \tilde{\Sigma}^{\sigma}_{\lambda}, i=1, 2, S^{(i)}_{\alpha} \not= 0 \right\} \\
& \quad\quad\quad\quad = \left\{ S^{(1)}_{\alpha} \ ;\ \alpha \in \tilde{\Sigma}^{\sigma}_{\lambda}, \sigma(\alpha) = \alpha \right\} \cup \left\{ S^{(1)}_{\alpha}, S^{(2)}_{\alpha} \ ;\ \alpha \in \tilde{\Sigma}^{\sigma}_{\lambda}, \sigma(\alpha) > \alpha \right\} 
\end{split}
\]
is an orthogonal basis of $\frak{k}_{\lambda}$, and 
\[
\begin{split}
& \left\{ T^{(i)}_{\alpha} \ ;\ \alpha \in \tilde{\Sigma}^{\sigma}_{\lambda}, i=1, 2, T^{(i)}_{\alpha} \not= 0 \right\} \\
& \quad\quad\quad\quad = \left\{ T^{(1)}_{\alpha} \ ;\ \alpha \in \tilde{\Sigma}^{\sigma}_{\lambda}, \sigma(\alpha) = \alpha \right\} 
\cup \left\{ T^{(1)}_{\alpha}, T^{(2)}_{\alpha} \ ;\ \alpha \in \tilde{\Sigma}^{\sigma}_{\lambda}, \sigma(\alpha) > \alpha \right\}
\end{split}
\]
is an orthogonal basis of $\frak{m}_{\lambda}$.
Moreover,
\[
\frak{k}_{0} = \frak{b} + \sum_{\alpha \in \tilde{\Sigma}^{+}_{(0)}} \Big( \mathbb{R}S_{\alpha}^{(1)} + \mathbb{R}S_{\alpha}^{(2)} \Big).
\]
\end{lemm}

Let $\lambda \in \Sigma^{+}$ and $\alpha \in \tilde{\Sigma}$ with $\bar{\alpha} = \lambda$.
If $S^{(i)}_{\alpha} \not= 0$, then $\langle S_{\alpha}^{(i)}, S_{\alpha}^{(i)} \rangle = \langle T_{\alpha}^{(i)}, T_{\alpha}^{(i)} \rangle = 4/(\lambda, \lambda)$.
Define
\[
\bar{S}^{(i)}_{\alpha} = \frac{\sqrt{(\lambda, \lambda)}}{2}S^{(i)}_{\alpha}, \quad
\bar{T}^{(i)}_{\alpha} = \frac{\sqrt{(\lambda, \lambda)}}{2}T^{(i)}_{\alpha}.
\]
Then, $\{ \bar{S}^{(i)}_{\alpha} \ ;\ \alpha \in \tilde{\Sigma}^{\sigma}_{\lambda}, i=1, 2, S^{(i)}_{\alpha} \not= 0 \}$ forms an orthonormal basis of $\frak{k}_{\lambda}$, and $\{ \bar{T}^{(i)}_{\alpha} \ ;\ \alpha \in \tilde{\Sigma}^{\sigma}_{\lambda}, i=1, 2, T^{(i)}_{\alpha} \not= 0 \}$ forms an orthonormal basis of $\frak{m}_{\lambda}$.
Moreover, if $S^{(i)}_{\alpha}$ and $T^{(i)}_{\alpha}$ are nonzero, then $[\bar{S}_{\alpha}^{(i)}, \bar{T}_{\alpha}^{(i)}] = i\bar{H}_{\alpha}$.

Let $F = \{ \lambda_{1}, \cdots, \lambda_{n} \}$, and $\delta \in \Sigma^{+}$ be the highest restricted root.
Define $\mathcal{F} = F \cup \{ \delta \}$ and
\[
Q = \{ iH \ ;\ H \in \frak{a}_{0}, \ 0 \leq \lambda(H) \leq \pi \ \text{for any $\lambda \in \mathcal{F}$} \}.
\]
Then, $Q$ is a convex polytope in $\frak{a}$.
For each subset $\Delta \subset \mathcal{F}$, we define
\[
Q_{\Delta} = 
\left\{ iH \in Q\ ;\ H \in \frak{a}_{0}, \ \  
\begin{array}{ll}
\lambda(H) > 0 \ \ (\lambda \in \Delta \cap F), & 0 < \delta(H) < \pi \ (\delta \in \Delta), \\
\mu(H) = 0 \ \  (\mu \in F - \Delta), & \delta(H) = \pi \quad\quad\  (\delta \not \in \Delta).
\end{array}
\right\}.
\]
Then, $Q = \bigsqcup_{\Delta \subset \mathcal{F}}Q_{\Delta}$.
It is evident that $Q_{\Delta} \subset \partial Q$ if and only if $\Delta \not= \mathcal{F}$, and $Q_{\Delta}$ is the interior of $Q$ if $\Delta = \mathcal{F}$.
Since $Q_{\{ \lambda_{i} \}}\ (1 \leq i \leq n)$ is a one-point set, we denote this point by $Q_{i}$.
Note that $Q$ is the convex hull of $\{ 0, Q_{1}, \cdots, Q_{n} \}$.
Then, 
\[
Q = \{ t_{1}Q_{1} + \cdots + t_{n}Q_{n}\ ;\ 0 \leq t_{1} + \cdots + t_{n} \leq 1, \ 0 \leq t_{i} \leq 1 \}.
\]
If $1 \leq i_{1} < \cdots < i_{k} \leq n$, then
\[
\begin{split}
Q_{ \{ \lambda_{i_{1}}, \lambda_{i_{2}}, \cdots, \lambda_{i_{k}} \}} 
& = \left\{ 
t_{i_{1}}Q_{i_{1}} + \cdots + t_{i_{k}}Q_{i_{k}}\ ;\ 
\begin{array}{lll}
t_{i_{1}} + \cdots + t_{i_{k}} = 1, \\
0 < t_{i_{m}} < 1\ (1 \leq m \leq k) 
\end{array}
\right\}, \\
Q_{ \{ \lambda_{i_{1}}, \lambda_{i_{2}}, \cdots, \lambda_{i_{k}}, \delta \}} 
& = \left\{ 
t_{i_{1}}Q_{i_{1}} + \cdots + t_{i_{k}}Q_{i_{k}}\ ;\ 
\begin{array}{lll}
0 < t_{i_{1}} + \cdots + t_{i_{k}} < 1, \\
0 < t_{i_{m}} < 1\ (1 \leq m \leq k) 
\end{array}
\right\}. \\
\end{split}
\]
It is evident that $Q_{\{\lambda_{i_{1}}, \cdots, \lambda_{i_{k}}\}}$ is the interior of the convex hull of $\{ Q_{i_{1}}, \cdots, Q_{i_{k}} \}$ and $Q_{\{ \lambda_{i_{1}}, \cdots, \lambda_{i_{k}}, \delta\}}$ is the interior the convex hull of $\{ 0, Q_{i_{1}}, \cdots, Q_{i_{k}} \}$.
Define $\Sigma^{+}_{iH} = \{ \lambda \in \Sigma^{+} \ ;\ \lambda(H) \not\in \mathbb{Z}\pi \}$ for $iH \in Q$.
Note that $\Sigma^{+}_{iH}$ depends only on $\Delta$ and not on $iH \in Q_{\Delta}$.
It is clear that $\Sigma^{+} - \Sigma^{+}_{iH} = \{ \lambda \in \Sigma^{+} \ ;\ \lambda(H) = 0, \pi \}$.
Since any Hermitian symmetric space is known to be simply connected, Lemma \ref{2-3-6} follows.

\begin{lemm}\cite{Helgason} \label{2-3-6}
Any $K$-orbit in $M$ meets $A = \mathrm{exp}(Q)(o)$ at exactly one point.
\end{lemm}

Let $iH \in Q$ and $x = \mathrm{exp}(iH)(o) \in M$.
We often denote $\Sigma^{+}_{iH}$ by $\Sigma^{+}_{x}$.
For each $S_{\alpha}^{(i)}\ (\alpha \in \tilde{\Sigma}, \bar{\alpha} = \lambda  \in \Sigma, i=1,2)$,
\[
\begin{split}
(S_{\alpha}^{(i)})^{*}_{x} 
& =
\left. \frac{d}{dt} \ \mathrm{exp}(tS^{(i)}_{\alpha}) \ \mathrm{exp}(iH)(o) \right|_{t=0} 
=
\left. \frac{d}{dt} \mathrm{exp}(iH) \ \mathrm{exp} \Big( \mathrm{Ad}(\mathrm{exp}(-iH)) S^{(i)}_{\alpha} \Big)(o) \right|_{t=0} \\
&=
(\mathrm{exp}(iH))_{*} \Big( \cos \lambda(H) S^{(i)}_{\alpha} - \sin \lambda(H) T^{(i)}_{\alpha} \Big)_{\frak{m}} 
=
- (\sin \lambda(H)) \big( \mathrm{exp}(iH) \big)_{*} T^{(i)}_{\alpha}.  \\
\end{split}
\]
Thus, we set
\[
\frak{m}_{x} = \sum_{\lambda \in \Sigma_{x}^{+}} \frak{m}_{\lambda}.
\]
Then, the orthogonal complement $\frak{m}^{\perp}_{x}$ of $\frak{m}_{x}$ in $\frak{m}$ is given by 
\[
\frak{m}_{x}^{\perp} = \frak{a} + \sum_{\lambda \in \Sigma^{+} - \Sigma^{+}_{x}} \frak{m}_{\lambda}.
\]

\begin{lemm} \label{2-3-7}
Let $iH \in Q$ and $x = \mathrm{exp}(iH)(o)$.
Then, $T_{x}K(x) = \mathrm{exp}(iH)_{*} \frak{m}_{x}$ and $T_{x}^{\perp}K(x) = \mathrm{exp}(iH)_{*} \frak{m}^{\perp}_{x}$.

\end{lemm}

%%%%%%%%%%%%%%%%%%%%%%%%%%%%%%%%%%%%%%%%%%%%%%%%%%%%%%%%%%%%%%%%%%%%%%%%%%%%%%%%%%%%%%%%%%%%%%%%%%%%%%%%%%%%%%%%%%%%%%%%%%%%%%%%%%%%%%%%%%%%%%%%%%%%%%%

\subsection{The complex structure and the restricted root system} \label{s2-4}

In this subsection, we study some relations between the complex structure and the restricted root system of $M$.
First, we consider the following root system of $\mathbb{R}^{n}$.
Let $e_{i} \ (1 \leq i \leq n)$ be the $1$-form of $\mathbb{R}^{n}$ such that $e_{i}(x) = x_{i} \ (x = (x_{1}, \cdots, x_{n}) \in \mathbb{R}^{n})$ and,
$v_{i} \in \mathbb{R}^{n} \ (1 \leq i \leq n)$ be the dual vector to $e_{i}$.
Then, $v_{1}, \cdots, v_{n}$ forms the standard basis of $\mathbb{R}^{n}$.
Fix a positive number $d>0$ and define the inner product $g$ on $\mathbb{R}^{n}$ such that $g(e_{i}, e_{j}) = d \delta_{ij}$.
Then,
\[
R_{C} = \Big\{ \pm e_{i} \pm e_{j} \ ;\ 1 \leq i < j \leq n \Big\} \cup \Big\{ \pm 2e_{i} \ ;\ 1 \leq i \leq n \Big\}
\]
is a root system of type $C_{n}$, and
\[
R_{BC} = \Big\{ \pm e_{i} \pm e_{j} \ ;\ 1 \leq i < j \leq n \Big\} \cup \Big\{ \pm e_{i}, \pm 2e_{i} \ ;\ 1 \leq i \leq n \Big\}
\]
is a root system of type $BC_{n}$.
If $H_{e_{i}} \in \mathbb{R}^{n}$ satisfies $g(H_{e_{i}}, H) = e_{i}(H)$ for any $H \in \mathbb{R}^{n}$, then $H_{e_{i}} = (1/d)v_{i}$ and $g(H_{e_{i}}, H_{e_{j}}) = (1/d)\delta_{ij}$.

Note that the restricted root system of a irreducible Hermitian symmetric space is of type $BC$ or $C$.
Fix an isomorphism between $(\frak{a}_{0}, (\ ,\ ), \Sigma)$ and $(\mathbb{R}^{n}, g, R_{C})$ or $(\mathbb{R}^{n}, g, R_{BC})$, and denote each element of $\frak{a}_{0}$ and $\Sigma$ by the corresponding element of $\mathbb{R}^{n}$ and $R_{C}$ or $R_{BC}$.
In this notation, $\frak{a} = \{ i(t_{1}v_{1} + \cdots t_{n}v_{n})\ ;\ t_{1}, \cdots, t_{n} \in \mathbb{R} \}$.
Denote $i(\sum_{i=1}^{n}t_{i}v_{i}) \in \frak{a}$ by $(t_{1}, \cdots, t_{n})$.
We can assume that if $\Sigma$ is of type $C$, then
\[
\Sigma^{+} = \Big\{ e_{i} \pm e_{j} \ ;\ 1 \leq i < j \leq n \Big\} \cup \Big\{ 2e_{i} \ ;\ 1 \leq i \leq n \Big\},
\]
and if of type $BC$, then
\[
\Sigma^{+} = \Big\{ e_{i} \pm e_{j} \ ;\ 1 \leq i < j \leq n \Big\} \cup \Big\{ 2e_{i} \ ;\ 1 \leq i \leq n \Big\} \cup \Big\{ e_{i} \ ;\ 1 \leq i \leq n \Big\}.
\]
In both cases, the fundamental system $F$ is given by
\[
\lambda_{1} = e_{1} - e_{2}, \cdots, \lambda_{n-1} = e_{n-1} - e_{n}, \lambda_{n} = e_{n}.
\]
Moreover, the highest root $\delta$ is $2e_{1}$ in both cases.
In this case,
\[
\begin{split}
& Q_{1}  =  ( \pi/2, 0, \cdots, 0, 0 ), \\
& Q_{2} =  ( \pi/2, \pi/2, 0, \cdots, 0), \\
& \quad\quad \vdots \\
& Q_{n-1} =  ( \pi/2, \pi/2, \cdots, \pi/2, 0), \\
& Q_{n} =  ( \pi/2, \pi/2, \cdots, \pi/2, \pi/2). 
\end{split}
\]
For each $1 \leq k \leq n$, we denote $\Sigma_{Q_{k}}^{+}$ by $\Sigma_{k}^{+}$.
Define $(\Sigma^{+} - \Sigma^{+}_{k})_{a} = \{ \lambda \in \Sigma^{+} \ ;\ (-i)\lambda(Q_{k}) = a \}$ for any $a \in \mathbb{R}$.
Since $\Sigma^{+} - \Sigma^{+}_{k} = \{ \lambda \in \Sigma^{+} \ ;\ (-i)\lambda(Q_{k}) = 0, \pi \}$, if $\Sigma$ is of type $BC$, then
\[
\begin{split}
& \Sigma^{+} - \Sigma_{k}^{+} = \Big\{ e_{i} \pm e_{j}\ ;\ 1 \leq i < j \leq k \Big\} \\
& \quad\quad\quad\quad\quad\quad\quad\quad \cup \Big\{ e_{i}, e_{i} \pm e_{j}\ ;\ k+1 \leq i < j \leq n \Big\} \cup \Big\{ 2e_{i} \ ;\ 1 \leq i \leq n \Big\}, \\
& (\Sigma^{+} - \Sigma_{k}^{+})_{0} = \Big\{ e_{i} - e_{j}\ ;\ 1 \leq i < j \leq k \Big\} \\
& \quad\quad\quad\quad\quad\quad\quad\quad \cup \Big\{ e_{i}, e_{i} \pm e_{j}\ ;\ k+1 \leq i < j \leq n \Big\} \cup \Big\{ 2e_{k+1}, \cdots, 2e_{n} \Big\}, \\
& (\Sigma^{+} - \Sigma_{k}^{+})_{\pi} = \Big\{ e_{i} + e_{j}\ ;\ 1 \leq i < j \leq k \Big\} \cup \Big\{ 2e_{1}, \cdots, 2e_{k} \Big\}. \\
\end{split}
\]
If $\Sigma$ is of type $C$, then remove $e_{i}$ from the above.
Let $1 \leq i_{1} < \cdots < i_{k} \leq n$.
If $H \in Q_{ \{ \lambda_{i_{1}}, \cdots, \lambda_{i_{k}}, \delta \} }$, then
\[
\Sigma^{+} - \Sigma^{+}_{H} = \bigcap_{j = 1}^{k} \Big( \Sigma^{+} - \Sigma^{+}_{i_{j}} \Big)_{0}.
\]
Moreover, if $H \in Q_{ \{ \lambda_{i_{1}}, \cdots, \lambda_{i_{k}} \} }$, then
\[
\Sigma^{+} - \Sigma^{+}_{H} = \left( \bigcap_{j = 1}^{k} \Big( \Sigma^{+} - \Sigma^{+}_{i_{j}} \Big)_{0} \right) \cup \left( \bigcap_{j = 1}^{k} \Big( \Sigma^{+} - \Sigma^{+}_{i_{j}} \Big)_{\pi} \right).
\]
We summarize some properties of $\tilde{\Sigma}$ and $\Sigma$ for each irreducible Hermitian symmetric space.
Let $\alpha \in \tilde{\Sigma}$.

\begin{itemize}

\item{$M = G_{k}(\mathbb{C}^{n})\ (2k < n)$}

$\tilde{\Sigma}$ is of type $A_{n-1}$ and $\Sigma$ is of type $BC_{k}$.
For any $1 \leq i \not= j \leq k$, $m(2e_{i}) = 1, m(e_{i} \pm e_{j}) = 2$, and $m(e_{i}) = 2(n-2k)$.
If $\bar{\alpha} = 2e_{i}$, then $\sigma(\alpha) = \alpha$.
If $\bar{\alpha} = e_{i} \pm e_{j}$ or $e_{i}$, then $\sigma(\alpha) \not= \alpha$.

\item{$M = G_{n}(\mathbb{C}^{2n})\ (n \geq 2)$}

$\tilde{\Sigma}$ is of type $A_{2n-1}$ and $\Sigma$ is of type $C_{n}$.
For any $1 \leq i \not= j \leq n$, $m(2e_{i}) = 1$ and $m(e_{i} \pm e_{j}) = 2$.
If $\bar{\alpha} = 2e_{i}$, then $\sigma(\alpha) = \alpha$.
If $\bar{\alpha} = e_{i} \pm e_{j}$, then $\sigma(\alpha) \not= \alpha$.

\item{$M = \tilde{G}_{2}(\mathbb{R}^{n})\ ( n = 2m, m \geq 3)$}

$\tilde{\Sigma}$ is of type $D_{m}$ and $\Sigma$ is of type $C_{2}$.
Then, $m(2e_{1}) = m(2e_{2}) = 1$ and $m(e_{1} \pm e_{2}) = 2(m-2)$.
If $\bar{\alpha} = 2e_{1}$ or $2e_{2}$, then $\sigma(\alpha) = \alpha$.
If $\bar{\alpha} = e_{1} \pm e_{2}$, then $\sigma(\alpha) \not= \alpha$.

\item{$M = \tilde{G}_{2}(\mathbb{R}^{n})\ ( n = 2m + 1, m \geq 2)$}

$\tilde{\Sigma}$ is of type $B_{m}$ and $\Sigma$ is of type $C_{2}$.
Then, $m(2e_{1}) = m(2e_{2}) = 1$ and $m(e_{1} \pm e_{2}) = 2(m-2) + 1$.
If $\bar{\alpha} = 2e_{1}$ or $2e_{2}$, then $\sigma(\alpha) = \alpha$.
If $\bar{\alpha} = e_{1} \pm e_{2}$ and $\alpha$ is a longest root, then $\sigma(\alpha) \not= \alpha$.
If $\bar{\alpha} = e_{1} \pm e_{2}$ and $\alpha$ is a shortest root, then $\sigma(\alpha) = \alpha$.

\item{$M = SO(2n)/U(n)\ (n = 2m, m \geq 2)$}

$\tilde{\Sigma}$ is of type $D_{n}$ and $\Sigma$ is of type $C_{m}$.
For any $1 \leq i \not= j \leq m$, $m(2e_{i}) = 1$ and $m(e_{i} \pm e_{j}) = 4$.
If $\bar{\alpha} = 2e_{i}$, then $\sigma(\alpha) = \alpha$.
If $\bar{\alpha} = e_{i} \pm e_{j}$, then $\sigma(\alpha) \not= \alpha$.

\item{$M = SO(2n)/U(n)\ (n = 2m+1, m \geq 2)$}

$\tilde{\Sigma}$ is of type $D_{n}$ and $\Sigma$ is of type $BC_{m}$.
For any $1 \leq i \not= j \leq m$, $m(2e_{i}) = 1, m(e_{i} \pm e_{j}) = 4$, and $m(e_{i}) = 4$.
If $\bar{\alpha} = 2e_{i}$, then $\sigma(\alpha) = \alpha$.
If $\bar{\alpha} = e_{i} \pm e_{j}$ or $e_{i}$, then $\sigma(\alpha) \not= \alpha$.

\item{$M = Sp(n)/U(n)\ (n \geq 2)$}

Both of $\tilde{\Sigma}$ and $\Sigma$ are of type $C_{n}$.
For any $1 \leq i \not= j \leq n$, $m(2e_{i}) = m(e_{i} \pm e_{j}) = 1$.
For any $\alpha \in \tilde{\Sigma}$, $\sigma(\alpha) = \alpha$.

\item{$M = EIII$}

$\tilde{\Sigma}$ is of type $E_{6}$ and $\Sigma$ is of type $BC_{2}$.
Then, $m(2e_{1}) = m(2e_{2}) = 1, m(e_{1} \pm e_{2}) = 6$, and $m(e_{1}) = m(e_{2}) = 8$.
If $\bar{\alpha} = 2e_{1}$ or $2e_{2}$, then $\sigma(\alpha) = \alpha$.
If $\bar{\alpha} = e_{1} \pm e_{2}, e_{1}, e_{2}$, then $\sigma(\alpha) \not= \alpha$.

\item{$M = EVII$}

$\tilde{\Sigma}$ is of type $E_{7}$ and $\Sigma$ is of type $C_{3}$.
For any $1 \leq i \not= j \leq 3$, $m(2e_{i}) = 1$ and $m(e_{i} \pm e_{j}) = 1$.
If $\bar{\alpha} = 2e_{i}$, then $\sigma(\alpha) = \alpha$.
If $\bar{\alpha} = e_{i} \pm e_{j}$ or $e_{i}$, then $\sigma(\alpha) \not= \alpha$.

\end{itemize}

We can observe that the multiplicity of the longest restricted root is $1$ for each irreducible Hermitian symmetric space.
Hence, set $\alpha_{1}, \cdots, \alpha_{n} \in \tilde{\Sigma}$ such that $\alpha_{i} = \bar{\alpha_{i}} = \lambda_{i}$.
Note that $\alpha_{i}$ is a longest root in $\tilde{\Sigma}$ since $2e_{1}$ is the highest restricted root.
Moreover, let $\alpha \in \tilde{\Sigma}$ satisfy $\bar{\alpha} \not= 0$ and $\bar{\alpha} \not= 2e_{i}$ for $1 \leq i \leq n$.
Then, $\sigma(\alpha) = \alpha$ if and only if $\alpha$ is a shortest root, and $\sigma(\alpha) \not= \alpha$ if and only if $\alpha$ is a longest root.
In general, $|iA_{\delta}|\pi$ is the length of a shortest closed geodesic of $(M, \langle \ ,\ \rangle)$ \cite{Helgason}.
Specifically, we have
\[
|iA_{2e_{1}}|^{2} = \langle iA_{2e_{1}}, iA_{2e_{1}} \rangle = \frac{4}{(2e_{1}, 2e_{1})^{2}} (2e_{1}, 2e_{1}) = \frac{4}{(2e_{1}, 2e_{1})} = d.
\]
Therefore, the length of any shortest closed geodesic of $(M, \langle\ ,\ \rangle)$ is $\sqrt{d}\pi$.
It is known that $G$-invariant Riemannian metrics on $M$ correspond to $d > 0$ one-to-one \cite{Helgason} since $M$ is irreducible.

Set $p_{k} = \mathrm{exp}\,Q_{k}(o)\ (1 \leq k \leq n)$ and $K_{k} = \{ g \in K\ ;\ g(p_{k}) = p_{k} \}$.
The $K$-orbit through $p_{k}$ is denoted by $M^{+}_{k}$.
It is known that $M^{+}_{k}$ is a connected component of $F(s_{o}, M) = \{ x \in M \ ;\ s_{o}(x) = x \}$, where $s_{0}$ is the geodesic symmetry at $o$, so each $M^{+}_{k}$ is a totally geodesic submanifold of $M$.
In particular, $F(s_{o}, M) = \{ o \} \sqcup M^{+}_{1} \sqcup \cdots \sqcup M^{+}_{n}$.
Each $M^{+}_{k}$ is called a {\it polar} \cite{Chen-Nagano}.
If a polar is a one-point set, then we call the polar a {\it pole}.
Moreover, it is known that each polar is a complex submanifold and a Hermitian symmetric space of compact type.
Since a Hermitian symmetric space is simply connected, $K_{k}$ is connected.
Let $\frak{k}_{k}$ be the Lie algebra of $K_{k}$.
Then,
\[
\frak{k}_{k} = \frak{k}_{0} + \sum_{\lambda \in \Sigma^{+} - \Sigma_{k}^{+}}\frak{k}_{\lambda}.
\]

\begin{lemm} \label{2-4-1}
$J \in \frak{k}_{1} \cap \cdots \cap \frak{k}_{n}$.

\end{lemm}

\begin{proof}
We observe that $\{ o,p_{1}, \cdots, p_{n}\}$ is an antipodal set of $M$ \cite{Chen-Nagano2}.
Consequently, there exists a maximal abelian subspace $\frak{t}$ of $\frak{g}$ containing $J$ such that $\{o,p_{1}, \cdots, p_{n}\} \subset M \cap \frak{t}$ \cite{Tanaka-Tasaki}.
Therefore, $[J, p_{1}] = \cdots = [J, p_{n}] = 0$, and $J \in \frak{k}_{1} \cap \cdots \cap \frak{k}_{n}$.
\end{proof}

Because $\Sigma^{+}_{1} \cap \cdots \cap \Sigma^{+}_{n} = \{ 2e_{1}, \cdots, 2e_{n} \}$,
\[
\frak{k}_{1} \cap \cdots \cap \frak{k}_{n} = \frak{k}_{0} + \frak{k}_{2e_{1}} + \cdots + \frak{k}_{2e_{n}} = \frak{k}_{0} + \sum_{i=1}^{n}\mathbb{R}\bar{S}^{(1)}_{\alpha_{i}}.
\]
By Lemma \ref{2-4-1}, there exist $a_{1}, \cdots, a_{n} \in \mathbb{R}$ and $C \in \frak{k}_{0}$ such that
\[
J = a_{1}\bar{S}^{(1)}_{\alpha_{1}} + \cdots + a_{n}\bar{S}^{(1)}_{\alpha_{n}} +C.
\]

\begin{lemm} \label{2-4-2}
$[J, \frak{a}] \subset \sum_{i=1}^{n}\frak{m}_{2e_{i}}$ and $[C, \bar{T}^{(1)}_{2e_{i}}] = 0$ for any $1 \leq i \leq n$.
\end{lemm}

\begin{proof}
For each $1 \leq i \leq n$,
\[
\begin{split}
[J, iH_{2e_{i}}] 
&= [a_{1}\bar{S}^{(1)}_{2e_{1}} + \cdots + a_{n}\bar{S}^{(1)}_{2e_{n}} + C, iH_{2e_{i}}] 
= a_{i}[\bar{S}^{(1)}_{2e_{i}}, iH_{2e_{i}}] \\
&= -a_{i}[iH_{2e_{i}}, \bar{S}^{(1)}_{2e_{i}}] 
= -a_{i} (H_{2e_{i}}, H_{2e_{i}})\bar{T}^{(1)}_{2e_{i}}
= - \left( \frac{4a_{i}}{d} \right) \bar{T}^{(1)}_{2e_{i}}. \\
\end{split}
\]
Therefore, $[J, iH_{2e_{i}}] \in \frak{m}_{2e_{i}}$ and the former part of the statement follows. 
Since $[J, \bar{T}^{(1)}_{2e_{i}}] \in \frak{a}$ and 
\[
[J, \bar{T}^{(1)}_{2e_{i}}] = [a_{i}\bar{S}^{(1)}_{e_{i}} + C, \bar{T}^{(1)}_{2e_{i}}] = a_{i}H_{2e_{i}} + [C, \bar{T}^{(1)}_{2e_{i}}],
\]
the latter part follows by $[\frak{k}_{0}, \frak{m}_{\lambda}] \subset \frak{m}_{\lambda}$ for any $\lambda \in \Sigma^{+}$.
\end{proof}

\begin{lemm} \label{2-4-3}
$|a_{1}| = \cdots = |a_{n}| = \sqrt{d}/2.$
\end{lemm}

\begin{proof}
For $H_{2e_{i}}\ (1 \leq i \leq n)$, using Lemma \ref{2-4-2} we have 
\[
\begin{split}
-iH_{2e_{i}} & = \mathrm{ad}J^{2}(iH_{2e_{i}}) = - \left( \frac{4a_{i}}{d} \right) [J, \bar{T}^{(1)}_{2e_{i}} ]  = - \left( \frac{4a_{i}}{d} \right) [a_{i}\bar{S}^{(1)}_{2e_{i}} + C, \bar{T}^{(1)}_{2e_{1}}] \\
& = - \left( \frac{4a_{i}^{2}}{d} \right) iH_{2e_{i}} -4a_{1}[C, \bar{T}^{(1)}_{2e_{1}}] = - \left( \frac{4a_{i}^{2}}{d} \right) iH_{2e_{i}}. \\
\end{split}
\]
Hence, $d = 4a^2_{i}$ and the statement follows.
\end{proof}

\begin{lemm} \label{2-4-4}
Let $\Sigma$ be of type $BC_{n}$ or $C_{n} \ (n \geq 2)$.
If $\alpha \in \tilde{\Sigma}$ is a longest root and satisfies $\bar{\alpha} = e_{k} + e_{l} \ (1 \leq k \not= l \leq n)$, then $(\alpha, \sigma(\alpha)) = 0$.
Moreover, $\alpha_{k} - \alpha = \sigma(\alpha) - \alpha_{l}$.
\end{lemm}

\begin{proof}
Since $\overline{\alpha + \sigma(\alpha)} = 2(e_{1} \pm e_{k}) \not\in \Sigma$, it follows that $\alpha + \sigma(\alpha) \not\in \tilde{\Sigma}$.
Additionally, it is known that $\alpha - \sigma(\alpha) \not\in \tilde{\Sigma}$ \cite{Araki}.
Thus, we conclude that $(\alpha, \sigma(\alpha)) = 0$.
The latter part is obvious.
\end{proof}

Using a similar argument as in the proof of Lemma \ref{2-4-4}, we obtain Lemma \ref{2-4-5}.

\begin{lemm} \label{2-4-5}
Let $\Sigma$ be of type $BC_{n}$ or $C_{n} \ (n \geq 2)$.
If $\alpha \in \tilde{\Sigma}$ is a longest root and satisfies $\bar{\alpha} = e_{k} - e_{l} \ (1 \leq k \not= l \leq n)$, then $(\alpha, \sigma(\alpha)) = 0$ and $a_{k} - \alpha = \sigma(\alpha) + \alpha_{l}$.
\end{lemm}

\begin{lemm} \label{2-4-6}
Let $\Sigma$ be of type $BC_{n}$ or $C_{n} \ (n \geq 2)$.
Then, $[J, \frak{m}_{e_{k} \pm e_{l}}] = \frak{m}_{e_{k} \mp e_{l}}$ for any $1 \leq k < l \leq n$.
Let $\alpha \in \tilde{\Sigma}$ satisfy $\bar{\alpha} = e_{k} \pm e_{l}$.
If $\alpha$ is a longest root, then $[J, \mathbb{R}T^{(1)}_{\alpha} + \mathbb{R}T^{(2)}_{\alpha}] = \mathbb{R}T^{(1)}_{\alpha_{k} - \alpha} + \mathbb{R}T^{(2)}_{\alpha_{k} - \alpha}$.
If $\alpha$ is a shortest root, then $[J, \mathbb{R}T^{(1)}_{\alpha}] \subset \mathbb{R}T^{(1)}_{\alpha_{k} - \alpha}$.
\end{lemm}

\begin{proof}
Since $J$ is an element of the center of $\frak{k}$, we have
\[
\begin{split}
0 & = [J, S^{(i)}_{\alpha}] = [a_{k}S^{(1)}_{2e_{k}} + a_{l}S^{(1)}_{2e_{l}}, S^{(i)}_{\alpha}] + [C, S^{(i)}_{\alpha}] \quad ( i = 1,2).
\end{split}
\]
Since $[a_{k}S^{(1)}_{2e_{k}} + a_{l}S^{(1)}_{2e_{l}}, S^{(i)}_{\alpha}] \in \frak{k}_{e_{k} \mp e_{l}}$ and $[C, S^{(i)}_{\alpha}] \in \frak{k}_{e_{k} \pm e_{l}}$, we obtain 
$[S^{(1)}_{2e_{k}} + S^{(1)}_{2e_{l}}, S^{(i)}_{\alpha}] = 0$ and $[C, S^{(i)}_{\alpha}] = 0$.
Hence, 
\[
\begin{split}
0 &= [iH_{e_{k} \pm e_{l}}, [C, S^{(i)}_{\alpha}]] = [[iH_{e_{k} \pm e_{l}}, C], S^{(i)}_{\alpha}] + [C, [iH_{e_{k} \pm e_{l}}, S^{(i)}_{\alpha}] ] \\
& = (H_{e_{k} \pm e_{l}}, H_{e_{k} \pm e_{l}})[C, T^{(i)}_{\alpha}] = \left( \frac{2}{d} \right) [C, T^{(i)}_{\alpha} ].
\end{split}
\]
Thus, $[C, T^{(i)}_{\alpha}] = 0$ and $[J, T^{(i)}_{\alpha} ] = [a_{k}S^{(1)}_{2e_{k}} + a_{l}S^{(1)}_{2e_{l}}, T^{(i)}_{\alpha}] \subset \frak{m}_{e_{k} \mp e_{l}}$.
If $\alpha$ is a longest root and $(\epsilon_{\alpha_{k}, -\alpha}, \epsilon_{\alpha_{l}, -\alpha}) = (1,1)$, then by Lemma \ref{2-4-4} and Lemma \ref{2-4-5},
\[
\begin{split}
[J, T^{(1)}_{\alpha}] 
&=
[a_{k}S^{(1)}_{\alpha_{k}} + a_{l}S^{(1)}_{\alpha_{l}}, W_{\alpha} + W_{\sigma(\alpha)}] \\
&=
a_{k}( [U_{\alpha_{k}}, W_{\alpha}] + [ U_{\alpha_{k}}, W_{\sigma(\alpha)}] ) + a_{l}( [U_{\alpha_{l}}, W_{\alpha}] + [U_{\alpha_{l}}, W_{\sigma(\alpha)}] ) \\
&=
a_{k} \Big(-N_{\alpha_{k}, - \alpha}W_{\alpha_{k} - \alpha} - N_{\alpha_{k}, -\sigma(\alpha)}W_{\alpha_{k} - \sigma(\alpha)} \Big) \\
& \quad\quad\quad\quad\quad\quad\quad\quad + a_{l} \Big( -N_{\alpha_{l}, - \alpha}W_{\alpha_{l} - \alpha} - N_{\alpha_{l}, -\sigma(\alpha)}W_{\alpha_{l} - \sigma(\alpha)} \Big) \\
&=
-a_{k}N_{\alpha_{k}, -\alpha}T^{(1)}_{\alpha_{k} - \alpha}  -a_{l}N_{\alpha_{l}, -\alpha}T^{(1)}_{\alpha_{l} - \alpha} \\
&=
-a_{k}N_{\alpha_{k}, -\alpha}T^{(1)}_{\alpha_{k} - \alpha} + a_{l}N_{\alpha_{l}, -\alpha}T^{(1)}_{\alpha_{k} - \alpha}. \\
\end{split}
\]
Thus, we have $[J, \mathbb{R}T^{(1)}_{\alpha}] = \mathbb{R}T^{(1)}_{\alpha_{k} - \alpha} + \mathbb{R}T^{(2)}_{\alpha_{k} - \alpha}$.
In the cases of $(\epsilon_{\alpha_{k}, -\alpha}, \epsilon_{\alpha_{l}, -\alpha}) = (1,-1), (-1,1), (-1,-1)$, similar calculations yield the same result.
Moreover, we also obtain $[J, \mathbb{R}T^{(2)}_{\alpha}] = \mathbb{R}T^{(1)}_{\alpha_{k} - \alpha} + \mathbb{R}T^{(2)}_{\alpha_{k} - \alpha}$.
If $\alpha$ is a shortest root, it is clear that $[J, \mathbb{R}T^{(1)}_{\alpha}] \subset \mathbb{R}T^{(1)}_{\alpha_{k} - \alpha}$ by considering the multiplicity of the restricted root system.
\end{proof}

\begin{lemm} \label{2-4-7}
Let $\Sigma$ be of type $BC_{n} \ (n \geq 1)$.
Then, $[J, \frak{m}_{e_{k}}] \subset \frak{m}_{e_{k}}$ for any $1 \leq k \leq n$.
Moreover, if $\alpha \in \tilde{\Sigma}$ satisfies $\bar{\alpha} = e_{k}\ ( 1 \leq k \leq n)$, then $[J, \mathbb{R}T^{(1)}_{\alpha} + \mathbb{R}T^{(2)}_{\alpha}] = \mathbb{R}T^{(1)}_{\alpha} + \mathbb{R}T^{(2)}_{\alpha}$.
\end{lemm}

\begin{proof}
By Lemma \ref{2-3-4}, it is evident that $[J, \frak{m}_{e_{k}}] \subset \frak{m}_{e_{k}}$.
Since $J$ is the element of the center of $\frak{k}$, we have $[C, S^{(i)}_{\alpha}] = -a_{k}[S^{(1)}_{\alpha_{k}}, S^{(i)}_{\alpha}] \ (i =1,2)$.
Hence, for any $H \in \frak{a}_{0}$ and $i = 1,2$,
\[
[C, T^{(i)}_{\alpha}] 
= 
\frac{1}{\alpha(H)}[iH, [C, S^{(i)}_{\alpha}]] = -\frac{a_{k}}{\alpha(H)}[iH, [S^{(1)}_{\alpha_{k}}, S^{(i)}_{\alpha}]].
\]
Note that $\alpha_{k} - \alpha = \sigma(\alpha)$.
If $\epsilon_{\alpha_{k}, -\alpha} = 1$, then
\[
\begin{split}
[S^{(1)}_{\alpha_{k}}, S^{(1)}_{\alpha}] 
&= 
\sqrt{2}[U_{\alpha_{k}}, U_{\alpha} + U_{\sigma(\alpha)}] \\
&=
\sqrt{2} \Big( N_{\alpha_{k}, -\alpha}U_{\alpha_{k} - \alpha} + N_{\alpha_{k}, -\sigma(\alpha)}U_{\alpha_{k} - \sigma(\alpha)} \Big) = N_{\alpha_{k}, -\alpha}S^{(1)}_{\alpha} , \\
[S^{(1)}_{\alpha_{k}}, S^{(2)}_{\alpha}] 
&= \sqrt{2}[U_{\alpha_{k}}, W_{\alpha} - W_{\sigma(\alpha)}] \\
& =
\sqrt{2} \Big( -N_{\alpha_{k}, -\alpha}W_{\alpha_{k} - \alpha} + N_{\alpha_{k}, -\sigma(\alpha)}W_{\alpha_{k} - \sigma(\alpha)} \Big) = N_{\alpha_{k}, -\alpha}S^{(2)}_{\alpha} \\
\end{split}
\]
Hence, $[C, T^{(i)}_{\alpha}] \subset \mathbb{R}T^{(1)}_{\alpha} + \mathbb{R}T^{(2)}_{\alpha} \ (i=1,2)$.
Moreover, 
\[
\begin{split}
[S^{(1)}_{\alpha_{k}}, T^{(1)}_{\alpha}] 
&=
[U_{\alpha_{k}}, \sqrt{2}(W_{\alpha} + W_{\sigma(\alpha)})] \\
&=
\sqrt{2} \Big( -N_{\alpha_{k}, -\alpha}W_{\alpha_{k} - \alpha} -N_{\alpha_{k}, -\sigma(\alpha)}W_{\alpha_{k} - \sigma(\alpha)} \Big)
= -N_{\alpha_{k}, -\alpha}T^{(1)}_{\alpha}, \\
[S^{(1)}_{\alpha_{k}}, T^{(2)}_{\alpha}] 
&=
[U_{\alpha_{k}}, \sqrt{2}(-U_{\alpha} + U_{\sigma(\alpha)})] \\
& =
\sqrt{2} \Big( -N_{\alpha_{k}, -\alpha}U_{\alpha_{k} - \alpha} + N_{\alpha_{k}, -\sigma(\alpha)}U_{\alpha_{k} - \sigma(\alpha)} \Big)= -N_{\alpha_{k}, -\alpha}T^{(2)}_{\alpha}. \\
\end{split}
\]
Therefore, we obtain $[J, \mathbb{R}T^{(1)}_{\alpha} + \mathbb{R}T^{(2)}_{\alpha}] = \mathbb{R}T^{(1)}_{\alpha} + \mathbb{R}T^{(2)}_{\alpha}$.
In the case of $\epsilon_{\alpha_{k}, -\alpha} = -1$, by a similar argument, we see that the statement follows.
\end{proof}

Summarizing Lemma \ref{2-4-2}, Lemma \ref{2-4-6}, and Lemma \ref{2-4-7}, if $\Sigma$ is of type $C_{n} \ (n \geq 2)$, then
\[
[J, \frak{m}_{2e_{i}}] \subset \frak{a}, \quad [J, \frak{m}_{e_{i} \pm e_{j}}] \subset \frak{m}_{e_{i} \mp e_{j}}, \quad(1 \leq i < j \leq n)
\]
and if $\Sigma$ is type of $BC_{n} (n \geq 1)$, then
\[
[J, \frak{m}_{2e_{i}}] \subset \frak{a}, \quad [J, \frak{m}_{e_{i} \pm e_{j}}] \subset \frak{m}_{e_{i} \mp e_{j}}, \quad [J, \frak{m}_{e_{i}}] \subset \frak{m}_{e_{i}}. \quad(1 \leq i < j \leq n).
\]
Let $\Sigma$ be of type $C_{n}$.
Then, we can easliy check that the $K$-orbit through $\mathrm{exp}tQ_{n}(o) \ (0 < t < 1)$ is a totally real submanifold.
In particular, these orbits are Lagrangian submanifolds.
Note that there are no totally real orbits, except for these orbits in each $M$.
By these arguments, we obtain Theorem \ref{2-4-8}.

\begin{thm} \label{2-4-8}
Any orbit $N$ of the isotropy group action on $M$ is a $CR$ submanifold.
In particular, $N$ is a complex submanifold if and only if $N$ is a polar.
Moreover, $N$ is a totally real submanifold if and only if $\Sigma$ is of type $C_{n}\ (n \geq 2)$ and $N$ is the orbit through $x = (\mathrm{exp}\,tQ_{n})(o)\ (0 < t < 1)$.
These totally real orbits are Lagrangian submanifolds.
\end{thm}

%%%%%%%%%%%%%%%%%%%%%%%%%%%%%%%%%%%%%%%%%%%%%%%%%%%%%%%%%%%%%%%%%%%%%%%%%%%%%%%%%%%%%%%%%%%%%%%%%%%%%%%%%%%%%%%%%%%%%%%%%%%%%%%%%%%%%%%%%%%%%%%%%%%%%%%

\section{Isotropy orbits as a $CR$ submanifold} \label{s3}

In this section, we study some properties of each orbit of the isotropy group action from the perspective of a $CR$ submanifold.
Let $H \in Q$ and $x = \mathrm{exp}(H)(o)$.
Denote by $N$ the $K$-orbit through $x$.
Define the following subsets of $\Sigma^{+}_{x}$:
\[
\begin{split}
(\Sigma^{+}_{x})_{R} &= \{ \lambda \in \Sigma^{+}_{x} \ ;\ J(\frak{m}_{\lambda}) \subset \frak{m}_{x}^{\perp} \}, \quad (\Sigma^{+}_{x})_{C} = \{ \lambda \in \Sigma^{+}_{x} \ ;\ J(\frak{m}_{\lambda}) \subset \frak{m}_{x} \}.
\end{split}
\]
By Theorem \ref{2-4-8}, we have $\Sigma^{+} = (\Sigma^{+}_{x})_{C} \sqcup (\Sigma^{+}_{x})_{R}$.
Define the subspaces $\mathcal{D}_{x}$ and $\mathcal{D}_{x}^{\perp}$ of $T_{x}N$ as follows:
\[
\mathcal{D}_{x} = (\mathrm{exp}H)_{*} \left( \sum_{\lambda \in (\Sigma^{+}_{x})_{C}}\frak{m}_{\lambda} \right), \quad\quad
\mathcal{D}_{x}^{\perp} = (\mathrm{exp}H)_{*} \left( \sum_{\lambda \in (\Sigma^{+}_{x})_{R}}\frak{m}_{\lambda} \right).
\]
Then, $J(\mathcal{D}_{x}) \subset \mathcal{D}_{x}$ and $J(\mathcal{D}_{x}^{\perp}) \subset T_{x}^{\perp}N$.
Since $J$ is invariant under $G$, the subspaces $\mathcal{D}_{x}$ and $\mathcal{D}_{x}^{\perp}$ are invariant under $K_{x}$.
Thus, we can define the $K$-invariant distributions $\mathcal{D}$ and $\mathcal{D}^{\perp}$ by $\mathcal{D}_{x}$ and $\mathcal{D}_{x}^{\perp}$, respectively.
By definition, $\mathcal{D}$ and $\mathcal{D}^{\perp}$ are the complex distribution and the totally real distribution of $N$, respectively.
In the following subsections, we will study the leaves of the totally real foliation $\mathcal{D}^{\perp}$ and the integrability of the complex distribution $\mathcal{D}$.
We will consider two cases separately as follows:  Let $1 \leq i_{1} < \cdots < i_{k} \leq n$. 
Case (i) is $H \in Q_{ \{ \lambda_{i_{1}}, \cdots, \lambda_{i_{k}}, \delta \} }$ and case (ii) is $\ H \in Q_{ \{ \lambda_{i_{1}}, \cdots, \lambda_{i_{k}} \}}$.
In the following subsections, we use the following notations.
For $1 \leq i_{1} < \cdots < i_{k} \leq n$, we define 
\[
\begin{split}
& I_{0} = \{ 1, 2,  \cdots, i_{1} \}, \\
&  I_{l} = \{ i_{l} + 1, i_{1} + 2, \cdots, i_{l+1} -1 \} \quad  (1 \leq l \leq k-1), \\ 
& I_{k} = \{ i_{k}+1, i_{k} + 2,  \cdots, n \}.
\end{split}
\]
Then, we set for $ 0 \leq l \leq k$,
\[
\begin{split}
& \Sigma^{+}(l,+) = \{ e_{i} + e_{j}\ ;\ i,j \in I_{l},\ i < j \},  \\
& \Sigma^{+}(l,-) = \{ e_{i} - e_{j}\ ;\ i,j \in I_{l},\ i < j \}, \\
& \Sigma^{+}(l) = \Sigma^{+}(l,+) \cup \Sigma^{+}(l,-) = \{ e_{i} \pm e_{j}\ ;\ i,j \in I_{l},\ i < j \}. \\
\end{split}
\]

%%%%%%%%%%%%%%%%%%%%%%%%%%%%%%%%%%%%%%%%%%%%%%%%%%%%%%%%%%%%%%%%%%%%%%%%%%%%%%%%%%%%%%%%%%%%%%%%%%%%%%%%%%%%%%%%%%%%%%%%%%%%%%%%%%%%%%%%%%%%%%%%%%%%%%%

\subsection{Case. 1} \label{s3-1}

In this subsection, we consider case (i).
We asuume that $\Sigma$ is of type $BC_{n} \ (n \geq 1)$.
In the case of type $C_{n} \ (n \geq 2)$, the same results follow using a similar arguments.
Let $H \in Q_{ \{ \lambda_{i_{1}}, \cdots, \lambda_{i_{k}}, \delta \} }$.
There exist $0 <t_{1}, \cdots, t_{k} < 1$ such that $0 < t_{1} + \cdots + t_{k} < 1$ and $H = t_{1}Q_{i_{1}} + \cdots t_{k}Q_{i_{k}}$.
Then,
\[
\begin{split}
& \Sigma^{+}_{x} = \Sigma^{+}(0,+) \cup \cdots \cup \Sigma^{+}(k-1,+) \cup \{ e_{i} \pm e_{j} \ ;\ i \in I_{l}, \ j \in I_{m}, \ 0 \leq l < m \leq k \} \\
& \quad\quad\quad\quad\quad\quad \quad\quad\quad\quad\quad\quad\quad\quad\quad \quad\quad\quad\quad\quad\quad\quad\quad\quad \quad\quad\quad   \cup \{ e_{v}, 2e_{v} \ ;\ 1 \leq v \leq i_{k} \}, \\
& (\Sigma^{+}_{x})_{C} = \{ e_{i} \pm e_{j} \ ;\ i \in I_{l}, \ j \in I_{m}, \ 0 \leq l < m \leq k \} \cup \{ e_{v} \ ;\ 1 \leq v \leq i_{k} \}. \\
& (\Sigma^{+}_{x})_{R} = \Sigma^{+}(0,+) \cup \cdots \cup \Sigma^{+}(k-1,+) \cup \{ 2e_{v}\ ;\ 1 \leq v \leq i_{k} \}, \\
& \Sigma^{+} - \Sigma^{+}_{x} = \Sigma^{+}(0,-) \cup \cdots \cup \Sigma^{+}(k-1,-) \cup \Sigma^{+}(k) \cup \{ e_{u}, 2e_{u} \ ;\ u \in I_{k} \}. \\
\end{split}
\]
Hence,
\[
\begin{split}
\mathcal{D}_{x} &= (\mathrm{exp}H)_{*} \left( \sum_{i \in I_{l}, j \in I_{m}, l < m} \frak{m}_{e_{i} \pm e_{j}} + \sum_{v=1}^{i_{k}} \frak{m}_{e_{v}} \right), \\
\mathcal{D}^{\perp}_{x} & = (\mathrm{exp}H)_{*} \left( \sum_{\lambda \in \Sigma^{+}(0,+) \cup \cdots \cup \Sigma^{+}(k-1,+)} \frak{m}_{\lambda} + \sum_{v=1}^{i_{k}} \frak{m}_{2e_{v}} \right).
\end{split}
\]
We denote the identity component of $K_{i_{1}} \cap \cdots \cap K_{i_{k}}$ by $K_{i_{1}, \cdots, i_{k}}$.
The Lie algebra $\frak{k}_{i_{1}, \cdots, i_{k}}$ of $K_{i_{1}, \cdots, i_{k}}$ is given by
\[
\begin{split}
\frak{k}_{i_{1}, \cdots, i_{k}} &= \frak{k}_{i_{1}} \cap \cdots \cap \frak{k}_{i_{k}} 
= \frak{k}_{0} + \sum_{\lambda \in \Sigma^{+} - ( \Sigma^{+}_{i_{1}} \cup \cdots \cup \Sigma^{+}_{i_{k}})} \frak{k}_{\lambda} \\
& =  \frak{k}_{0} + \sum_{s=1}^{n}\frak{k}_{2e_{s}} + \sum_{u \in I_{k}}\frak{k}_{e_{u}} + \sum_{\lambda \in \Sigma^{+}(0) \cup \cdots \cup \Sigma^{+}(k)}\frak{k}_{\lambda}.
\end{split}
\]
The $K_{i_{1}, \cdots, i_{k}}$-orbit through $x$ is denoted by $L$.

\begin{lemm} \label{3-2-1}
$L$ is the integral submanifold of the totally real distribution $\mathcal{D}^{\perp}$ through $x$.
\end{lemm}

\begin{proof}
Since $\Sigma^{+} - (\Sigma^{+}_{i_{1}} \cup \cdots \cup \Sigma^{+}_{i_{k}}) = \Sigma^{+}(0) \cup \cdots \Sigma^{+}(k) \cup \{ e_{i} \ ;\ i \in I_{k} \} \cup \{ 2e_{i} \ ;\ 1 \leq i \leq n \}$, we have
\[
\begin{split}
& \Big\{ \lambda \in \Sigma^{+} - (\Sigma^{+}_{i_{1}} \cap \cdots \cap \Sigma^{+}_{i_{k}}) \ ;\ \lambda(H) \in i\pi\mathbb{Z} \Big\} \\
& \quad\quad\quad\quad\quad\quad = \Sigma^{+}(0,-) \cup \cdots \cup \Sigma^{+}(k-1,-) \cup \Sigma^{+}(k) \cup \Big\{ e_{u}, 2e_{u} \ ;\ u \in I_{k} \Big\}.
\end{split}
\]
Thus, $T_{x}L$ is given by
\[
(\mathrm{exp}H)_{*} \left( \sum_{\lambda \in \Sigma^{+}(0,+) \cup \cdots \cup \Sigma^{+}(k-1,+)}\frak{m}_{\lambda} + \sum_{v=1}^{i_{k}}\frak{m}_{2e_{v}}  \right).
\]
In particular, $\mathcal{D}_{x}^{\perp} = T_{x}L$.
Since the totally real distribution $\mathcal{D}^{\perp}$ is invariant under $K$, the integral submanifold of $\mathcal{D}^{\perp}$ through $x$ is $L$.
\end{proof}

Since $\mathcal{D}^{\perp}$ is $K$-invariant, $k(L)$ is the integral submanifold of $\mathcal{D}^{\perp}$ through $k(x)$ for any $k \in K$.
Therefore, the set $\{ k(L) \ ;\ k \in K \}$ is the totally real foliation of $N$.
The second fundamental form of $L \subset N$ is denoted by $h$, and the second fundamental form of  $L \subset M$ is denoted by $h^{L}$.

\begin{prop} \label{3-2-2}
For any $y \in N$, the leaf $L_{y}$ is a totally geodesic submanifold of $N$.
In particular, $N$ is a Hopf $CR$ submanifold.
\end{prop}

\begin{proof}
It is sufficient to show that $L$ is a totally geodesic submanifold of $N$.
Let $\alpha, \beta \in \tilde{\Sigma}$ be $\bar{\alpha}, \bar{\beta} \in (\Sigma^{+}_{x})_{R}$.
Set $\lambda = \bar{\alpha}$ and $\mu = \bar{\beta}$.
Then, by Lemma \ref{2-2-1}, for any $i, j=1,2$,
\[
\begin{split}
& \Big( \nabla^{M}_{(S^{(i)}_{\alpha})^{*}}(S^{(j)}_{\beta})^{*} \Big)_{x} \\
&=
(\mathrm{exp}H)_{*} \Big[ \Big( (\mathrm{Ad(exp}H)^{-1})S^{(i)}_{\alpha} \Big)_{\frak{m}}, \mathrm{Ad(exp}H)^{-1}S^{(j)}_{\beta} \Big]_{\frak{m}} \\
&= - \Big( \sin ((-i) \lambda(H)) \cos ((-i)\mu(H)) \Big) (\mathrm{exp}H)_{*} \big[ T^{(i)}_{\alpha}, S^{(j)}_{\beta} \big] \in (\mathrm{exp}H)_{*} [\frak{m}_{\lambda}, \frak{k}_{\mu}].
\end{split}
\]
Assume $\lambda \not= \mu$.
Then, $\lambda + \mu \not\in \Sigma$.
If $\lambda - \mu \in \Sigma$, then either $\lambda - \mu$ or $\mu - \lambda$ is an element of $\Sigma^{+} - \Sigma^{+}_{x}$ by the definition of $(\Sigma^{+}_{x})_{R}$.
Assume $\lambda = \mu$.
Since $2\lambda \not\in \Sigma^{+}$ by the definition of $(\Sigma^{+}_{x})_{R}$, we see $[\frak{m}_{\lambda}, \frak{k}_{\mu}] \subset \frak{a}$ by Lemma \ref{2-3-4}.
Thus, $(\mathrm{exp}H)_{*}[\frak{m}_{\lambda}, \frak{k}_{\mu}] \subset T_{x}^{\perp}N$ and
\[
\begin{split}
h \Big( (S^{(i)}_{\alpha})_{x}^{*}, (S^{(j)}_{\beta})_{x}^{*} \Big) & = \Big( \big( \nabla^{M}_{(S^{(i)}_{\alpha})^{*}}(S^{(j)}_{\beta}))^{*} \big)_{h} \Big)_{\mathcal{D}_{x}} = 0. \\
\end{split}
\]
Therefore, $L$ is a totally geodesic submanifold of $N$, and the statement follows.
\end{proof}

By the proof of Lemma \ref{3-2-2}, we obtain Lemma \ref{3-2-3}.

\begin{lemm} \label{3-2-3}
$L$ is a totally geodesic submanifold of $M$ if and only if $\lambda(H) = i\pi/2$ for any $\lambda \in (\Sigma^{+}_{H})_{R}$.
\end{lemm}

\begin{lemm} \label{3-2-4}
$\lambda(H) =i \pi/2$ for any $\lambda \in (\Sigma^{+}_{x})_{R}$ if and only if $H = (1/2)Q_{a}$ for some $1 \leq a \leq n$.
\end{lemm}

\begin{proof}
If $H = (1/2)Q_{a}$, then it is obvious that $\lambda(H) =i \pi/2$ for any $\lambda \in (\Sigma^{+}_{x})_{R}$.
Conversely, assume that $\lambda(H) =i \pi/2$ for any $\lambda \in (\Sigma^{+}_{x})_{R}$.
Let $H = (t_{1}, \cdots, t_{n})$ and $H \in Q_{i_{1}, \cdots, i_{k}}$ for $1 \leq i_{1} < \cdots < i_{k} \leq n$.
Then, since $2e_{j}(H) = i\pi/2$ for any $1 \leq j \leq i_{k}$, we have $t_{j} = \pi/4$ for any $1 \leq j \leq i_{k}$ and the statement follows.
\end{proof}

Summarizing Lemma \ref{3-2-3} and Lemma \ref{3-2-4}, we obtain Lemma \ref{3-2-5}.

\begin{prop} \label{3-2-5}
$N$ is a ruled $CR$ submanifold of $M$ if and only if $H = (1/2)Q_{a}$ for some $1 \leq a \leq n$.
\end{prop}

\begin{remark} \label{3-2-6}
(1) \ Let $H = tQ_{a}$ for any $0 < t < 1$ and $1 \leq a \leq n$, and $x = \mathrm{exp}H (o)$.
In this case, $K_{x}$ is a subgroup of $K_{a}$.
Thus, we can consider the fibration $K_{x}/K_{a} \rightarrow K/K_{x} \rightarrow K/K_{a}$.
Since $K_{x}/K_{a} \cong L_{x}$, the $K$-orbit $N$ through $x$ is a fiber bundle, where the fiber is $k(L_{x})\ (k \in K)$, and the base space is $M^{+}_{a}$.
In particular, the fiber of $N$ is the totally real foliation, and the base space of $N$ is the leaf space of the totally real foliation.

(2) \ If $\Sigma$ is of type $C_{n}$ and $H = (1/2)Q_{n}$, then $N$ is a totally geodesic Lagrangian submanifold.
\end{remark}

Next, we consider the integrability of the complex distribution $\mathcal{D}$.

\begin{lemm} \label{3-2-7}
Let $\lambda \in \Sigma^{+}$ and $\alpha, \beta \in \tilde{\Sigma}$ be $\bar{\alpha} = \bar{\beta} = \lambda$.
Moreover, assume $\alpha \not= \beta, \sigma(\beta)$.
If $2\lambda \not\in \Sigma^{+}$, then $[S^{(i)}_{\alpha}, T^{(j)}_{\beta}] = 0$ for any $i,j = 1,2$.
\end{lemm}

\begin{proof}
Since $2\lambda \not\in \Sigma$, we obtain $[\frak{k}_{\lambda}, \frak{m}_{\lambda}] \subset \frak{a}$ by Lemma \ref{2-3-4}.
For any $H \in \frak{a}_{0}$,
\[
\langle iH, [S^{(i)}_{\alpha}, T^{(j)}_{\beta} ] \rangle = \langle [iH, S^{(i)}_{\alpha}], T^{(j)}_{\beta} \rangle = \lambda(H) \langle T^{(i)}_{\alpha}, T^{(j)}_{\beta} \rangle = 0.
\]
Thus, we conclude that $[S^{(i)}_{\alpha}, T^{(j)}_{\beta} ] = 0$.
\end{proof}

The second fundamental form $N \subset M$ is denoted by $h^{N}$.

\begin{prop} \label{3-2-8}
For any $H = t_{1}Q_{i_{1}} + \cdots + t_{k}Q_{i_{k}}$, where $0 < t_{1} + \cdots + t_{i_{k}} < 1$ and $0 < t_{l} < 1$ for $1 \leq l \leq k$, the complex distribution of $N$ is not integrable.
\end{prop}

\begin{proof}
We use Lemma \ref{2-2}.
First, let $\Sigma$ is of type $BC_{n}$ or $C_{n} \ (n \geq 2)$.
Let $\alpha \in \tilde{\Sigma}$ be a longest root and $\bar{\alpha} = e_{i} + e_{j} \ (i \in I_{l}, \, j \in I_{m}, \, 1 \leq l < m \leq k-1)$.
By Lemma \ref{3-2-7}, for any $a = 1,2$,
\[
\begin{split}
& h^{N} \Big( (\bar{S}^{(a)}_{\alpha})_{x}^{*}, (\bar{S}^{(a)}_{\alpha})_{x}^{*} \Big) \\ 
&= 
\Big( - \frac{1}{2}\sin \big( 2(-i)(e_{i} + e_{j})(H) \big)(\mathrm{exp}H)_{*} \Big[ \bar{T}^{(a)}_{\alpha}, \bar{S}^{(a)}_{\alpha} \Big] \Big)_{T_{x}^{\perp}N} \\
&= 
\frac{1}{2}\sin \big( 2(-i)(e_{i} + e_{j})(H) \big)(\mathrm{exp}H)_{*} iH_{e_{i} + e_{j}}, \\
& h^{N} \Big( (\bar{S}^{(1)}_{\alpha_{i} - \alpha})_{x}^{*}, (\bar{S}^{(2)}_{\alpha_{i} - \alpha})_{h}^{*} \Big) \\
&=
\Big( - \frac{1}{2}(\mathrm{exp}H)_{*} \sin \big( 2(-i)(e_{i} - e_{j})(H) \big) [ \bar{T}^{(1)}_{\alpha_{i} - \alpha}, \bar{S}^{(2)}_{\alpha_{i} - \alpha} ] \Big)_{T_{x}^{\perp}N} = 0. \\
\end{split}
\]
By Lemma \ref{2-4-6}, there exists $(x_{1}, x_{2}) \in \mathbb{R}^{2} \ (x_{1}^{2} + x_{2}^{2} =1)$ such that $J(\bar{T}^{(1)}_{\alpha}) = x_{1}\bar{T}^{(1)}_{\alpha_{i} - \alpha} + x_{2}\bar{T}^{(2)}_{\alpha_{i} - \alpha}$.
Then,
\[
\begin{split}
J(\bar{S}^{(1)}_{\alpha})^{*}_{x} 
&= - (\mathrm{exp}H)_{*} J \Big( \sin \big( (-i) (e_{i} + e_{j})(H) \big) \bar{T}^{(1)}_{\alpha} \Big) \\
&= - \sin \big( (-i) (e_{i} + e_{j})(H) \big) (\mathrm{exp}H)_{*} \left( x_{1}\bar{T}^{(1)}_{\alpha_{i} - \alpha} + x_{2}\bar{T}^{(2)}_{\alpha_{i} - \alpha} \right) \\
&= \frac{ \sin \big( (-i) (e_{i} + e_{j})(H) \big) }{ \sin \big( (-i) (e_{i} - e_{j})(H) \big) }\left( x_{1}(\bar{S}^{(1)}_{\alpha_{i} - \alpha})^{*}_{x} + x_{2}(\bar{S}^{(2)}_{\alpha_{i} - \alpha})^{*}_{x} \right). \\
\end{split}
\]
Therefore,
\[
\begin{split}
& h^{N}( J(\bar{S}_{\alpha}^{(1)})_{x}^{*} ,J(\bar{S}_{\alpha}^{(1)})_{x}^{*} ) \\
&= 
(\mathrm{exp}H)_{*} 
\left(
\frac{ \sin \big( (-i) (e_{i} + e_{j})(H) \big) }{ \sin \big( (-i) (e_{i} - e_{j})(H) \big) }
\right)^{2}
\Big( x_{1}^{2} h^{N} \Big( (\bar{S}^{(1)}_{\alpha_{i} -\alpha})_{x}^{*}, (\bar{S}^{(1)}_{\alpha_{i} -\alpha})_{x}^{*} \Big) \\
& \quad\quad\quad\quad\quad\quad\quad\quad\quad\quad\quad\quad\quad\quad\quad\quad\quad\quad\quad\quad + x_{2}^{2} h^{N} \Big( (\bar{S}^{(2)}_{\alpha_{i} -\alpha})_{x}^{*}, (\bar{S}^{(2)}_{\alpha_{i} -\alpha})_{x}^{*} \Big)
\\
&=
(\mathrm{exp}H)_{*} 
\left(
\frac{ \sin \big( (-i) (e_{i} + e_{j})(H) \big) }{ \sin \big( (-i) (e_{i} - e_{j})(H) \big) } \right)^{2}\frac{1}{2} \sin \big( 2(-i)(e_{i} - e_{j})(H) \big) \big)iH_{e_{i} - e_{j}}. \\
\end{split}
\]
Since $(\mathrm{exp}H)_{*}T^{(1)}_{\alpha_{i}}, (\mathrm{exp}H)_{*}T^{(1)}_{\alpha_{j}} \in \mathcal{D}_{x}^{\perp}$, we obtain $(\mathrm{exp}H)_{*}(iH_{e_{i} \pm e_{j}}) \in J(\mathcal{D}_{x}^{\perp})$.
Assume that 
\[
h^{N} \Big( (\bar{S}^{(1)}_{\alpha})_{x}^{*}, (\bar{S}^{(1)}_{\alpha})_{x}^{*} \Big)  + h^{N} \Big( J(\bar{S}^{(1)}_{\alpha})_{x}^{*} ,J(\bar{S}^{(1)}_{\alpha})_{x}^{*} \Big)
\]
is orthogonal to $(\mathrm{exp}H)_{*}(iH_{e_{i} \pm e_{j}})$.
This implies that $(e_{i} \pm e_{j})(H) = i(\pi/2)$, so $e_{i}(H) = i\pi/2$ and $e_{j}(H) = 0$.
Thus, if we denote $H = (t_{1}, \cdots, t_{n})$, then $t_{i} = 0$ or $\pi/2$ for any $1 \leq i \leq n$.
However, there does not exist such $H$ in case (i) and this is a contradiction.
Hence, $h^{N} ((\bar{S}_{\alpha})_{h}^{*}, (\bar{S}_{\alpha})_{h}^{*}) + h^{N}( I((\bar{S}_{\alpha})^{*}_{h}) ,I((\bar{S}_{\alpha})^{*}_{h}) )$ is not orthogonal to at least one of  $(\mathrm{exp}H)_{*}(iH_{e_{i} \pm e_{j}})$.
Therefore, by Lemma \ref{2-2}, the complex distribution $\mathcal{D}$ of $N$ is not integrable.
In the case where $\Sigma$ is of type $BC_{1}$, by replacing $e_{i} - e_{j}$ with $e_{i}$ and using a similar arguments, we can verify that the statement is true.
\end{proof}

%%%%%%%%%%%%%%%%%%%%%%%%%%%%%%%%%%%%%%%%%%%%%%%%%%%%%%%%%%%%%%%%%%%%%%%%%%%%%%%%%%%%%%%%%%%%%%%%%%%%%%%%%%%%%%%%%%%%%%%%%%%%%%%%%%%%%%%%%%%%%%%%%%%%%%%

\subsection{Case. 2} \label{s3-2}

In this subsection, we consider case (ii).
We assume that $\Sigma$ is of type $BC_{n}\ (n \geq 2)$, as in the previous subsection.
Let $H \in Q_{\{ \lambda_{i_{1}}, \cdots, \lambda_{i_{k}}\}}$.
There exist $0 < t_{i} < 1 \ (1 \leq i \leq k)$ such that $ t_{1} + \cdots + t_{k} = 1$ and $H = t_{1}Q_{i_{1}} + \cdots + t_{k}Q_{i_{k}}$.
Then,
\[
\begin{split}
& \Sigma^{+}_{x} = \Sigma^{+}(1,+) \cup \cdots \cup \Sigma^{+}(k-1,+) \cup \{ 2e_{q}\ ;\ i_{1} + 1 \leq q \leq i_{k} \} \cup \{ e_{p}\ ;\ 1 \leq p \leq i_{k} \} \\
& \quad\quad\quad\quad\quad\quad\quad\quad\quad\quad\quad\quad\quad\quad\quad\quad\quad\quad  \cup \{ e_{i} \pm e_{j} \ ;\ i \in I_{l}, j \in I_{m}, 0 \leq l < m \leq k \}, \\
& (\Sigma^{+}_{x})_{C} =  \{ e_{i} \pm e_{j} \ ;\ i \in I_{l}, j \in I_{m}, 0 \leq l < m \leq k \} \cup \{ e_{p}\ ;\ 1 \leq p \leq i_{k} \}, \\
& (\Sigma^{+}_{x})_{R} = \Sigma^{+}(1,+) \cup \cdots \cup \Sigma^{+}(k-1,+) \cup \{ 2e_{q}\ ;\ i_{1} + 1 \leq q \leq i_{k} \}, \\
& \Sigma^{+} - \Sigma^{+}_{x} = \Sigma^{+}(0) \cup \Sigma^{+}(1,-) \cup \cdots \cup \Sigma^{+}(k-1,-) \cup \Sigma^{+}(k) \\
& \quad\quad\quad\quad\quad\quad\quad\quad\quad\quad\quad\quad\quad\quad\quad\quad\quad\quad\quad\quad \cup \{ 2e_{v} \ ;\ v \in I_{0} \} \cup \{ 2e_{s}, e_{s}\ ;\ s \in I_{k} \}.  \\
\end{split}
\]
Therefore,
\[
\begin{split}
\mathcal{D}_{x} &= (\mathrm{exp}H)_{*} \left( \sum_{i \in I_{l}, j \in I_{m}, 0 \leq l < m \leq k} \frak{m}_{e_{i} \pm e_{j}} + \sum_{p=1}^{i_{k}} \frak{m}_{e_{p}} \right), \\
\mathcal{D}^{\perp}_{x} &= (\mathrm{exp}H)_{*} \left( \sum_{\lambda \in \Sigma^{+}(1,+) \cup \cdots \cup \Sigma^{+}(k-1,+)} \frak{m}_{\lambda} + \sum_{q = i_{1} + 1}^{i_{k}} \frak{m}_{2e_{q}} \right).
\end{split}
\]
Define $K_{i_{1}, \cdots, i_{k}}, \frak{k}_{i_{1}, \cdots, i_{k}}$, and $L$ as in the previous subsection.

\begin{lemm} \label{3-3-1}
$L$ is the integral submanifold of the totally real distribution $\mathcal{D}^{\perp}$ through $x$.
\end{lemm}

\begin{proof}
Since
\[
\begin{split}
& \Big\{ \lambda \in \Sigma^{+}(0) \cup \cdots \cup \Sigma^{+}(k) \cup \{ e_{i} \ ;\ i \in I_{k} \} \cup \{ 2e_{j} \ ;\ 1 \leq j \leq n \} \ ;\ \lambda(H) \in i\pi \mathbb{Z} \Big\} \\
& =  
\Sigma^{+}(0) \cup \Sigma^{+}(1,-) \cup \cdots \cup \Sigma^{+}(k-1,-) \cup \Sigma^{+}(k) \cup \{ 2e_{v} \ ;\ v \in I_{0} \} \cup \{ e_{i}, 2e_{i} \ ;\ i \in I_{k} \},
\end{split}
\]
we see $T_{x}L = \mathcal{D}^{\perp}_{x}$.
\end{proof}

We use the notations introduced in the previous subsection.

\begin{prop} \label{3-3-2}
$L$ is a totally geodesic submanifold of $N$.
In particular, $N$ is a Hopf $CR$ submanifold.
\end{prop}

\begin{proof}
For any $\lambda, \mu \in (\Sigma^{+}_{x})_{R} \ (\lambda \not= \mu)$, we have $\lambda + \mu \not\in \Sigma$, and if $\lambda - \mu \in \Sigma$, then either $\lambda - \mu$ or $\mu - \lambda$ is an element of $\Sigma^{+} - \Sigma^{+}_{x}$.
Moreover, $2\lambda \not\in \Sigma$.
Thus, by a similar argument to the Proof of Lemma \ref{3-2-2}, the statement follows.
\end{proof}

By calculating the second fundamental form $h^{N}$ in a manner similar to the proof of Lemma \ref{3-2-4}, we obtain Lemma \ref{3-3-3}

\begin{lemm} \label{3-3-3}
$L$ is a totally geodesic submanifold of $M$ if and only if $\lambda(H) = i\pi/2$ for any $\lambda \in (\Sigma^{+}_{H})_{R}$.
\end{lemm}

\begin{lemm} \label{3-3-4}
$\lambda(H) =i \pi/2$ for any $\lambda \in (\Sigma^{+}_{H})_{R}$ if and only if $H = Q_{a}$ or $H = (1/2)(Q_{a} + Q_{b})$ for some $1 \leq a < b \leq n$.
\end{lemm}

\begin{proof}
If $H = Q_{a}$ or $H = (1/2)(Q_{a} + Q_{b})$ for some $1 \leq a < b \leq n$, then it is obvious that $\lambda(H) =i \pi/2$ for any $\lambda \in (\Sigma^{+}_{H})_{R}$.
Conversely, assume that $\lambda(H) =i \pi/2$ for any $\lambda \in (\Sigma^{+}_{H})_{R}$.
Let $H = (t_{1}, \cdots, t_{n})$ and assume $H \in Q_{ \{ \lambda_{i_{1}}, \cdots, \lambda_{i_{k}} \}}$ for $1 \leq i_{1} < \cdots < i_{k} \leq n$.
Suppose $k \geq 3$.
By the assumption, $2e_{i}(H) = i \pi/2 \ (i_{1} + 1 \leq i \leq i_{k})$, so $t_{i} = \pi/4$.
However, this contradicts to the condition that $t_{a} \not= t_{b} \ (a \in I_{l}, b \in I_{m}, 1 \leq l < m \leq n)$.
Thus, we conclude $k=1$ or $k = 2$.
If $k = 1$, the statement is trivial.
Assume $k=2$.
In this case, we obtain $t_{i} = \pi/4 \ (i_{1} + 1 \leq i \leq i_{2})$.
Since  $H \in Q_{ \{ \lambda_{i_{1}}, \lambda_{i_{2}} \}}$, it follows that $H = (1/2)(Q_{i_{1}} + Q_{i_{2}})$.
\end{proof}

Summarizing Lemma \ref{3-3-3} and Lemma \ref{3-3-4}, we obtain Proposition \ref{3-3-5}.

\begin{prop} \label{3-3-5}
$N$ is a proper ruled $CR$ submanifold if and only if $H = (1/2)(Q_{a} + Q_{b})$ for some $1 \leq a < b \leq n$.
\end{prop}

Next, we consider the integrability of the complex distribution.
If $H = Q_{k}$ for $1 \leq k \leq n$, then $N$ is a polar and a complex submanifold.
Thus, the integrability of the complex distribution is trivial.
In the case of $k \geq 2$, we can obtain Proposition \ref{3-3-6} using a similar argument to the proof of Proposition \ref{3-2-8}.

\begin{prop} \label{3-3-6}
Let $H \in Q_{\{ \lambda_{i_{1}}, \cdots, \lambda_{i_{k}} \}}$ with $k \geq 2$.
Then, the complex distribution $\mathcal{D}$ of $N$ is not integrable.
\end{prop}

Summarizing Lemma \ref{3-2-2}, Proposition \ref{3-2-5}, Proposition \ref{3-2-8}, Lemma \ref{3-3-2}, Proposition \ref{3-3-5}, and Proposition \ref{3-3-6}, we obtain the main result of this section

\begin{thm} \label{3-3-7}
Any orbit of the isotropy group action on an irreducible Hermitian symmetric space $M$ is a Hopf $CR$ submanifold.
Set $n = \mathrm{rank}M$.
Let $H \in Q$ and $x = \mathrm{exp}H(o)$.
Denote by $N$ the orbit through $x$.
Then, $N$ is a ruled $CR$ submanifold if and only if $H = Q_{a}, (1/2)Q_{a}, (1/2)(Q_{a} + Q_{b})$ for some $1 \leq a < b \leq n$.
Moreover, $N$ is a totally geodesic totally real submanifold if and only if $\Sigma$ is of type $C_{n}$ and $H = (1/2)Q_{n}$, and $N$ is a complex submanifold if and only if $H = Q_{a}$ for some $1 \leq a \leq n$, that is, $N$ is a polar.
The complex distribution of $N$ is integrable if and only if $N$ is a polar.
\end{thm}

%%%%%%%%%%%%%%%%%%%%%%%%%%%%%%%%%%%%%%%%%%%%%%%%%%%%%%%%%%%%%%%%%%%%%%%%%%%%%%%%%%%%%%%%%%%%%%%%%%%%%%%%%%%%%%%%%%%%%%%%%%%%%%%%%%%%%%%%%%%%%%%%%%%%%%%

\section{Contact $CR$ orbits} \label{s4}

In this section, we consider orbits of the isotropy group action such that the rank of the totally distribution $\mathcal{D}^{\perp}$ is $1$ and study whether such orbits are contact $CR$ submanifolds or Sasaki $CR$ submanifold.
Let $H \in Q$ and $x = \mathrm{exp}H(o)$.
Denote the $K$-orbit through $x$ by $N$.
Let $\Sigma$ be of type $BC_{n}\ (n \geq 1)$ or $C_{n}\ (n \geq 2)$.
Then, the rank of $\mathcal{D}^{\perp}$ is $1$ if and only if 
\[
\begin{split}
H = C_{i}(t) = (1-t)\ Q_{i} + t\ Q_{i+1} \quad (0 < t < 1, \ i= 0,1, \cdots, n-1),
\end{split}
\]
where $Q_{0} = 0$.
Then, $\mathcal{D}_{x}^{\perp}$ is obtained by $(\mathrm{exp}H)_{*} \frak{m}_{2e_{i+1}}$.

%%%%%%%%%%%%%%%%%%%%%%%%%%%%%%%%%%%%%%%%%%%%%%%%%%%%%%%%%%%%%%%%%%%%%%%%%%%%%%%%%%%%%%%%%%%%%%%%%%%%%%%%%%%%%%%%%%%%%%%%%%%%%%%%%%%%%%%%%%%%%%%%%%%%%%%

\subsection{Contact $CR$ orbits} \label{s4-1}

In this subsection, we study the condition under which $N$ is a contact $CR$ submanifold of $M$.

\begin{lemm}
Let $H = C_{i}(t)$ for $0 \leq i \leq n-1$ and $0 < t < 1$.
Define $\xi$ as follows:
\[
\xi = \frac{-2}{\sqrt{d}\sin t\pi}J^{*}.
\]
Then, $\xi$ is a $K$-invariant Killing vector filed with unit length on $N$ such that $\xi_{p} \in \mathcal{D}^{\perp}_{p}$ for any $p \in N$.
\end{lemm}

\begin{proof}
First,
\[
\begin{split}
J^{*}_{x} 
&=
(\mathrm{exp}H)_{*} \Big( \mathrm{Ad}(\mathrm{exp}H)^{-1}J \Big)_{\frak{m}} \\
&=
(\mathrm{exp}H)_{*} \Big( \mathrm{Ad}(\mathrm{exp}H)^{-1} \big( a_{1}\bar{S}^{(1)}_{\alpha_{1}} + \cdots + a_{n}\bar{S}^{(1)}_{\alpha_{n}} + C \big) \Big)_{\frak{m}} \\
&=
(\mathrm{exp}H)_{*} \Big(  \sum_{i=1}^{n} a_{i} \Big( \cos (2(-i)e_{1}(H)) \bar{S}^{(1)}_{\alpha_{i}} - \sin (2(-i)e_{i}(H)) \bar{T}^{(1)}_{\alpha_{i}} \Big) \Big)_{\frak{m}} \\
&=
(\mathrm{exp}H)_{*} \Big( - a_{i+1} \sin (2(-i)e_{i+1}(H)) \bar{T}^{(1)}_{\alpha_{i+1}}  \Big) \\
&=
(\mathrm{exp}H)_{*} \Big( - a_{i+1} \sin t\pi \bar{T}^{(1)}_{\alpha_{i+1}}  \Big). \\
\end{split}
\]
By Lemma \ref{2-4-3}, we see that $|\xi_{x}| = 1$.
Moreover, $\xi_{x} \in \mathcal{D}^{\perp}_{x}$.
Since $J$ is an element of the center of $K$, the vector field $\xi$ is invariant under $K$, that is, $\xi_{k(p)} = k_{*}\xi_{p}$ for any $k \in K$ and $p \in N$.
Therefore, the statement follows.
\end{proof}

Define a 1-form on $N$ as
\[
\eta(X) = \langle \xi_{p}, X \rangle \quad\quad(X \in T_{p}N, \ p \in N).
\]
Then, $\eta_{p} \not = 0$ for any $p \in N$.
Since $\xi$ is invariant under $K$, the 1-form $\eta$ is also invariant under $K$.
Next, we calculate $d \eta$.

\begin{lemm}
Let $X,Y \in \frak{k}$.
Then,
\[
(d \eta(X^{*}, Y^{*}))_{x} = \frac{3a_{i+1}\sin t \pi }{\sqrt{d}} \Big\langle [\bar{S}^{(1)}_{\alpha_{i+1}}, X], Y \Big\rangle. 
\]
\end{lemm}

\begin{proof}
Set $\bar{J} = (-2/\sqrt{d} \sin t \pi)J$. 
Then, $\xi = \bar{J}^{*}$ and $( \mathrm{Ad}(\mathrm{exp}H)^{-1}J )_{\frak{m}} = -a_{i+1}\sin t\pi \bar{T}^{(1)}_{\alpha_{i+1}}$ since $H = C_{i}(t) \ (0 < t < 1, \ 0 \leq i \leq n-1)$.
First,
\[
\begin{split}
& X^{*} \eta(Y^{*}) \\
&=
\frac{d}{dt} \langle \xi, Y^{*}_{\mathrm{exp}tX(x)} \rangle \\
&=
\frac{d}{dt} \Big\langle (\mathrm{exp}tX\mathrm{exp}H)_{*} \Big(\mathrm{Ad}(\mathrm{exp}tX\mathrm{exp}H)^{-1}\bar{J} \Big)_{\frak{m}}, (\mathrm{exp}tX\mathrm{exp}H)_{*} \Big(\mathrm{Ad}(\mathrm{exp}tX\mathrm{exp}H)^{-1}Y \Big)_{\frak{m}}  \Big\rangle \\
&=
\frac{d}{dt} \Big\langle \Big( \mathrm{Ad}(\mathrm{exp}H)^{-1} \mathrm{Ad}(\mathrm{exp}tX)^{-1} \bar{J} \Big)_{\frak{m}},  \Big( \mathrm{Ad}(\mathrm{exp}H)^{-1} \mathrm{Ad}(\mathrm{exp}tX)^{-1} Y \Big)_{\frak{m}} \Big\rangle \\
&=
- \Big\langle \Big( \mathrm{Ad}(\mathrm{exp}H)^{-1} [X, \bar{J}] \Big)_{\frak{m}}, \Big( \mathrm{Ad}(\mathrm{exp}H)^{-1}Y \Big)_{\frak{m}} \Big\rangle \\
& \hspace{60mm} - \Big\langle \Big( \mathrm{Ad}(\mathrm{exp}H)^{-1}\bar{J} \Big)_{\frak{m}}, \Big( \mathrm{Ad}(\mathrm{exp}H)^{-1}[X, Y] \Big)_{\frak{m}} \Big\rangle \\
&=
- \Big\langle \Big( \mathrm{Ad}(\mathrm{exp}H)^{-1}\bar{J} \Big)_{\frak{m}}, \Big( \mathrm{Ad}(\mathrm{exp}H)^{-1}[X, Y] \Big)_{\frak{m}} \Big\rangle \\
&=
 \frac{2}{\sqrt{d} \sin t \pi} \Big\langle \Big( \mathrm{Ad}(\mathrm{exp}H)^{-1}J \Big)_{\frak{m}}, \Big( \mathrm{Ad}(\mathrm{exp}H)^{-1}[X,Y] \Big)_{\frak{m}} \Big\rangle,
\end{split}
\]
where we use $[X,J] = 0$ since $X \in \frak{k}$ and $J$ is an element of the center of $\frak{k}$.
Moreover, 
\[
\begin{split}
\eta([X^{*}, Y^{*}]) 
&= 
\Big\langle \xi, [X^{*}, Y^{*}] \Big\rangle 
=
\Big\langle \xi, [X, Y]^{*} \Big\rangle \\
&=
\Big\langle (\mathrm{exp}H)_{*} \Big( \mathrm{Ad}(\mathrm{exp}H)^{-1}\bar{J} \Big)_{\frak{m}}, (\mathrm{exp}H)_{*} \Big( \mathrm{Ad}(\mathrm{exp}H)^{-1}[X,Y] \Big)_{\frak{m}} \Big\rangle \\
&=
 -\frac{2}{\sqrt{d} \sin t \pi} \Big\langle \Big( \mathrm{Ad}(\mathrm{exp}H)^{-1}J \Big)_{\frak{m}}, \Big( \mathrm{Ad}(\mathrm{exp}H)^{-1}[X,Y] \Big)_{\frak{m}} \Big\rangle.
\end{split}
\]
Therefore, 
\[
\begin{split}
d \eta (X^{*}, Y^{*}) 
&=
\frac{1}{2} \Big( X^{*} \eta(Y^{*}) - Y^{*} \eta(X^{*}) - \eta([X^{*}, Y^{*}]) \Big) \\
&=
\frac{3}{\sqrt{d} \sin t \pi} \Big\langle \Big( \mathrm{Ad}(\mathrm{exp}H)^{-1}J \Big)_{\frak{m}}, \Big( \mathrm{Ad}(\mathrm{exp}H)^{-1}[X,Y] \Big)_{\frak{m}} \Big\rangle \\
&=
\frac{3}{\sqrt{d} \sin t \pi} \Big\langle -a_{i+1} \sin t \pi \bar{T}^{(1)}_{\alpha_{i+1}}, \mathrm{Ad}(\mathrm{exp}H)^{-1}[X,Y] \Big\rangle \\
&=
\frac{-3a_{i+1}}{\sqrt{d}} \Big\langle \mathrm{Ad}(\mathrm{exp}H)\bar{T}^{(1)}_{\alpha_{i+1}}, [X,Y] \Big\rangle \\
&=
\frac{-3a_{i+1}}{\sqrt{d}} \Big\langle \cos t \pi \bar{T}^{(1)}_{\alpha_{i+1}} - \sin t \pi \bar{S}^{(1)}_{\alpha_{i+1}}, [X,Y] \Big\rangle \\
&=
\frac{3a_{i+1}\sin t \pi}{\sqrt{d}} \Big\langle [\bar{S}^{(1)}_{\alpha_{i+1}}, X], Y \Big\rangle, \\
\end{split}
\]
where we use $\bar{T}^{(1)}_{\alpha_{i+1}} \in \frak{m}$ and $[X,Y] \in \frak{k}$.
\end{proof}

If $\Sigma$ is of type $BC_{n}\ (n \geq 2)$ and $H = C_{i}(t) \ (1 \leq i \leq n-1, 0 < t < 1)$, then $d \eta$ is degenerate on $\mathcal{D}$.
For example, let $\alpha \in \tilde{\Sigma}$ be $\bar{\alpha} = e_{1}$.
Note that $e_{1} \in (\Sigma^{+}_{x})_{C}$.
Then, $(S^{(1)}_{\alpha})^{*}_{x} \not= 0$, but $[S^{(1)}_{\alpha_{i+1}}, S^{(1)}_{\alpha}] = 0$ because $\overline{\alpha_{i+1} \pm \alpha} = e_{1} \pm 2e_{i+1} \not\in \Sigma$.
Hence, for any $X \in \frak{k}$,
\[
d\eta( (S^{(1)}_{\alpha})^{*}_{x}, X^{*}) = \frac{3a_{i+1}\sin \pi t}{\sqrt{d}} \Big\langle [\bar{S}^{(1)}_{\alpha_{i+1}}, S^{(1)}_{\alpha}], X \Big\rangle = 0. \\
\]
If $\Sigma$ is of type $C_{n} \ (n \geq 2)$ and $H = C_{i}(t) \ (1 \leq i \leq n-2, 0 < t < 1)$, then $d \eta$ is also degenerate on $\mathcal{D}$.
For example, let $\beta \in \tilde{\Sigma}$ be $\bar{\beta} = e_{a} + e_{b} \ (1 \leq a \leq i , i+2 \leq b \leq n)$.
Note $e_{a} + e_{b} \in (\Sigma^{+}_{x})_{C}$.
Then, $[S^{(1)}_{\alpha_{i+1}}, S^{(1)}_{\beta}] = 0$.
In both cases, $\eta \wedge (d\eta)^{(1/2)(\dim N - 1)} = 0$ for any $p \in N$.
Thus, $N$ is not a contact $CR$ submanifold because, for any $1$-form $\omega$ on $N$ such that $\omega|_{\mathcal{D}} = 0$, there exists some function $f$ of $N$ such that $\omega = f\eta$, since the rank of $\mathcal{D}^{\perp}$ is 1.

Let $\Sigma$ be of type $BC_{n} (n \geq 1)$, and $H = C_{0}(t)\ (0 < t < 1)$.
Then,
\[
\begin{split}
\Sigma^{+}_{x} &= \{ 2e_{1} \} \cup \{ e_{1} \pm e_{k} \ ;\ 2 \leq k \leq n \} \cup \{ e_{1} \}, \\
(\Sigma^{+}_{x})_{C} &= \{ e_{1} \pm e_{k} \ ;\ 2 \leq k \leq n \} \cup \{ e_{1} \}, \\
(\Sigma^{+}_{x})_{R} &= \{ 2e_{1} \}.
\end{split}
\]
If $\alpha \in \tilde{\Sigma}$ satisfies $\bar{\alpha} \in (\Sigma^{+}_{x})_{C}$, then $\alpha_{1} - \alpha, \alpha_{1} - \sigma(\alpha) \in \tilde{\Sigma}$.
Therefore, $[S^{(1)}_{\alpha_{1}}, S^{(i)}_{\alpha}] = 0$ if and only if $S^{(i)}_{\alpha} = 0$.
Moreover, 
\[
[S^{(1)}_{\alpha_{1}}, S^{(i)}_{\alpha}] \in \sum_{\lambda \in (\Sigma^{+}_{x})_{C}}\frak{k}_{\lambda}.
\]
Hence, $d \eta$ is non-degenerate on $\mathcal{D}$, and $\eta \wedge (d \eta)^{(1/2)(\dim N -1)} \not= 0$ for any point of $N$.
Thus, $\eta$ is a contact form on $N$, and $N$ is a contact $CR$ submanifold of $M$.

Let $\Sigma$ be of type $C_{n}\ (n \geq 2)$, and $H = C_{0}(t)\ (0 < t < 1)$.
Then,
\[
\begin{split}
\Sigma^{+}_{x} &= \{ 2e_{1} \} \cup \{ e_{1} \pm e_{k} \ ;\ 2 \leq k \leq n \}, \\
(\Sigma^{+}_{x})_{C} &= \{ e_{1} \pm e_{k} \ ;\ 2 \leq k \leq n \}, \\
(\Sigma^{+}_{x})_{R} &= \{ 2e_{1} \}.
\end{split}
\]
and we can easily check that $\eta$ is a contact form on $N$ and that $N$ is a contact $CR$ submanifold of $M$.
If $H = C_{n-1}(t)$, then $N$ is also a contact $CR$ submanifold with a contact form $\eta$, because it is known that there exists some holomorphic isometry $g$ of $M$ such that $g(K(\mathrm{exp}C_{0}(t))(o)) = K(\mathrm{exp}C_{n-1}(1-t))(o)$ \cite{Chen2}.
Summarizing the above arguments, we obtain Theorem \ref{4-1-1}.

\begin{thm} \label{4-1-1}
Let $H \in Q$ and $x = \mathrm{exp}H(o)$.
The $K$-orbit through $x$ is denoted by $N$.
Then, $N$ is a contact $CR$ submanifold of $M$ if and only if either $\Sigma$ is of type $BC_{n}\ (n \geq 1)$ and $H = C_{0}\ (0 < t < 1)$, or $\Sigma$ is of type $C_{n}\ (n \geq 2)$ and $H = C_{0}(t)$ or $H = C_{n}(1-t) \ ( 0 < t < 1)$.
\end{thm}

%%%%%%%%%%%%%%%%%%%%%%%%%%%%%%%%%%%%%%%%%%%%%%%%%%%%%%%%%%%%%%%%%%%%%%%%%%%%%%%%%%%%%%%%%%%%%%%%%%%%%%%%%%%%%%%%%%%%%%%%%%%%%%%%%%%%%%%%%%%%%%%%%%%%%%%

\subsection{Some preliminaries I} \label{s4-2}

In this subsection, we provide some preliminaries that will be used in the later subsections.
Let $H = C_{0}(t)\ (0 < t < 1)$ and $x = \mathrm{exp}H(o)$.
The $K$-orbit through $x$ is denoted by $N$.
Assume that $\Sigma$ is of type $BC_{n} \ (n \geq 1)$ or $C_{n}\ (n \geq 2)$.
First, Lemma \ref{5-2-1} is immediately obtained by the general theory of root systems.

\begin{lemm} \label{5-2-1}
Let $\alpha \in \tilde{\Sigma}$ be a longest root and $\beta \in \tilde{\Sigma}$ satisfy $2(\alpha, \beta)/(\alpha, \alpha) = \pm 1$.
Then,
\[
\begin{split}
\Big[ U_{\alpha}, \big[ U_{\alpha}, U_{\beta} \big] \Big] = - U_{\beta}, \quad & \Big[ U_{\alpha}, \big[ U_{\alpha}, W_{\beta} \big] \Big] = -W_{\beta}, \\
\Big[ W_{\alpha}, \big[ W_{\alpha}, U_{\beta} \big] \Big] = - U_{\beta}, \quad & \Big[ W_{\alpha}, \big[ W_{\alpha}, W_{\beta} \big] \Big] = -W_{\beta}.
\end{split}
\]
\end{lemm}

By Lemma \ref{5-2-1}, we immediately obtain Lemma \ref{4-2-1}.

\begin{lemm} \label{4-2-1}
Let $\gamma \in \tilde{\Sigma}$ satisfy $\bar{\gamma} \in (\Sigma^{+}_{x})_{C}$.
Then,
\[
\Big[ [ T^{(i)}_{\gamma}, \ T^{(1)}_{\alpha_{1}}], \ T^{(1)}_{\alpha_{1}} \Big] = -T^{(i)}_{\gamma}.
\]
\end{lemm}

\begin{lemm} \label{4-2-2}
Let $\gamma \in \tilde{\Sigma}$ satisfy $\bar{\gamma} \in (\Sigma^{+}_{x})_{C}$.
Then,
\[
\begin{array}{ll}
\Big[ [ T^{(1)}_{\gamma}, \ T^{(1)}_{\alpha_{1}}], \ T^{(2)}_{\gamma} \Big]_{\frak{m}_{x}} = 0, \quad & \Big[ [ T^{(2)}_{\gamma}, \ T^{(1)}_{\alpha_{1}}], \ T^{(1)}_{\gamma} \Big]_{\frak{m}_{x}} = 0.
\end{array}
\]
Let $c \in \mathbb{R}$ such that $c|\gamma|^{2} = |\alpha_{1}|^{2}$.
If $\alpha \in \tilde{\Sigma}_{(1)}$, then
\[
\begin{array}{ll}
\Big[ [ T^{(1)}_{\gamma}, \ T^{(1)}_{\alpha_{1}}], \ T^{(1)}_{\gamma} \Big]_{\frak{m}_{x}} = c\:T^{(1)}_{\alpha_{1}}.
\end{array}
\]
If $\gamma \in \tilde{\Sigma}_{(2)}$, then for any $i = 1,2$,
\[
\begin{array}{ll}
\Big[ [ T^{(i)}_{\gamma}, \ T^{(1)}_{\alpha_{1}}], \ T^{(i)}_{\gamma} \Big]_{\frak{m}_{x}} = 2c\:T^{(1)}_{\alpha_{1}}.
\end{array}
\]
If $\gamma \in \tilde{\Sigma}_{(3)}$, then for any $i= 1,2$,
\[
\begin{array}{ll}
\Big[ [ T^{(i)}_{\gamma}, \ T^{(1)}_{\alpha_{1}}], \ T^{(i)}_{\gamma} \Big]_{\frak{m}_{x}} = 4c\:T^{(1)}_{\alpha_{1}} \quad (i =1,2).
\end{array}
\]
\end{lemm}

\begin{proof}
We observe that $\alpha_{1} + \gamma, \sigma(\gamma) + \alpha \not\in \tilde{\Sigma}$ and $\alpha_{1} - \gamma, \alpha_{1} - \sigma(\gamma) \in \tilde{\Sigma}$.
By Lemma \ref{2-3-2}, we obtain $N_{\gamma - \alpha_{1}, -\gamma} = cN_{-\gamma, \alpha_{1}} = c\bar{N}_{\gamma, -\alpha_{1}}$ and $\epsilon_{\gamma - \alpha_{1}, -\gamma} = \epsilon_{-\gamma, \alpha_{1}} = \epsilon_{\gamma, -\alpha_{1}}$.
Assume $\epsilon_{\gamma, -\alpha_{1}} = 1$.
If $\gamma \in \tilde{\Sigma}_{(1)}$, then either $(\gamma - \alpha_{1}) + \gamma \not \in \tilde{\Sigma}$ or $(\gamma - \alpha_{1}) + \gamma = \alpha_{k}$ for some $2 \leq k \leq n$.
Thus,
\[
\begin{split}
\Big[ [ T^{(1)}_{\gamma}, T^{(1)}_{\alpha_{1}} ], T^{(1)}_{\gamma} \Big]_{\frak{m}_{x}}
&= 
\Big[ [W_{\gamma}, W_{\alpha_{1}}], W_{\gamma} \Big]_{\frak{m}_{x}}
= N_{\gamma, -\alpha_{1}} \Big[ U_{\gamma - \alpha_{1}}, W_{\gamma} \Big]_{\frak{m}_{x}} \\
& =  N_{\gamma, -\alpha_{1}}(-N_{\gamma - \alpha_{1}, -\gamma})W_{-\alpha_{1}} 
= c N_{\gamma, -\alpha_{1}}\bar{N}_{\gamma, -\alpha_{1}}W_{\alpha_{1}} =
c\:T^{(1)}_{\alpha_{1}}. \\
\end{split}
\]
If $\gamma \in \tilde{\Sigma}_{(2)}$, then either $(\gamma - \alpha_{1}) + \gamma \not \in \tilde{\Sigma}$ or $(\gamma - \alpha_{1}) + \gamma = \alpha_{k}$ for some $2 \leq k \leq n$.
Similarly, either $(\gamma - \alpha_{1}) + \sigma(\gamma) \not\in \tilde{\Sigma}$ or $(\gamma - \alpha_{1}) + \sigma(\gamma) = \alpha_{k}$ for some $2 \leq k \leq n$.
Moreover, $\gamma - \alpha_{1} - \sigma(\gamma) \not\in \tilde{\Sigma}$.
Assume for contradiction that $\gamma - \alpha_{1} - \sigma(\gamma) \in \tilde{\Sigma}$.
Then, $\overline{\gamma - \alpha_{1} - \sigma(\gamma)} = -\overline{\alpha_{1}}$, and $\gamma - \alpha_{1} - \sigma(\gamma) = -\alpha_{1}$ since $m(2e_{1}) = 1$.
Thus, $\gamma = \sigma(\gamma)$, but this contradicts to the assumption that $\gamma \in \tilde{\Sigma}_{(2)}$.
Hence,
\[
\begin{split}
\Big[ [ T^{(1)}_{\gamma}, T^{(1)}_{\alpha_{1}} ], T^{(1)}_{\gamma} \Big]_{\frak{m}_{x}}
&= \Big[ [W_{\gamma} + W_{\sigma(\gamma)}, \:W_{\alpha_{1}}], \: W_{\gamma} + W_{\sigma(\gamma)} \Big]_{\frak{m}_{x}} \\
&= \Big[ N_{\gamma, -\alpha_{1}}U_{\gamma - \alpha_{1}} + N_{\sigma(\gamma) - \alpha_{1}}U_{\sigma(\gamma) - \alpha_{1}}, \: W_{\gamma} + W_{\sigma(\gamma)} \Big]_{\frak{m}_{x}} \\
&= N_{\gamma , -\alpha_{1}} \Big[ U_{\gamma - \alpha_{1}}, W_{\gamma} \Big]_{\frak{m}_{x}} + N_{\sigma(\gamma), -\alpha_{1}} \Big[ U_{\sigma(\gamma) - \alpha_{1}}, W_{\sigma(\gamma)} \Big] _{\frak{m}_{x}} \\
&= N_{\gamma, -\alpha_{1}}(-N_{\gamma - \alpha_{1}, -\gamma})W_{-\alpha_{1}} + N_{\sigma(\gamma) - \alpha_{1}}(-N_{\sigma(\gamma) - \alpha_{1}, -\sigma(\gamma)})W_{-\alpha_{1}} \\
&= cN_{\gamma, -\alpha_{1}} \bar{N}_{\gamma, -\alpha_{1}}W_{\alpha_{1}} + cN_{\sigma(\gamma), -\alpha_{1}} \bar{N}_{\sigma(\gamma), -\alpha_{1}}W_{\alpha_{1}} \\
&= 2cW_{\alpha_{1}} = 2c\:T^{(1)}_{\alpha_{1}}, \\
\end{split}
\]
\[
\begin{split}
\Big[ [T^{(2)}_{\gamma}, T^{(1)}_{\alpha_{1}}], T^{(2)}_{\gamma} \Big]_{\frak{m}_{x}}
&= \Big[ [-U_{\gamma} + U_{\sigma(\gamma)}, \: W_{\alpha_{1}}], \: -U_{\gamma} + U_{\sigma(\gamma)} \Big]_{\frak{m}_{x}} \\
&= \Big[ N_{\gamma, -\alpha_{1}}W_{\gamma - \alpha_{1}} - N_{\sigma(\gamma), - \alpha_{1}}W_{\sigma(\gamma) - \alpha_{1}}, \: - U_{\gamma} + U_{\sigma(\gamma)} \Big]_{\frak{m}_{x}} \\ 
&= -N_{\gamma, \alpha_{1}} \Big[ W_{\gamma - \alpha_{1}}, U_{\gamma} \Big]_{\frak{m}_{x}} - N_{\sigma(\gamma), -\alpha_{1}} \Big[ W_{\sigma(\gamma) - \alpha_{1}}, U_{\sigma(\gamma)} \Big]_{\frak{m}_{x}} \\
&= (-N_{\gamma, -\alpha_{1}})\bar{N}_{\gamma - \alpha_{1}, -\alpha_{1}} W_{-\alpha_{1}} - N_{\sigma(\gamma), -\alpha_{1}}N_{\sigma(\gamma) -\alpha_{1}, -\sigma(\alpha_{1})}W_{-\alpha_{1}} \\
&= cN_{\gamma, -\alpha_{1}}\bar{N}_{\gamma, -\alpha_{1}}W_{\alpha_{1}} + cN_{\sigma(\gamma), -\alpha_{1}}\bar{N}_{\sigma(\gamma), -\alpha_{1}}W_{\alpha_{1}} \\
&= 2cW_{\alpha_{1}} = 2c\:T^{(1)}_{\alpha_{1}}, \\
\end{split}
\]
\[
\begin{split}
\Big[ [T^{(1)}_{\gamma}, T^{(1)}_{\alpha_{1}}], T^{(2)}_{\gamma} \Big]_{\frak{m}_{x}}
&= \Big[ [ W_{\gamma} + W_{\sigma(\gamma)}, W_{\alpha_{1}} ], -U_{\gamma} + U_{\sigma(\gamma)} \Big]_{\frak{m}_{x}} \\ 
&= \Big[ N_{\gamma, -\alpha_{1}}U_{\gamma - \alpha_{1}} + N_{\sigma(\gamma), -\alpha_{1}}U_{\sigma(\gamma) - \alpha_{1}}, \: -U_{\gamma} + U_{\sigma(\gamma)} \Big]_{\frak{m}_{x}} \\ 
&= -N_{\gamma -\alpha_{1}} \Big[ U_{\gamma - \alpha_{1}}, U_{\gamma} \Big]_{\frak{m}_{x}} + N_{\sigma(\gamma), \sigma(\gamma) - \alpha_{1}} \Big[ U_{\sigma(\gamma) - \alpha_{1}}, U_{\sigma(\gamma)} \Big]_{\frak{m}_{x}} \\
&= -N_{\gamma, -\alpha_{1}}N_{\gamma - \alpha_{1}, -\gamma}U_{-\alpha_{1}} + N_{\sigma(\gamma), -\alpha_{1}}N_{\sigma(\gamma) - \alpha_{1}, -\sigma(\gamma)}U_{-\alpha_{1}} \\
&= -cN_{\gamma, -\alpha_{1}}\bar{N}_{\gamma, -\alpha_{1}}U_{\alpha_{1}} + cN_{\sigma(\gamma), -\alpha_{1}}\bar{N}_{\sigma(\gamma), -\alpha_{1}}U_{\alpha_{1}} \\
&= -cU_{\alpha_{1}} + cU_{\alpha_{1}} = 0, \\
\end{split}
\]
\[
\begin{split}
\Big[ [T^{(2)}_{\gamma}, T^{(1)}_{\alpha_{1}}], T^{(1)}_{\gamma} \Big]_{\frak{m}_{x}}
& = \Big[ [-U_{\gamma} + U_{\sigma(\gamma)}, W_{\alpha_{1}}], W_{\gamma} + W_{\sigma(\gamma)} \Big]_{\frak{m}_{x}} \\
&= \Big[ N_{\gamma, -\alpha_{1}}W_{\gamma - \alpha_{1}} - N_{\sigma(\gamma), -\alpha_{1}}W_{\sigma(\gamma) - \alpha_{1}}, \: W_{\gamma} + W_{\sigma(\gamma)} \Big]_{\frak{m}_{x}} \\
&= N_{\gamma, -\alpha_{1}} \Big[ W_{\gamma - \alpha_{1}}, W_{\gamma} \Big]_{\frak{m}_{x}} - N_{\sigma(\gamma), -\alpha_{1}} \Big[ W_{\sigma(\gamma) - \alpha_{1}}, W_{\sigma(\gamma)} \Big]_{\frak{m}_{x}} \\
&= N_{\gamma, -\alpha_{1}}N_{\gamma - \alpha_{1}, -\gamma}U_{-\alpha_{1}} - N_{\sigma(\gamma), -\alpha_{1}}N_{\sigma(\gamma) - \alpha_{1}, -\sigma(\gamma)}U_{-\alpha_{1}} \\
&= c N_{\gamma, -\alpha_{1}}\bar{N}_{\gamma, -\alpha_{1}}U_{\alpha_{1}} - c N_{\sigma(\gamma), -\alpha_{1}}\bar{N}_{\sigma(\gamma), -\alpha_{1}}U_{\alpha_{1}} \\
&= cU_{\alpha_{1}} - cU_{\alpha_{1}} = 0.
\end{split}
\]
If $\gamma \in \tilde{\Sigma}_{(3)}$, then either $(\gamma - \alpha_{1}) + \gamma \not\in \tilde{\Sigma}$ or $(\gamma - \alpha_{1}) + \gamma \in \tilde{\Sigma}_{0}$.
Similarly, either $(\gamma - \alpha_{1}) + \sigma(\gamma) \not\in \tilde{\Sigma}$ or $(\gamma - \alpha_{1}) + \sigma(\gamma) \in \tilde{\Sigma}_{0}$.
By using a similar argument as in the case of $\gamma \in \tilde{\Sigma}_{(2)}$, we conclude that $\gamma - \alpha_{1} - \sigma(\gamma) \not\in \tilde{\Sigma}$.
Hence, 
\[
\begin{array}{lllll}
\pi_{1} \Big( \Big[ [ T^{(1)}_{\gamma}, T^{(1)}_{\alpha_{1}} ], T^{(1)}_{\gamma} \Big] \Big) 
= 4c\:T_{\alpha_{1}}, &
& \pi_{1} \Big( \Big[ [ T^{(2)}_{\gamma}, T^{(1)}_{\alpha_{1}} ], T^{(2)}_{\gamma} \Big] \Big) \vspace{2mm}
= 4c\:T^{(1)}_{\alpha_{1}}, \\
\pi_{1} \Big( \Big[ [T^{(1)}_{\gamma}, T^{(1)}_{\alpha_{1}}], T^{(2)}_{\gamma} \Big] \Big)
= 0,  &
& \pi_{1} \Big( \Big[ [T^{(2)}_{\lambda}, T^{(1)}_{\alpha_{1}}], T^{(1)}_{\lambda} \Big] \Big)
= 0 \\
\end{array}
\]
by similar calculations.
In the case of $\epsilon_{\gamma, -\alpha_{1}} = -1$, we can verify that the statements are also true.
\end{proof}

Let $\gamma, \delta \in \tilde{\Sigma}$.
If $\bar{\gamma}, \bar{\delta} \in (\Sigma^{+}_{x})_{C}$ and $\bar{\gamma} \not= \bar{\delta}$, then by Lemma \ref{2-3-4}, for any $i,j = 1,2$,
\[
\Big[ [ T^{(i)}_{\gamma}, T^{(1)}_{\alpha_{1}}], T^{(j)}_{\delta} \Big]_{\frak{m}_{x}} = 0.
\]

\begin{lemm} \label{4-2-3}
Let $\gamma, \delta \in \tilde{\Sigma}$.
If $\bar{\gamma} = \bar{\delta} = e_{1}$, then $[[ T^{(i)}_{\gamma}, T^{(1)}_{\alpha_{1}} ], T^{(j)}_{\delta} ] \in \frak{a}$ for $i,j = 1,2$.
If $\bar{\gamma} = \bar{\delta} = e_{1} \pm e_{k} \ (2 \leq k \leq n)$, then $\Big[ [ T^{(i)}_{\gamma}, T^{(1)}_{\alpha_{1}} ], T^{(j)}_{\delta} \Big] \in \frak{m}_{2e_{k}}$ for $i,j = 1,2$.
In particular, $[ [ T^{(i)}_{\gamma}, T^{(1)}_{\alpha_{1}} ], T^{(j)}_{\delta}]_{\frak{m}_{x}} = 0$ in both cases.
\end{lemm}

\begin{proof}
We consider the former part of the statement.
By Lemma \ref{2-3-4}, we have $[[ T^{(i)}_{\gamma}, T^{(1)}_{\alpha_{1}}], T^{(j)}_{\delta}] \in \frak{m}_{2e_{1}} + \frak{a}$.
Since 
\[
\begin{split}
\Big\langle [[ T^{(i)}_{\gamma}, T^{(1)}_{\alpha_{1}}], T^{(j)}_{\delta}], T^{(1)}_{\alpha_{1}} \Big\rangle
&= \Big\langle [ T^{(i)}_{\gamma}, T^{(1)}_{\alpha_{1}}], [T^{(j)}_{\delta}, T^{(1)}_{\alpha_{1}}] \Big\rangle \\
&= - \Big\langle [[ T^{(i)}_{\gamma}, T^{(1)}_{\alpha_{1}}], T^{(1)}_{\alpha_{1}}], T^{(j)}_{\delta} \Big\rangle \\
&= \Big\langle T^{(i)}_{\gamma}, T^{(j)}_{\delta} \Big\rangle = 0
\end{split}
\]
and $m(2e_{1}) = 1$, the $\frak{m}_{2e_{1}}$-part of $[[ T^{(i)}_{\gamma}, T^{(1)}_{\alpha}], T^{(j)}_{\delta}]$ is $0$ and $[[ T^{(i)}_{\gamma}, T^{(1)}_{\alpha_{1}} ], T^{(j)}_{\delta} ] \in \frak{a}$.
By a similar argument, we can show that the latter part of the statement is also true.
\end{proof}

Let $\alpha \in \tilde{\Sigma}$ be $\bar{\alpha} \in \Sigma^{+}_{x}$.
Set $^N (S^{(i)}_{\alpha})^{*}\ (i =1,2)$ such that
\[
^N(S^{(i)}_{\alpha})^{*} = \frac{-1}{\sin (-i)\lambda(H) }(\bar{S}^{(i)}_{\alpha})_{*} \quad (i =1,2).
\]
Then, $^N (S^{(i)}_{\alpha})^{*}_{x} = (\mathrm{exp}H)_{*}\bar{T}^{(i)}_{\alpha}$.
Set an orthonormal basis $P$ of $T_{x}N$ as follows:
\[
\begin{split}
P & = \left\{ {}^{N}(S^{(1)}_{\beta})^{*}_{x}, {}^{N}(S^{(2)}_{\beta})^{*}_{x} \ ;\ \bar{\beta} \in (\Sigma^{+}_{x})_{C}, \sigma(\beta) \not= \beta \right\} \\
& \quad\quad\quad\quad\quad\quad\quad\quad \cup \left\{ {}^{N}(S^{(1)}_{\gamma})^{*}_{x} \ ;\ \bar{\gamma} \in (\Sigma^{+}_{x})_{C}, \sigma(\gamma) = \gamma \right\} \cup \{ {}^{N}(S^{(1)}_{\alpha_{1}})^{*}_{x} \}.
\end{split}
\]

\begin{lemm} \label{4-2-4}
Let $h^{N}$ be the second fundamental form of $N \subset M$.
Then,
\[
h^{N} \Big( {}^{N}(S^{(1)}_{\alpha_{1}})^{*}_{x}, {}^{N}(S^{(1)}_{\alpha_{1}})^{*}_{x} \Big) = \cot t \pi (\mathrm{exp}H)_{*} (iH_{2e_{1}}).
\]
If $\beta \in \tilde{\Sigma}$ satisfies $\bar{\beta} \in (\Sigma^{+}_{x})_{C}$, then for any $1 \leq i \not= j \leq 2$,
\[
\begin{split}
& h^{N} \Big( {}^{N}(S^{(1)}_{\alpha_{1}})^{*}_{x}, {}^{N}(S^{(i)}_{\beta})^{*}_{x} \Big) = 0, \quad h^{N} \Big( {}^{N}(S^{(i)}_{\beta})^{*}_{x}, {}^{N}(S^{(j)}_{\beta})^{*}_{x} \Big) = 0. \\
\end{split}
\]
Moreover, if $S^{(i)}_{\beta}$ is not $0$, then 
\[
h^{N} \Big( {}^{N}(S^{(i)}_{\beta})^{*}_{x}, {}^{N}(S^{(i)}_{\beta})^{*}_{x} \Big) = \cot \frac{t\pi}{2} (\mathrm{exp}H)_{*}(iH_{\bar{\beta}}).
\]
If $\gamma \in \tilde{\Sigma}$ satisfies $\bar{\gamma} = \bar{\beta}$ and $\gamma \not= \beta, \sigma(\beta)$, then for any $i,j = 1,2$, 
\[
h^{N} \Big( {}^{N}(S^{(i)}_{\beta})^{*}_{x}, {}^{N}(S^{(j)}_{\gamma})^{*}_{x} \Big) = 0.
\]
If $\delta \in \tilde{\Sigma}$ satisfies $\bar{\delta} \not= \bar{\beta}$, then for any $i,j = 1,2$,
\[
h^{N} \Big( {}^{N}(S^{(i)}_{\beta})^{*}_{x}, {}^{N}(S^{(j)}_{\delta})^{*}_{x} \Big) \in (\mathrm{exp}H)_{*} \frak{m}_{\overline{\beta - \delta}}.
\]
\end{lemm}

\begin{proof}
By Lemma \ref{2-3-5}, 
\[
\begin{split}
h^{N}\Big( {}^{N}(S^{(1)}_{\alpha_{1}})^{*}_{x}, {}^{N}(S^{(1)}_{\alpha_{1}})^{*}_{x} \Big) 
& = \Big( (\mathrm{exp}H)_{*} [\bar{T}^{(1)}_{\alpha_{1}}, -\frac{\cos t\pi}{\sin t\pi}\bar{S}^{(1)}_{\alpha_{1}} ] \Big)_{T^{\perp}_{x}N} \\
& = \cot t \pi (\mathrm{exp}H)_{*} (iH_{2e_{1}}), 
\\
h^{N}\Big( {}^{N}(S^{(i)}_{\beta})^{*}_{x}, {}^{N}(S^{(i)}_{\beta})^{*}_{x} \Big) 
&= \Big( (\mathrm{exp}H)_{*} [\bar{T}^{(i)}_{\beta}, -\frac{\cos t(\pi/2)}{\sin t(\pi/2)}\bar{S}^{(i)}_{\beta} ] \Big)_{T^{\perp}_{x}N} \\
& = \cot t(\pi/2) (\mathrm{exp}H_{*}) (iH_{\bar{\beta}}).
\\
\end{split}
\]
Note that if $2\bar{\beta} \in \Sigma$, then $2\bar{\beta} = 2e_{1}$.
By a similar argument to the proof of Lemma \ref{3-2-7}, 
\[
\begin{split}
h^{N}\Big( {}^{N}(S^{(i)}_{\beta})^{*}_{x}, {}^{N}(S^{(j)}_{\beta})^{*}_{x} \Big) 
&= \Big( (\mathrm{exp}H)_{*} [\bar{T}^{(i)}_{\beta}, -\frac{\cos t(\pi/2)}{\sin t(\pi/2)}\bar{S}^{(j)}_{\beta} ] \Big)_{T^{\perp}_{x}N}
= 0, \\
h^{N}\Big( {}^{N}(S^{(i)}_{\beta})^{*}_{x}, {}^{N}(S^{(j)}_{\gamma})^{*}_{x} \Big) 
&= \Big( (\mathrm{exp}H)_{*} [\bar{T}^{(i)}_{\beta}, -\frac{\cos t(\pi/2)}{\sin t(\pi/2)}\bar{S}^{(j)}_{\gamma} ] \Big)_{T^{\perp}_{x}N}
=0.
\end{split}
\]
Since $\overline{\alpha_{1} - \beta} \in (\Sigma^{+}_{x})_{C}$ and either $\beta + \delta \not\in \tilde{\Sigma}$ or $\beta + \delta = \alpha_{1}$, by Lemma \ref{2-3-4},
\[
\begin{split}
& h^{N}\Big( {}^{N}(S^{(1)}_{\alpha_{1}})^{*}_{x}, {}^{N}(S^{(i)}_{\beta})^{*}_{x} \Big) 
= \Big( (\mathrm{exp}H)_{*} [\bar{T}^{(1)}_{\alpha_{1}}, -\frac{\cos t(\pi/2)}{\sin t(\pi/2)}\bar{S}^{(i)}_{\beta} ] \Big)_{T^{\perp}_{x}N} = 0, \\
& h^{N}\Big( {}^{N}(S^{(i)}_{\beta})^{*}_{x}, {}^{N}(S^{(j)}_{\delta})^{*}_{x} \Big) 
= \Big( (\mathrm{exp}H)_{*} [\bar{T}^{(i)}_{\beta}, -\frac{\cos t(\pi/2)}{\sin t(\pi/2)}\bar{S}^{(j)}_{\delta} ] \Big)_{T^{\perp}_{x}N} 
\in (\mathrm{exp}H)_{*} \frak{m}_{\overline{\beta- \delta}}.
\end{split}
\]
\end{proof}

\begin{lemm} \label{4-2-5}
Let $\beta \in \tilde{\Sigma}$ be $\bar{\beta} \in (\Sigma^{+}_{x})_{C}$.
Denote by $R^{M}$ the curvature tensor of $M$.
If $S^{(i)}_{\beta}\ (i = 1,2)$ is not $0$, then
\[
\begin{split}
\Big( R^{M} \Big( {}^{N}(S^{(i)}_{\beta})^{*}_{x}, {}^{N}(S^{(1)}_{\alpha_{1}})^{*}_{x} \Big) {}^{N}(S^{(1)}_{\alpha_{1}})^{*}_{x} \Big)_{T_{x}N} &= \left( \frac{1}{d} \right){}^{N}(S^{(i)}_{\beta})^{*}_{x}, \\
\Big( R^{M} \Big( {}^{N}(S^{(i)}_{\beta})^{*}_{x}, {}^{N}(S^{(1)}_{\alpha_{1}})^{*}_{x} \Big) {}^{N}(S^{(i)}_{\beta})^{*}_{x} \Big)_{T_{x}N} 
& = - \left( \frac{1}{d} \right) {}^{N}(S^{(1)}_{\alpha_{1}})^{*}_{x}. \\
\end{split}
\]
\end{lemm}

\begin{proof}
Let $\omega_{\alpha} = (\bar{\alpha}, \bar{\alpha})^{\frac{1}{2}}/2$ for any $\alpha \in \tilde{\Sigma}$ such that $\bar{\alpha} \not= 0$.
Then, $\bar{T}^{(i)}_{\beta} = \omega_{\beta} T^{(i)}_{\beta}$.
Note that $\omega_{\alpha_{1}} = 1/d$, $\omega_{\beta} = 1/\sqrt{2d}$ if $\bar{\beta} = e_{1} \pm e_{k} \ (2 \leq k \leq n)$, and $\omega_{\beta} = 1/(2\sqrt{d})$ if $\bar{\beta} = e_{1}$.
By Lemma \ref{2-2-2} and Lemma \ref{4-2-1}, 
\[
\begin{split}
\Big( R^{M} \Big( {}^{N}(S^{(i)}_{\beta})^{*}_{x}, {}^{N}(S^{(1)}_{\alpha_{1}})^{*}_{x} \Big) {}^{N}(S^{(1)}_{\alpha_{1}})^{*}_{x} \Big)_{T_{x}N} 
&= 
- \Big( (\mathrm{exp}H)_{*} \big[ [ \bar{T}^{(i)}_{\beta}, \bar{T}^{(1)}_{\alpha_{1}} ], \bar{T}^{(1)}_{\alpha_{1}} \big] \Big)_{T_{x}N} \\
&=
- \Big( (\mathrm{exp}H)_{*} \left( \frac{1}{\sqrt{d}} \right)^{2} \omega_{\beta} \big[ [T^{(i)}_{\beta}, T^{(1)}_{\alpha_{1}} ], T^{(1)}_{\alpha_{1}} \big] \Big)_{T_{x}N} \\
&=
\left( \frac{1}{d} \right) (\mathrm{exp}H)_{*} \omega_{\beta} T^{(i)}_{\beta} 
= 
\left( \frac{1}{d} \right) {}^{N}(S^{(i)}_{\beta})^{*}_{x}.
\end{split}
\]
Moreover, by Lemma \ref{4-2-2},
\[
\begin{split}
\Big( R^{M} \Big( {}^{N}(S^{(i)}_{\beta})^{*}_{x}, {}^{N}(S^{(1)}_{\alpha_{1}})^{*}_{x} \Big) {}^{N}(S^{(i)}_{\beta})^{*}_{x} \Big)_{T_{x}N} 
&= 
- \Big( (\mathrm{exp}H)_{*} \big[ [ \bar{T}^{(i)}_{\beta}, \bar{T}^{(1)}_{\alpha_{1}} ], \bar{T}^{(i)}_{\beta} \big] \Big)_{T_{x}N} \\
&=
- \left( \frac{1}{\sqrt{d}} \right) (\mathrm{exp}H)_{*} \omega^{2}_{\beta} \Big( \big[ [T^{(i)}_{\beta}, T^{(1)}_{\alpha_{1}} ], T^{(i)}_{\beta} \big] \Big)_{\frak{m}_{x}} \\
&=
-(\mathrm{exp}H)_{*} \left( \frac{1}{\sqrt{d}} \right) ( \omega_{\beta}^{2} \phi_{\beta} c_{\beta}) T^{(1)}_{\alpha_{1}} \\
&=
-(\omega_{\beta}^{2} \phi_{\beta} c_{\beta}) {}^{N}(S^{(1)}_{\alpha_{1}})^{*}_{x},
\end{split}
\]
where $c_{\beta}, \phi_{\beta} \in \mathbb{R}$ such that $c_{\beta}|\beta|^{2} = |\alpha_{1}|^{2}$, and 
\[
\phi_{\beta} =
\left\{
\begin{array}{llllll}
1 & & (\beta \in \tilde{\Sigma}_{(1)}), \\
2 & & (\beta \in \tilde{\Sigma}_{(2)}), \\
4 & & (\beta \in \tilde{\Sigma}_{(3)}). \\
\end{array}
\right.
\]
If $M = \tilde{G}_{2}(\mathbb{R}^{n})\ (\text{$n$ is odd and $n \geq 5$})$ or $Sp(n)/U(n)\ (n \geq 3)$, then $(\omega_{\beta}, c_{\beta}, \phi_{\beta})$ is either
$( 1/\sqrt{2d}, 1, 2)$ or $(1/\sqrt{2d}, 2, 1)$.
If $M$ is one of the others, then $(\omega_{\beta}, c_{\beta}, \phi_{\beta})$ is either $(1/\sqrt{2d}, 1, 2)$ or $(1/2\sqrt{d}, 1, 4)$.
Thus, $\omega_{\beta}^{2} c_{\beta} \phi_{\beta} = 1/d$ for any $M$ and the statement follows.
\end{proof}

By Lemma \ref{2-3-4}, Lemma \ref{4-2-2}, and Lemma \ref{4-2-3}, we obtain Lemma \ref{4-2-6}.

\begin{lemm} \label{4-2-6}
Let $\beta, \gamma \in \tilde{\Sigma}$ satisfy $\bar{\beta}, \bar{\gamma} \in (\Sigma^{+}_{x})_{C}$ and $\beta \not= \gamma, \sigma(\gamma)$.
Then, for any $1 \leq i \not= j \leq 2$,
\[
\begin{split}
& \Big( R^{M} \Big( {}^{N}(S^{(i)}_{\beta})^{*}_{x}, {}^{N}(S^{(1)}_{\alpha_{1}})^{*}_{x} \Big) {}^{N}(S^{(j)}_{\beta})^{*}_{x} \Big)_{T_{x}N} = 0, \\
& \Big( R^{M} \Big( {}^{N}(S^{(i)}_{\beta})^{*}_{x}, {}^{N}(S^{(1)}_{\alpha_{1}})^{*}_{x} \Big) {}^{N}(S^{(i)}_{\gamma})^{*}_{x} \Big)_{T_{x}N} = 0, \\
& \Big( R^{M} \Big( {}^{N}(S^{(i)}_{\beta})^{*}_{x}, {}^{N}(S^{(1)}_{\alpha_{1}})^{*}_{x} \Big) {}^{N}(S^{(j)}_{\gamma})^{*}_{x} \Big)_{T_{x}N} = 0. \\
\end{split}
\]
\end{lemm}

%%%%%%%%%%%%%%%%%%%%%%%%%%%%%%%%%%%%%%%%%%%%%%%%%%%%%%%%%%%%%%%%%%%%%%%%%%%%%%%%%%%%%%%%%%%%%%%%%%%%%%%%%%%%%%%%%%%%%%%%%%%%%%%%%%%%%%%%%%%%%%%%%%%%%%%

\subsection{Sasaki $CR$ orbits} \label{s4-3}

In this subsection, we study the condition under which a contact $CR$ orbit becomes a Sasaki $CR$ orbit.
Let $H = C_{0}(t)\ (0 < t < 1)$ and $x = \mathrm{exp}H(o)$.
Denote by $N$ the $K$-orbit through $x$.

\begin{lemm} \label{4-3-1}
Let $k = (1/d)\left( 1 + 2 \cot t(\pi/2) \cot t\pi \right)$.
Then, for any $X,Y \in T_{x}N$,
\[
R^{N}(X, \xi_{x})Y = k \Big( \langle \xi_{x}, Y \rangle X - \langle X, Y \rangle \xi_{x} \Big).
\]
\end{lemm}

\begin{proof}
First, recall the Gauss equation
\[
\Big\langle R^{N}(X,\xi_{x})Y, \ Z \Big\rangle  = \Big\langle R^{M}(X,\xi_{x})Y, \ Z \Big\rangle - \Big\langle h(X, Y), h(\xi_{x}, Z) \Big\rangle + \Big\langle h(X, Z), h(\xi_{x}, Y) \Big\rangle,
\]
where $X,Y,Z \in T_{x}N$.
Note $\xi_{x} = {}^{N}(S^{(1)}_{\alpha_{1}})^{*}_{x}$.
Let $\beta \in \tilde{\Sigma}$ satisfy $\bar{\beta} \in (\Sigma^{+}_{x})_{C}$.
If $S^{(i)}_{\beta} \not= 0$, then for any $A \in P$, by Lemma \ref{4-2-4} and Lemma \ref{4-2-5},
\[
\begin{split}
& \Big\langle R^{M} \Big( {}^{N}(S^{(i)}_{\beta})^{*}_{x}, \xi_{x} \Big) \xi_{x}, \ A \Big\rangle 
=
\Big\langle \frac{1}{d} {}^{N}(S^{(i)}_{\beta})^{*}_{x}, A \Big\rangle = 
\left\{
\begin{array}{llll}
\vspace{2mm}
1/d & (A = {}^{N}(S^{(i)}_{\beta})^{*}_{x}), \\
0 & (\text{others}),
\end{array}
\right. 
\\
& \Big\langle h^{N} \Big( {}^{N}(S^{(i)}_{\beta})^{*}_{x}, \xi_{x} \Big),  h^{N} \Big( \xi_{x}, \ A \Big)  \Big\rangle
= \Big\langle 0, h^{N} \Big( \xi_{x}, A \Big)  \Big\rangle = 0, \\
& \Big\langle h^{N} \Big( {}^{N}(S^{(i)}_{\beta})^{*}_{x}, A \Big),  h^{N} \Big( \xi_{x}, \xi_{x} \Big)  \Big\rangle
= 
\left\{
\begin{array}{lllll}
\vspace{1mm}
(2/d) \cot t(\pi/2) \cot t\pi & (A = {}^{N}(S^{(i)}_{\beta})^{*}_{x}), \\
0 & (\text{others}). \\
\end{array}
\right.
\end{split}
\]
Thus, we obtain
\[
R^{N} \Big( {}^{N}(S^{(i)}_{\beta})^{*}_{x}, \xi_{x} \Big) \xi_{x} = (1/d)(1 + 2 \cot t(\pi/2) \cot t\pi ){}^{N}(S^{(i)}_{\beta})^{*}_{x} = k {}^{N}(S^{(i)}_{\beta})^{*}_{x}.
\]
Moreover, by Lemma \ref{4-2-4} and Lemma \ref{4-2-5}, for any $A \in P$,
\[
\begin{split}
& \Big\langle R^{M} \Big( {}^{N}(S^{(i)}_{\beta})^{*}_{x}, \xi_{x} \Big) {}^{N}(S^{(i)}_{\beta})^{*}_{x}, \ A \Big\rangle 
=
\Big\langle -\frac{1}{d}\xi_{x}, A \Big\rangle = 
\left\{
\begin{array}{llll}
\vspace{2mm}
-(1/d) & (A = \xi_{x}), \\
0 & (\text{others}),
\end{array}
\right. 
\\
& \Big\langle h^{N} \Big( {}^{N}(S^{(i)}_{\beta})^{*}_{x}, {}^{N}(S^{(i)}_{\beta})^{*}_{x} \Big),  h^{N} \Big( \xi_{x}, A \Big)  \Big\rangle
= 
\left\{
\begin{array}{lllll}
\vspace{1mm}
(2/d) \cot t (\pi/2) \cot t\pi & (A = \xi_{x}), \\
0 & (\text{others}), \\
\end{array}
\right. \\
& \Big\langle h^{N} \Big( {}^{N}(S^{(i)}_{\beta})^{*}_{x}, \ A \Big),  h^{N} \Big( {}^{N}(S^{(i)}_{\beta})^{*}_{x} , \xi_{x} \Big)  \Big\rangle
= 
\Big\langle h^{N} \Big( {}^{N}(S^{(i)}_{\beta})^{*}_{x}, \ A \Big),  0 \Big\rangle = 0.
\end{split}
\]
Hence,
\[
R^{N} \Big( {}^{N}(S^{(i)}_{\beta})^{*}_{x}, \xi_{x} \Big) {}^{N}(S^{(i)}_{\beta})^{*}_{x} = - (1/d)( 1 + 2 \cot t(\pi/2) \cot t \pi)\xi_{x} = -k \xi_{x}.
\]
Assume that $S^{(1)}_{\beta}, S^{(2)}_{\beta} \not = 0$.
Then, by Lemma \ref{4-2-4} and Lemma \ref{4-2-6}, for any $A \in P$ and $1 \leq i \not= j \leq 2$,
\[
\begin{split}
& \Big\langle R^{M} \Big( {}^{N}(S^{(i)}_{\beta})^{*}_{x}, \xi_{x} \Big) {}^{N}(S^{(j)}_{\beta})^{*}_{x}, \ A \Big\rangle 
= \langle 0, A \rangle = 0, \\
& \Big\langle h^{N} \Big( {}^{N}(S^{(i)}_{\beta})^{*}_{x}, {}^{N}(S^{(j)}_{\beta})^{*}_{x} \Big),  h^{N} \Big( \xi_{x}, A \Big)  \Big\rangle
= \Big\langle 0, h^{N} \Big( \xi^{*}_{x}, A \Big)  \Big\rangle = 0, \\
& \Big\langle h^{N}\Big( {}^{N}(S^{(i)}_{\beta})^{*}_{x}, A \Big),  h^{N} \Big( {}^{N}(S^{(j)}_{\beta})^{*}_{x}, \xi^{*}_{x} \Big)  \Big\rangle
= 
\Big\langle h^{N}\Big( {}^{N}(S^{(i)}_{\beta})^{*}_{x}, A \Big), 0 \Big\rangle 
=
0.
\end{split}
\]
Moreover, let $\gamma \in \tilde{\Sigma}$ be $\bar{\gamma} \in (\Sigma^{+}_{x})_{C}$ and $\gamma \not= \beta, \sigma(\beta)$.
Then, by Lemma \ref{4-2-4} and Lemma \ref{4-2-6}, for any $A \in P$ and $1 \leq i \not= j \leq 2$,
\[
\begin{split}
& \Big\langle \Big( R^{M} \Big( {}^{N}(S^{(i)}_{\beta})^{*}_{x}, \xi_{x} \Big) {}^{N}(S^{(j)}_{\gamma})^{*}_{x}, \ A \Big\rangle 
= \langle 0, A \rangle = 0, \\
& \Big\langle h^{N}\Big( {}^{N}(S^{(i)}_{\beta})^{*}_{x}, {}^{N}(S^{(j)}_{\gamma})^{*}_{x} \Big),  h^{N} \Big( \xi_{x}, \ A \Big)  \Big\rangle
=  0, \\
& \Big\langle  h^{N}\Big( {}^{N}(S^{(i)}_{\beta})^{*}_{x}, \ A \Big),  h^{N} \Big( {}^{N}(S^{(j)}_{\gamma})^{*}_{x}, \xi_{x} \Big)  \Big\rangle
= 
\Big\langle  h^{N}\Big( {}^{N}(S^{(i)}_{\beta})^{*}_{x}, \ A \Big),  0 \Big\rangle = 0.
\end{split}
\]
Therefore, for any $1 \leq i \not= j \leq 2$, 
\[
R^{N} \big( {}^{N}(S^{(i)}_{\beta})^{*}_{x}, {}^{N}(S^{(1)}_{\alpha_{1}})^{*}_{x} \big) {}^{N}(S^{(j)}_{\beta})^{*}_{x} = 0,
\]
and for any $i,j = 1,2$,
\[
R^{N} \Big( {}^{N}(S^{(i)}_{\beta})^{*}_{x}, {}^{N}(S^{(1)}_{\alpha_{1}})^{*}_{x} \Big) {}^{N}(S^{(j)}_{\gamma})^{*}_{x} = 0.
\]
Summarizing these arguments, for any $A,B \in P$,
\[
R^{N}( A, \xi_{x} )B = 
\left\{
\begin{array}{lllll}
k A & & (A \not= \xi_{x}, \ B = \xi_{x}), \\
- k \xi_{x} & & (A = B \ \text{and} \ A \not= \xi_{x}), \\
0 & & (\text{others}).
\end{array}
\right.
\]
Since $P$ is an orthonormal basis of $T_{x}N$, for any $X,Y \in T_{x}N$,
\[
R^{N}(X,\xi_{x})Y = k \Big( \langle \xi_{x}, Y \rangle X - \langle X, Y \rangle \xi_{x} \Big).
\]
\end{proof}

\begin{thm} \label{4-3-2}
Let $\langle \ ,\ \rangle$ be an invariant Riemannian metric of an irreducible Hermitian symmetric space $M$ of compact type and $\sqrt{d}\pi$ be the length of a shortest closed geodesic of $M$.
Moreover, define $0 < t < 1$ such that
\[
\tan^{2} \frac{\pi}{2}t = \frac{1}{d}.
\]
The $K$-orbit $N$ is a Sasaki $CR$ submanifold if and only if either $\Sigma$ is of type $BC_{n}\ (n \geq 1)$ and $N$ is the orbit through $\mathrm{exp}C_{0}(t)(o) \ ( 0 < t < 1)$, or $\Sigma$ is of type $C_{n}\ (n \geq 2)$ and $N$ is the orbit through $\mathrm{exp}C_{0}(t)(o)$ or $\mathrm{exp}C_{n-1}(1-t)(o) \ ( 0 < t < 1)$.
\end{thm}

\begin{proof}
By easy calculations, we see that $k=1$ if and only if $\tan^{2} t \pi/2 = 1/d$.
Therefore, by Lemma \ref{4-3-1}, the statement follows.
\end{proof}

By Theorem \ref{3-3-7}, we obtain Corollary \ref{4-3-3} immediately.

\begin{coro} \label{4-3-3}
If $d =1$, then any Sasaki $CR$ orbit is a ruled $CR$ orbit.
\end{coro}

\begin{remark}
Let $M = \mathbb{C}P^{n-1} \ (n \geq 3)$, and the holomorphic sectional curvature of $M$ is denoted by $c$.
Then, the length of any shortest closed geodesic of $M$ is $2\pi/\sqrt{c}$.
Hence, $\sqrt{d}\pi = 2\pi/\sqrt{c}$.
According to the result of \cite{Adachi-Kameda-Maeda} and \cite{Berndt}, if $|C_{0}(t)| = (2/\sqrt{c})\tan^{-1}(\sqrt{c}/2)$, then the $K$-orbit $N$ through $x = (\mathrm{exp}\:C_{0}(t))(o)$ is a Sasakian manifold by the induced metric.
We apply Theorem \ref{4-3-2} to this case.
Since $\tan^{2} (t \pi/2) = 1/d$, we see $t \pi/2 = \tan^{-1}(1/\sqrt{d})$.
Therefore, we have
\[
|C_{0}(t)| = \frac{t\pi}{2}\sqrt{d} = \sqrt{d} \tan^{-1} \frac{1}{\sqrt{d}} = \frac{2}{\sqrt{c}} \tan^{-1} \frac{\sqrt{c}}{2}.
\]
Thus, our result is consistent with the result in \cite{Adachi-Kameda-Maeda} and \cite{Berndt}.
\end{remark}

%%%%%%%%%%%%%%%%%%%%%%%%%%%%%%%%%%%%%%%%%%%%%%%%%%%%%%%%%%%%%%%%%%%%%%%%%%%%%%%%%%%%%%%%%%%%%%%%%%%%%%%%%%%%%%%%%%%%%%%%%%%%%%%%%%%%%%%%%%%%%%%%%%%%%%%

\section{Sasaki-Einstein $CR$ orbits} \label{s5}

In this section, we consider the condition under which a Sasaki $CR$ orbit is Einstein with respect to the induced metric.
Let $H = C_{0}(t)\ (0 < t < 1)$ and $x = \mathrm{exp}H(o)$.
Denote by $N$ the $K$-orbit through $x$.
Moreover, assume that $N$ is a Sasaki $CR$ orbit, that is, $\tan^{2}t\pi/2 = 1/d$.
Then,
\[
\sin^{2} \frac{t\pi}{2} = \frac{1}{d+1}, \quad \cos^{2} \frac{t\pi}{2} = \frac{d}{d+1}.
\]
We denote the induced metric on $N$ by the same symbol $\langle \ ,\ \rangle$.
By Remark \ref{3-2-6}, $N$ is an $S^{1}$-bundle over the polar $M^{+}_{1} = K(p_{1})$.
Note that $M^{+}_{1}$ is a Hermitian symmetric space of compact type.
In particular, $M^{+}_{1}$ is a leaf space of the Reeb foliation of $N$, and the projection $\pi:N = K/K_{x} \rightarrow M^{+}_{1} = K/K_{1}$ is the natural projection onto the leaf space.
Then, by the arguments in Section 2, there exists a Riemannian metric $\langle \ ,\ \rangle_{\pi}$ on $M^{+}_{1}$ such that $\pi : (N, \langle \ ,\ \rangle) \rightarrow (M^{+}_{1}, \langle \ .\ \rangle_{\pi})$ is a Riemannian submersion.
The induced metric on $M^{+}_{1}$ by the Riemannian metric on $M$ is denoted by the same symbol $\langle \ ,\ \rangle$.
By straightforward calculations, we see that $(M^{+}_{1}, \langle \ ,\ \rangle_{\pi})$ is isometric to $(M^{+}_{1}, \sin^{2} (t\pi/2)\langle \ ,\ \rangle) = (M^{+}_{1}, (1/(d+1))\langle \ ,\ \rangle)$.
According to the general theory of Sasakian manifolds, $(N, \langle\ ,\ \rangle)$ is a $\eta$-Einstein Sasakian manifold if and only if $(M^{+}_{1}, (1/(d+1))\langle \ ,\ \rangle)$ is a K\"{a}hler Einstein manifold.
Moreover, $(N, \langle\ ,\ \rangle)$ is a Sasaki-Einstein manifold if and only if $(M^{+}_{1}, (1/(d+1))\langle \ ,\ \rangle)$ is a K\"{a}hler manifold with the Einstein constant $\dim M^{+}_{1} + 2$.

For each irreducible Hermitian symmetric space of compact type, the polar $M^{+}_{1}$ is listed in Table 2 \cite{Chen-Nagano}.
In general, any irreducible symmetric space is Einstein with respect to the invariant metric \cite{Besse}.
In $G_{k}(\mathbb{C}^{n})\ (k \not= 1, \ 2k \not= n)$, we can easily verify that $M_{1}^{+} = \mathbb{C}P^{k-1} \times \mathbb{C}P^{n-k-1}$ is not Einstein with respect to the induced metric.
On the other hand, in $\tilde{G}_{2}(\mathbb{R}^{6})$, we see that $M^{+}_{1} = \tilde{G}_{2}(\mathbb{R}^{4}) = S^{2} \times S^{2}$ is Einstein with respect to the induced metric.
Moreover, in $G_{k}(\mathbb{C}^{2k})$, we see that $M_{1}^{+} = \mathbb{C}P^{k-1} \times \mathbb{C}P^{k-1}$ is Einstein with respect to the induced metric, too.

\begin{table}[htb]
\caption{The list of $M_{1}^{+}$}
\begin{tabular}{c|ccccccccccccccccccccccccccccccccccc} \hline
$M$ & $G_{k}(\mathbb{C}^{n})$ & $\tilde{G}_{2}(\mathbb{R}^{n})$ & $SO(n)/U(n)$ & $Sp(n)/U(n)$ & $EIII$ & $EVII$ \\ 
 & $(n > k)$ & $(n \geq 3)$ & $(n \geq 2)$ & $(n \geq 1)$ & & \\ \hline
$M_{1}^{+}$ & $\mathbb{C}P^{k-1} \times \mathbb{C}P^{n-k-1}$ & $\tilde{G}_{2}(\mathbb{R}^{n-2})$ & $G_{2}(\mathbb{C}^{n})$ & $\mathbb{C}P^{n-1}$ & $SO(10)/U(5)$ & $EIII$ \\ \hline
\end{tabular}   
\end{table}

\begin{thm} \label{5-1-1}
If $M$ is an irreducible symmetric space of compact type, except for $G_{k}(\mathbb{C}^{n})\ (k< n, \ k \not= 1, \ 2k \not= n)$, then any Sasaki $CR$ orbit of the isotropy group action is $\eta$-Einstein.
\end{thm}

Let $M$ be $\mathbb{C}P^{n-1}\ (n \geq 3)$.
Then, $M^{+}_{1} = \mathbb{C}P^{n-2}$, and the length of any shortest closed geodesic of $(M, \langle\ ,\ \rangle)$ is equal to that of $(M_{1}^{+}, \langle \ ,\ \rangle)$.
Since the length of a shortest closed geodesic of $(M^{+}_{1}, \langle \ ,\ \rangle)$ is $\sqrt{d} \pi$, that of $(M^{+}_{1}, (1/(d+1))\langle \ ,\ \rangle)$ is $\pi \sqrt{d/(d+1)}$.
In general, if the length of a shortest closed geodesic of $\mathbb{C}P^{n-1}$ is $\sqrt{a}\pi$ for some $a > 0$, then the holomorphic sectional curvature is $4/a$, and the Einstein constant is $2n/a$.
Thus, the Einstein constant of $(M^{+}_{1}, (1/(d+1))\langle \ ,\ \rangle)$ is $2(n-1)/(d/(d+1))$.
However, 
\[
2(n-1) \frac{d+1}{d} = 2 + \dim M^{+}_{1} = 2(n-1)
\]
has no solution for $d >0$.
Thus, we obtain Theorem \ref{5-1-2}.

\begin{lemm} \label{5-1-2}
If $M = \mathbb{C}P^{n-1} \ (n \geq 3)$, then for any invariant metric, there are no Sasaki-Einstein $CR$ orbits.
\end{lemm}

In the following arguments, let $M$ be an irreducible Hermitian symmetric space of compact type, except for $\mathbb{C}P^{n-1}$ and $G_{k}(\mathbb{C}^{n})\ (k \not= 1, 2k \not= n)$, and we 
study the Einstein constant of $(M^{+}_{1}, (1/(d+1))\langle \ ,\ \rangle)$.
Then, the restricted root system $\Sigma$ is of type $BC_{n}$ or $C_{n} \ (n \geq 2)$.
It is well known that a polar is a totally geodesic submanifold.
The corresponding Lie triple system $\frak{m}_{1}^{+}$ of $M^{+}_{1}$ in $\frak{m}$ is given by
\[
\frak{m}_{1}^{+} = \sum_{\lambda \in (\Sigma^{+}_{x})_{C}}\frak{m}_{\lambda}.
\]
In particular, $M^{+}_{1} = (\mathrm{exp}H)(\mathrm{exp}\frak{m}_{1}^{+})(o)$.
The Ricci tensor of $(M^{+}, (1/(d+1))\langle \ ,\ \rangle)$ is denoted by $r^{+}$.
For any $X,Y \in \frak{m}_{1}^{+}$, it is known that the Ricci tensor $r^{+}$ is given by
\[
\begin{split}
r^{+} \Big( (\mathrm{exp}H)_{*}X, (\mathrm{exp}H)_{*}Y \Big) 
&= \mathrm{tr} \big( \frak{m}_{1}^{+} \ni Z \mapsto [X, [Z, Y]] \in \frak{m}_{1}^{+} \big) \\
&= \sum_{i=1}^{l} \Big\langle \Big[ X, \big[ D_{i}, Y \big] \Big], D_{i} \Big\rangle 
= \sum_{i=1}^{l} \Big\langle \Big[ D_{i}, X \Big], \Big[ D_{i}, Y \Big] \Big\rangle,
\end{split}
\]
where $D_{1}, \cdots, D_{l}$ is an orthonormal basis of $\frak{m}_{1}^{+}$ with respect to $(1/(d+1))\langle \ ,\ \rangle$ \cite{Besse}.

%%%%%%%%%%%%%%%%%%%%%%%%%%%%%%%%%%%%%%%%%%%%%%%%%%%%%%%%%%%%%%%%%%%%%%%%%%%%%%%%%%%%%%%%%%%%%%%%%%%%%%%%%%%%%%%%%%%%%%%%%%%%%%%%%%%%%%%%%%%%%%%%%%%%%%%

\subsection{Some preliminaries II} \label{s5-1}

In this subsection, we provide some preliminaries to study the Ricci tensor $r^{+}$ in the next subsection.
We assume that $\Sigma$ is of type $BC_{n}$ and $C_{n}\ (n \geq 2)$.
As stated in subsection \ref{s2-4}, if $\alpha \in \tilde{\Sigma}$ satisfies $\bar{\alpha} \in (\Sigma^{+}_{x})_{C}$, then $\alpha$ is a longest root if and only if $\sigma(\alpha) \not= \alpha$, and $\alpha$ is a shortest root if and only if $\sigma(\alpha) = \alpha$.

Let $\alpha \in \tilde{\Sigma}$ be a longest root such that $\bar{\alpha} \in (\Sigma^{+}_{x})_{C}$ and $2\bar{\alpha} \not\in \Sigma$.
Then, as in the proof of Lemma \ref{2-4-4}, we have $\alpha \pm \sigma(\alpha) \not\in \tilde{\Sigma}$.
Hence, $[W_{\alpha}, W_{\sigma(\alpha)}] = [U_{\alpha}, U_{\sigma(\alpha)}] = 0$, and for any $X \in \frak{g}$,
\[
\begin{split}
& \Big[ T_{\alpha}^{(1)}, \big[ T_{\alpha}^{(1)}, X \big] \Big] 
= \Big[ W_{\alpha}, \big[ W_{\alpha}, X \big] \Big] + 2 \Big[ W_{\sigma(\alpha)}, \big[ W_{\alpha}, X \big] \Big] + \Big[ W_{\sigma(\alpha)}, \big[ W_{\sigma(\alpha)}, X \big] \Big], \\
& \Big[ T_{\alpha}^{(2)}, \big[ T_{\alpha}^{(2)}, X \big] \Big] 
= \Big[ U_{\alpha}, \big[ U_{\alpha}, X \big] \Big] - 2 \Big[ U_{\sigma(\alpha)}, [U_{\alpha}, X \big] \Big] + \Big[ U_{\sigma(\alpha)}, \big[ U_{\sigma(\alpha)}, X \big] \Big]. \\
\end{split}
\]

\begin{lemm}\label{5-2-2}
Let $\alpha, \beta \in \tilde{\Sigma}$ be longest roots and $\bar{\alpha} = e_{1} + e_{k}, \bar{\beta} = e_{1} \pm e_{l} \ (2 \leq k \not= l \leq n)$.
Then, either $\beta$ or $\sigma(\beta)$ is orthogonal to $\alpha$, while the other is not orthogonal to $\alpha$.
Moreover, for any $i,j = 1,2$,
\[
\begin{split}
\Big\langle \big[ T^{(i)}_{\alpha}, T^{(j)}_{\beta} \big], \big[ T^{(i)}_{\alpha}, T^{(j)}_{\beta} \big] \Big\rangle = 2d.
\end{split}
\]
\end{lemm}

\begin{proof}
The $\frak{a}_{0}$-part and $i\frak{b}$-part of $\alpha$ (resp.\ $\beta$) are denoted by $\lambda_{1}$ and $\mu_{1}$ (resp.\ $\lambda_{2}$ and $\mu_{2}$), respectively.
Then, $\lambda_{1} = e_{1} + e_{k}$ and $\lambda_{2} = e_{1} \pm e_{l}$.
Assume $(\alpha, \beta) = (\alpha, \sigma(\beta)) = 0$.
Since
\[
\begin{split}
0 &= (\alpha, \beta) = (\lambda_{1}, \lambda_{2}) + ( \mu_{1}, \mu_{2} ), \quad 
0 = (\alpha, \sigma(\beta)) = (\lambda_{1}, \lambda_{2}) - (\mu_{1}, \mu_{2}),
\end{split}
\]
we obtain $(\lambda_{1}, \lambda_{2}) = (\mu_{1}, \mu_{2}) = 0$, which is a contradiction.
Next, assume that neither $(\alpha, \beta)$ nor $(\alpha, \sigma(\beta))$ is $0$.
Then, since
\[
\begin{split}
\dfrac{2(\alpha, \beta)}{(\alpha, \alpha)} = \dfrac{2(\alpha, \sigma(\beta))}{(\alpha, \alpha)} = 1,
\end{split}
\]
we obtain $(\mu_{1}, \mu_{2}) = 0$.
By assumption, $\alpha - \beta \in \tilde{\Sigma}$.
Since $\lambda_{1} - \lambda_{2} = e_{k} \mp e_{l}$ and $(\mu_{1}, \mu_{2}) = 0$, we have $|\alpha - \beta| > |\alpha|$, which contradicts to the assumption that $\alpha$ is a longest root.
Thus, the former part of the statement is true.
Next, we will show that the latter part follows.
We assume $(\alpha, \sigma(\beta)) = 0$.
Then, 
\[
\big[ U_{\alpha}, U_{\sigma(\beta)} \big] = \big[ U_{\alpha}, W_{\sigma(\beta)} \big] = \big[ W_{\alpha}, U_{\sigma(\beta)} \big] = \big[ W_{\alpha}, W_{\sigma(\beta)} \big] = 0.
\]
Moreover, since $(\sigma(\alpha), \alpha) = 0$ by Lemma \ref{2-4-4} and $(\sigma(\alpha), \beta) = 0$, we have $(\sigma(\alpha), \alpha - \beta) = 0$.
Hence,
\[
\begin{split}
\Big[U_{\sigma(\alpha)}, \big[ U_{\alpha}, U_{\beta} \big] \Big] 
&= \Big[U_{\sigma(\alpha)}, \big[ U_{\alpha}, W_{\beta} \big] \Big] = \Big[W_{\sigma(\alpha)}, \big[ W_{\alpha}, W_{\beta} \big] \Big] \\
&= \Big[W_{\sigma(\alpha)}, \big[ W_{\alpha}, U_{\beta} \big] \Big] = 0.
\end{split}
\]
Therefore,
\[
\begin{split}
& \Big[ T^{(i)}_{\alpha}, \big[ T^{(i)}_{\alpha}, T^{(1)}_{\beta} \big] \Big] 
= 
-W_{\beta} - W_{\sigma(\beta)} = -T^{(1)}_{\beta}, \\
& \Big[ T^{(i)}_{\alpha}, \big[ T^{(i)}_{\alpha}, T^{(2)}_{\beta} \big] \Big] 
= 
U_{\beta} - U_{\sigma(\beta)} = -T^{(2)}_{\beta}. \\
\end{split}
\]
Thus,
\[
\begin{split}
\Big\langle \big[ T^{(i)}_{\alpha}, T^{(j)}_{\beta} \big], \big[ T^{(i)}_{\alpha}, T^{(j)}_{\beta} \big] \Big\rangle
& = - \Big\langle \Big[ T^{(i)}_{\alpha}, \big[ T^{(i)}_{\alpha}, T^{(j)}_{\beta} \big] \Big], T^{(j)}_{\beta} \Big\rangle \\
& = \Big\langle T^{(j)}_{\beta}, T^{(j)}_{\beta} \Big\rangle 
= \frac{4}{(\bar{\beta}, \bar{\beta})}  = 2d.
\end{split}
\]
In the case of $(\alpha, \beta) = 0$, using a similar argument we can verify that the latter part of the statement is true.
\end{proof}

\begin{lemm} \label{5-2-3}
Let $\Sigma$ be of type $C_{n}\ (n \geq 2)$.
If shortest roots $\alpha, \beta \in \tilde{\Sigma}$ satisfy $\bar{\alpha} = e_{1} \pm e_{k}, \bar{\beta} = e_{1} \pm e_{l} \ (1 \leq k \not= l \leq n)$, then 
\[
\Big\langle \big[ T^{(1)}_{\alpha}, T^{(1)}_{\beta} \big], \big[ T^{(1)}_{\alpha}, T^{(1)}_{\beta} \big] \Big\rangle = 2d.
\]
\end{lemm}

\begin{proof}
By assumption, $\alpha, \beta \in \tilde{\Sigma}_{(1)}$.
We assume $\epsilon_{\alpha, -\beta} = -1$.
By definition,
\[
[T^{(1)}_{\alpha}, T^{(1)}_{\beta}] = [W_{\alpha}, W_{\beta}] = -iN_{\alpha, -\beta}W_{\alpha - \beta} = -iN_{\alpha, -\beta}T^{(1)}_{\alpha - \beta},
\]
and $N_{\alpha, -\beta} = \pm i$, which contradicts to $[T^{(1)}_{\alpha}, T^{(1)}_{\beta}] \in \frak{k}$.
Thus, $\epsilon_{\alpha, -\beta} = 1$.
Therefore, $N_{\alpha, -\beta} = \pm 1$, and 
\[
[T^{(1)}_{\alpha}, T^{(1)}_{\beta}] = N_{\alpha, -\beta}U_{\alpha - \beta} = N_{\alpha, -\beta}S^{(1)}_{\alpha - \beta}.
\]
Hence,
\[
\Big\langle [T^{(1)}_{\alpha}, T^{(1)}_{\beta}], [T^{(1)}_{\alpha}, T^{(1)}_{\beta}] \Big\rangle = \Big\langle S^{(1)}_{\alpha - \beta}, S^{(1)}_{\alpha - \beta} \Big\rangle = \frac{4}{(\overline{\alpha - \beta}, \overline{\alpha - \beta})} = 2d
\]
and the statement follows.
\end{proof}

\begin{lemm} \label{5-2-4}
Let $\alpha \in \tilde{\Sigma}$ be a longest root and $\bar{\alpha} = e_{1} + e_{k} \ (2 \leq k \leq n)$.
If $\beta \in \tilde{\Sigma}$ satisfies $\bar{\beta} = e_{1} - e_{k}$, then $(\alpha, \beta) = 0$, except for $\beta = \alpha_{1} - \alpha$ and $\beta =  \alpha - \alpha_{k}$.
\end{lemm}

\begin{proof}
Let $(\alpha, \beta) \not= 0$.
Since $\alpha$ is a longest root, $2(\alpha, \beta)/(\alpha, \alpha)$ is either $1$ or $-1$.
If $2(\alpha, \beta) = (\alpha, \alpha)$, then $\alpha - \beta \in \tilde{\Sigma}$.
Since $\overline{\alpha - \beta} = 2e_{k}$ and $m(2e_{k}) = 1$, we have $\alpha - \beta = \alpha_{k}$ and $\beta = \alpha - \alpha_{k}$.
If $2(\alpha, \beta) = -(\alpha, \alpha)$, then $\alpha + \beta \in \tilde{\Sigma}$.
Since $\overline{\alpha + \beta} = 2e_{1}$ and $m(2e_{1}) = 1$, we have $\alpha + \beta = \alpha_{1}$ and $\beta = \alpha_{1} - \alpha$.
\end{proof}

Note that $\sigma(\alpha - \alpha_{k}) = \alpha_{1} - \alpha$ and $\alpha - \alpha_{k}$ is a longest root.

\begin{lemm} \label{5-2-5}
Let $\alpha \in \tilde{\Sigma}$ be a longest root and $\bar{\alpha} = e_{1} + e_{k} \ (2 \leq k \leq n)$.
Then, for any $i = 1, 2$,
\[
\Big\langle \big[ T^{(i)}_{\alpha}, T^{(1)}_{\alpha_{1} - \alpha} \big], \big[ T^{(i)}_{\alpha}, T^{(1)}_{\alpha_{1} - \alpha} \big] \Big\rangle 
+ \Big\langle \big[ T^{(i)}_{\alpha}, T^{(2)}_{\alpha_{1} - \alpha} \big], \big[ T^{(i)}_{\alpha}, T^{(2)}_{\alpha_{1} - \alpha} \big] \Big\rangle = 8d.
\]
\end{lemm}

\begin{proof}
We can show the statement by direct calculations.
If $(\epsilon_{\alpha, -(\alpha - \alpha_{k})}, \epsilon_{\alpha, \alpha_{1} - \alpha}) = (1,1)$, then $N^{2}_{\alpha, -(\alpha - \alpha_{k})} = N^{2}_{\alpha, \alpha_{1} - \alpha} = 1$ by definition.
Hence,
\[
\begin{split}
& [T^{(1)}_{\alpha}, T^{(1)}_{\alpha - \alpha_{k}}] \\
&= 
[W_{\alpha}, W_{\alpha - \alpha_{k}}] + [W_{\sigma(\alpha)}, W_{\sigma(\alpha) - \alpha_{k}}] + [W_{\alpha}, W_{\alpha_{1} - \alpha}] + [W_{\sigma(\alpha)}, W_{\alpha_{1} - \sigma(\alpha)}]  \\
&=
N_{\alpha, -(\alpha - \alpha_{k})}U_{\alpha_{k}} + N_{\sigma(\alpha), -(\sigma(\alpha) - \alpha_{k})}U_{\alpha_{k}} + N_{\alpha, -(\alpha - \alpha_{k})}U_{\alpha_{k}} + N_{\sigma(\alpha), -(\sigma(\alpha) - \alpha_{k})}U_{\alpha_{k}} \\
&=
2N_{\alpha, -(\alpha - \alpha_{k})}S^{(1)}_{\alpha_{k}} -2N_{\alpha, \alpha_{1} - \alpha}S^{(1)}_{\alpha_{1}}, \\
& [T^{(1)}_{\alpha}, T^{(2)}_{\alpha - \alpha_{k}}] \\
&=
-[W_{\alpha}, U_{\alpha - \alpha_{k}}] + [W_{\sigma(\alpha)}, U_{\sigma(\alpha) - \alpha_{k}}] + [W_{\alpha}, U_{\alpha_{1} - \alpha}] - [W_{\sigma(\alpha)}, W_{\alpha_{1} - \sigma(\alpha)}] \\
&=
-N_{\alpha, -(\alpha - \alpha_{k})}W_{\alpha_{k}} + N_{\sigma(\alpha), -(\sigma(\alpha) - \alpha_{k})}W_{\alpha_{k}} + N_{\alpha, \alpha_{1} - \alpha}W_{\alpha_{1}} - N_{\sigma(\alpha), \alpha_{1} - \sigma(\alpha)}W_{\alpha_{1}} = 0, \\
& [T^{(2)}_{\alpha}, T^{(1)}_{\alpha - \alpha_{k}}] \\
&=
-[U_{\alpha}, W_{\alpha - \alpha_{k}}] + [U_{\sigma(\alpha)}, W_{\sigma(\alpha) - \alpha_{k}}] - [U_{\alpha}, W_{\alpha_{1} - \alpha}] + [U_{\sigma(\alpha)}, W_{\alpha_{1} - \sigma(\alpha)}] \\
&=
N_{\alpha, -(\alpha - \alpha_{k})}W_{\alpha_{k}} - N_{\sigma(\alpha), -(\sigma(\alpha) - \alpha_{k})}W_{\alpha_{k}} -N_{\alpha, \alpha_{1} - \alpha}W_{\alpha_{1}} + N_{\sigma(\alpha), \alpha_{1} - \sigma(\alpha)}W_{\alpha_{1}} = 0, \\
& [T^{(2)}_{\alpha}, T^{(2)}_{\alpha - \alpha_{k}}] \\
&=
[U_{\alpha}, U_{\alpha - \alpha_{k}}] + [U_{\sigma(\alpha)}, U_{\sigma(\alpha) - \alpha_{k}}] - [U_{\alpha}, U_{\alpha_{1} - \alpha}] + [U_{\sigma(\alpha)}, U_{\alpha_{1} - \sigma(\alpha)}] \\
&=
N_{\alpha, -(\alpha - \alpha_{k})}U_{\alpha_{k}} + N_{\sigma(\alpha), -(\sigma(\alpha) - \alpha_{k})}U_{\alpha_{k}} - N_{\alpha, \alpha_{1} - \alpha}U_{\alpha_{1}} - N_{\sigma(\alpha), \alpha_{1} - \sigma(\alpha)}U_{\alpha_{1}} \\
&=
2N_{\alpha, -(\alpha - \alpha_{k})}S^{(1)}_{\alpha_{k}} - 2N_{\alpha, \alpha_{1} - \alpha}S^{(1)}_{\alpha_{1}}. \\
\end{split}
\]
If $(\epsilon_{\alpha, -(\alpha - \alpha_{k})}, \epsilon_{\alpha, \alpha_{1} - \alpha}) = (1,-1)$, then $N^{2}_{\alpha, -(\alpha - \alpha_{k})} = 1$ and $N^{2}_{\alpha, \alpha_{1} - \alpha} = -1$.
Hence,
\[
\begin{split}
& [T^{(1)}_{\alpha}, T^{(1)}_{\alpha - \alpha_{k}}] \\
&= 
N_{\alpha, -(\alpha - \alpha_{k})}U_{\alpha_{k}} + N_{\sigma(\alpha), -(\sigma(\alpha) - \alpha_{k})}U_{\alpha_{k}} + iN_{\alpha, \alpha_{1} - \alpha}W_{\alpha_{1}} + iN_{\sigma(\alpha), \alpha_{1} - \sigma(\alpha)}W_{\alpha_{1}} \\
&=
2N_{\alpha, -(\alpha - \alpha_{k})}S^{(1)}_{\alpha_{k}}, \\
& [T^{(1)}_{\alpha}, T^{(2)}_{\alpha - \alpha_{k}}] \\
&=
-N_{\alpha, -(\alpha - \alpha_{k})}W_{\alpha_{k}} + N_{\sigma(\alpha), -(\sigma(\alpha) - \alpha_{k})}W_{\alpha_{k}} + iN_{\alpha, \alpha_{1} - \alpha}U_{\alpha_{1}} + iN_{\sigma(\alpha), \alpha_{1} - \sigma(\alpha)}U_{\alpha_{1}} \\
&=
2iN_{\alpha, \alpha_{1} - \alpha}S^{(1)}_{\alpha_{1}}, \\
& [T^{(2)}_{\alpha}, T^{(1)}_{\alpha - \alpha_{k}}] \\
&= 
N_{\alpha, -(\alpha - \alpha_{k})}W_{\alpha_{k}} - N_{\sigma(\alpha), -(\sigma(\alpha) - \alpha_{k})}W_{\alpha_{k}} -iN_{\alpha, \alpha_{1} - \alpha}U_{\alpha_{1}} + iN_{\sigma(\alpha), \alpha_{1} - \sigma(\alpha)}U_{\alpha_{1}} \\
&=
-2iN_{\alpha, \alpha_{1} - \alpha}S^{(1)}_{\alpha_{1}}, \\
& [T^{(2)}_{\alpha}, T^{(2)}_{\alpha - \alpha_{k}}] \\
&=
N_{\alpha, -(\alpha - \alpha_{k})}U_{\alpha_{k}} + N_{\sigma(\alpha), -(\sigma(\alpha) - \alpha_{k})}U_{\alpha_{k}} + iN_{\alpha, \alpha_{1} - \alpha}W_{\alpha_{1}} + iN_{\sigma(\alpha), \alpha_{1} - \sigma(\alpha)}W_{\alpha_{1}} \\
&=
2N_{\alpha, -(\alpha - \alpha_{k})}S^{(1)}_{\alpha_{k}}. \\
\end{split}
\]
If $(\epsilon_{\alpha, -(\alpha - \alpha_{k})}, \epsilon_{\alpha, \alpha_{1} - \alpha}) = (-1,1)$, then $N^{2}_{\alpha, -(\alpha - \alpha_{k})} = -1$ and $N^{2}_{\alpha, -\sigma(\alpha - \alpha_{k})} = 1$.
Hence,
\[
\begin{split}
& [T^{(1)}_{\alpha}, T^{(1)}_{\alpha - \alpha_{k}}] \\
&=
-iN_{\alpha, -(\alpha - \alpha_{k})}W_{\alpha_{k}} - iN_{\sigma(\alpha), -(\sigma(\alpha) - \alpha_{k})}W_{\alpha_{k}} -N_{\alpha, \alpha_{1} - \alpha}U_{\alpha_{1}} - N_{\sigma(\alpha), \alpha_{1} - \sigma(\alpha)}U_{\alpha_{1}} \\
&=
-2N_{\alpha, \alpha_{1} - \alpha}S_{\alpha_{1}}, \\
& [T^{(1)}_{\alpha}, T^{(2)}_{\alpha - \alpha_{k}}] \\
&=
-iN_{\alpha, -(\alpha - \alpha_{k})}U_{\alpha_{k}} + iN_{\sigma(\alpha), -(\sigma(\alpha) - \alpha_{k})}U_{\alpha_{k}} + N_{\alpha, \alpha_{1} - \alpha}W_{\alpha_{1}} - N_{\sigma(\alpha), \alpha_{1} - \sigma(\alpha)}W_{\alpha_{1}} \\
&=
-2iN_{\alpha, -(\alpha - \alpha_{k})}S^{(1)}_{\alpha_{k}}, \\
& [T^{(2)}_{\alpha}, T^{(1)}_{\alpha - \alpha_{k}}] \\
&=
iN_{\alpha, -(\alpha - \alpha_{k})}U_{\alpha_{k}} - iN_{\sigma(\alpha), -(\sigma(\alpha) - \alpha_{k})}U_{\alpha_{k}} -N_{\alpha, \alpha_{1} - \alpha}W_{\alpha_{1}} + N_{\sigma(\alpha), \alpha_{1} - \sigma(\alpha)}W_{\alpha_{1}} \\
&=
2iN_{\alpha, -(\alpha - \alpha_{k})}S^{(1)}_{\alpha_{k}}, \\
& [T^{(2)}_{\alpha}, T^{(2)}_{\alpha - \alpha_{k}}] \\
&=
-iN_{\alpha, -(\alpha - \alpha_{k})}W_{\alpha_{k}} - iN_{\sigma(\alpha), -(\sigma(\alpha) - \alpha_{k})}W_{\alpha_{k}} -N_{\alpha, \alpha_{1} - \alpha}U_{\alpha_{1}} - N_{\sigma(\alpha), \alpha_{1} - \sigma(\alpha)}U_{\alpha_{1}} \\
&=
-2N_{\alpha, \alpha_{1} - \alpha}S^{(1)}_{\alpha_{1}}.
\end{split}
\]
If $(\epsilon_{\alpha, -(\alpha - \alpha_{k})}, \epsilon_{\alpha, \alpha_{1} - \alpha}) = (-1,-1)$, then $N^{2}_{\alpha, -(\alpha - \alpha_{k})} = N^{2}_{\alpha, -\sigma(\alpha - \alpha_{k})} = -1$.
Hence,
\[
\begin{split}
& [T^{(1)}_{\alpha}, T^{(1)}_{\alpha - \alpha_{k}}] \\
&=
-iN_{\alpha, -(\alpha - \alpha_{k})}W_{\alpha_{k}} - iN_{\sigma(\alpha), -(\sigma(\alpha) - \alpha_{k})}W_{\alpha_{k}} + iN_{\alpha, \alpha_{1} - \alpha}W_{\alpha_{1}} + iN_{\sigma(\alpha), \alpha_{1} - \sigma(\alpha)}W_{\alpha_{1}} = 0, \\
& [T^{(1)}_{\alpha}, T^{(2)}_{\alpha - \alpha_{k}}] \\
&=
-iN_{\alpha, -(\alpha - \alpha_{k})}U_{\alpha_{k}} + iN_{\sigma(\alpha), -(\sigma(\alpha) - \alpha_{k})}U_{\alpha_{k}} + iN_{\alpha, \alpha_{1} - \alpha}U_{\alpha_{1}} + iN_{\sigma(\alpha), \alpha_{1} - \sigma(\alpha)}U_{\alpha_{1}} \\
&=
-2iN_{\alpha, -(\alpha - \alpha_{k})}S^{(1)}_{\alpha_{k}} + 2iN_{\alpha, \alpha_{1} - \alpha}S^{(1)}_{\alpha_{1}}, \\
& [T^{(2)}_{\alpha}, T^{(1)}_{\alpha - \alpha_{k}}] \\
&=
iN_{\alpha, -(\alpha - \alpha_{k})}U_{\alpha_{k}} - iN_{\sigma(\alpha), -(\sigma(\alpha) - \alpha_{k})}U_{\alpha_{k}} -iN_{\alpha, \alpha_{1} - \alpha}U_{\alpha_{1}} + iN_{\sigma(\alpha), \alpha_{1} - \sigma(\alpha)}U_{\alpha_{1}} \\
&=
2iN_{\alpha, -(\alpha - \alpha_{k})}S^{(1)}_{\alpha_{k}} -2iN_{\alpha, \alpha_{1} - \alpha}S^{(1)}_{\alpha_{1}}, \\
& [T^{(2)}_{\alpha}, T^{(2)}_{\alpha - \alpha_{k}}] \\
&=
-iN_{\alpha, -(\alpha - \alpha_{k})}W_{\alpha_{k}} - iN_{\sigma(\alpha), -(\sigma(\alpha) - \alpha_{k})}W_{\alpha_{k}} -iN_{\alpha, \alpha_{1} - \alpha}U_{\alpha_{1}} + iN_{\sigma(\alpha), \alpha_{1} - \sigma(\alpha)}U_{\alpha_{1}} = 0. \\
\end{split}
\]
Since $|S^{(1)}_{\alpha_{1}}|^{2} = |S^{(1)}_{\alpha_{k}}|^{2} = d$, the statement is true.
\end{proof}

\begin{lemm} \label{5-2-6}
Let $\Sigma$ be of type $C_{n}\ (n \geq 2)$ and $\alpha, \beta \in \tilde{\Sigma}$ be shortest roots.
If $\bar{\alpha} = e_{1} + e_{k}, \bar{\beta} = e_{1} - e_{k}\ (2 \leq k \leq n)$, then
\[
\Big\langle [T^{(1)}_{\alpha}, T^{(1)}_{\beta}] , [T^{(1)}_{\alpha}, T^{(1)}_{\beta}]  \Big\rangle = 8d.
\]
\end{lemm}

\begin{proof}
In this case, $\alpha + \beta = \alpha_{1}$ and $\alpha - \beta = \alpha_{k}$ since $\sigma(\alpha) = \alpha$ and $\sigma(\beta) = \beta$.
As in the proof of Lemma \ref{5-2-3}, we immediately obtain $(\epsilon_{\alpha, \beta}, \epsilon_{\alpha, -\beta}) = (1,1)$.
Hence,
\[
[T^{(1)}_{\alpha}, T^{(1)}_{\beta}] = [W_{\alpha}, W_{\beta}] = -N_{\alpha, \beta}S^{(1)}_{\alpha_{1}} + N_{\alpha, -\beta}S^{(1)}_{\alpha_{k}},
\]
where $N^2_{\alpha, \beta} = N^2_{\alpha, -\beta} = 4$ since $\alpha + 2\beta, \alpha - 2\beta \not\in \tilde{\Sigma}$.
Therefore, 
\[
\begin{split}
\Big\langle [T^{(1)}_{\alpha}, T^{(1)}_{\beta}], [T^{(1)}_{\alpha}, T^{(1)}_{\beta}] \Big\rangle 
& = 4\Big\langle S^{(1)}_{\alpha_{1}}, S^{(1)}_{\alpha_{1}} \Big\rangle + 4 \Big\langle S^{(1)}_{\alpha_{k}}, S^{(1)}_{\alpha_{k}} \Big\rangle \\
& = 4 \frac{4}{(2e_{1}, 2e_{1})} + 4\frac{4}{(2e_{k}, 2e_{k})} = 4d + 4d = 8d.
\end{split}
\]
\end{proof}

\begin{lemm} \label{5-2-7}
Let $\Sigma$ be of type $C_{2}$, and $\alpha \in \tilde{\Sigma}$ be a longest root and $\beta \in \tilde{\Sigma}$ be a shortest root.
If $\bar{\alpha} = e_{1} + e_{2}$ and $\bar{\beta} = e_{1} - e_{2}$, then for any $i = 1,2$,
\[
\Big\langle \big[T^{(i)}_{\alpha}, T^{(1)}_{\beta} \big], \big[T^{(i)}_{\alpha}, T^{(1)}_{\beta} \big] \Big\rangle = 0.
\]
\end{lemm}

\begin{proof}
Note that $\sigma(\alpha) \not= \alpha$ and $\sigma(\beta) = \beta$, so $\sigma(\alpha + \beta) \not= \alpha + \beta$.
Moreover, $\overline{\alpha + \beta} = 2e_{1}$.
Hence, $\alpha + \beta \not\in \tilde{\Sigma}$ since $\sigma(\alpha_{1}) = \alpha_{1}$.
By a similar argument, we have $\alpha - \beta \not\in \tilde{\Sigma}$.
Therefore, $[ T^{(i)}_{\alpha}, T^{(1)}_{\beta} ] = 0$ for any $i =1,2$, and the statement follows.
\end{proof}

\begin{lemm} \label{5-2-8}
Let $\alpha, \beta \in \tilde{\Sigma}$ be longest roots.
If $\alpha \not= \beta, \sigma(\beta)$ and $\bar{\alpha} = \bar{\beta} = e_{1} + e_{k}\ (2 \leq k \leq n)$, then $\alpha - \beta, \alpha - \sigma(\beta) \in \tilde{\Sigma}$ and for any $i,j = 1,2$,
\[
\Big\langle \big[ T^{(i)}_{\alpha}, T^{(j)}_{\beta} \big], \big[ T^{(i)}_{\alpha}, T^{(j)}_{\beta} \big] \Big\rangle = 8d.
\]
\end{lemm}

\begin{proof}
The $\frak{a}_{0}$-part and $i\frak{b}$-part of $\alpha$ (resp.\ $\beta$) are denoted by $\lambda$ and $\mu_{1}$ (resp.\ $\lambda$ and $\mu_{2}$), respectively.
Assume that neither $\alpha - \beta$ nor $\alpha - \sigma(\beta)$ is a root.
Since $\alpha + \beta, \alpha + \sigma(\beta) \not\in \tilde{\Sigma}$, we obtain $(\alpha, \beta) = (\alpha, \sigma(\beta)) = 0$.
Hence,
\[
\begin{split}
0 &= (\alpha, \beta) = (\lambda + \mu_{1}, \lambda + \mu_{2}) = (\lambda, \lambda) + (\mu_{1}, \mu_{2}), \\
0 &= (\alpha, \sigma(\beta)) = (\lambda + \mu_{1}, \lambda - \mu_{2}) = (\lambda, \lambda) - (\mu_{1}, \mu_{2}).
\end{split}
\]
Thus, $(\lambda, \lambda) = 0$, which is a contradiction.
Next, assume $\alpha - \beta \in \tilde{\Sigma}$ and $\alpha - \sigma(\beta) \not\in \tilde{\Sigma}$.
Then,
\[
\frac{2(\alpha, \beta)}{(\alpha, \alpha)} = 1, \quad \frac{2(\alpha, \sigma(\beta))}{(\alpha, \alpha)} = 0.
\]
Therefore,
\[
\begin{split}
& (\lambda, \lambda) + (\mu_{1}, \mu_{1}) = 2(\lambda, \lambda) + 2(\mu_{1}, \mu_{2}), \quad (\lambda, \lambda) - (\mu_{1}, \mu_{2}) = 0,
\end{split}
\]
and $3(\lambda, \lambda) = (\mu_{1}, \mu_{1})$.
On the other hand, since $(\alpha, \sigma(\alpha)) = 0$ by $\alpha \pm \sigma(\alpha) \not\in \tilde{\Sigma}$, 
\[
0 = (\alpha, \sigma(\alpha)) = (\lambda, \lambda) - (\mu_{1}, \mu_{1})
\]
and $(\lambda, \lambda) = (\mu_{1}, \mu_{1})$.
Note that $\lambda \not = 0$ and $\mu_{1} \not= 0$ by $\sigma(\alpha) \not= \alpha$, and this is a contradiction.
In the case of $\alpha - \beta \not\in \tilde{\Sigma}$ and $\alpha - \sigma(\beta) \in \tilde{\Sigma}$, a similar argument shows that a contradiction arises.
Thus, the former part of the statement is true.
Next, we consider the latter part.
If $(\epsilon_{\alpha, - \beta}, \epsilon_{\alpha, -\sigma(\beta)}) = (1,1)$, then $N_{\alpha, -\beta}^{2} = N^{2}_{\alpha, -\sigma(\beta)} = 1$, and
\[
\begin{split}
&[T^{(1)}_{\alpha}, T^{(1)}_{\beta}] \\
&=
[W_{\alpha}, W_{\beta}] + [W_{\sigma(\alpha)}, W_{\sigma(\beta)}] + [W_{\alpha}, W_{\sigma(\beta)}]  + [W_{\sigma(\alpha)}, W_{\beta}] \\
&=
N_{\alpha, -\beta}U_{\alpha - \beta} + N_{\sigma(\alpha), -\sigma(\beta)}U_{\sigma(\alpha - \beta)} + N_{\alpha, -\sigma(\beta)}U_{\alpha - \sigma(\beta)} + N_{\sigma(\alpha), - \beta}U_{\sigma(\alpha) -\beta} \\
&=
2N_{\alpha, -\beta}S^{(1)}_{\alpha - \beta} + 2N_{\alpha, -\sigma(\beta)}S^{(1)}_{\alpha - \sigma(\beta)}, \\
&[T^{(1)}_{\alpha}, T^{(2)}_{\beta}] \\
&=
-[W_{\alpha}, U_{\beta}] + [W_{\sigma(\alpha)}, U_{\sigma(\beta)}] + [W_{\alpha}, U_{\sigma(\beta)}] - [W_{\sigma(\alpha)}, U_{\beta}] \\
&=
-N_{\alpha, -\beta}W_{\alpha - \beta} + N_{\sigma(\alpha), -\sigma(\beta)}W_{\sigma(\alpha - \beta)} + N_{\alpha, -\sigma(\beta)}W_{\alpha - \sigma(\beta)} - N_{\sigma(\alpha), - \beta}W_{\sigma(\alpha) -\beta} \\
&=
-2N_{\alpha, -\beta}S^{(2)}_{\alpha - \beta} + 2N_{\alpha, -\sigma(\beta)}S^{(2)}_{\alpha - \sigma(\beta)}, \\
&[T^{(2)}_{\alpha}, T^{(1)}_{\beta}] \\
&=
-[U_{\alpha}, W_{\beta}] + [U_{\sigma(\alpha)}, W_{\sigma(\beta)}] - [U_{\alpha}, W_{\sigma(\beta)}] + [U_{\sigma(\alpha)}, W_{\beta}] \\
&=
N_{\alpha, -\beta}W_{\alpha - \beta} - N_{\sigma(\alpha), -\sigma(\beta)}W_{\sigma(\alpha - \beta)} + N_{\alpha, -\sigma(\beta)}W_{\alpha - \sigma(\beta)} - N_{\sigma(\alpha), - \beta}W_{\sigma(\alpha) -\beta} \\
&=
2N_{\alpha, -\beta}S^{(2)}_{\alpha - \beta} + 2N_{\alpha, -\sigma(\beta)}S^{(2)}_{\alpha - \sigma(\beta)}, \\
& [T^{(2)}_{\alpha}, T^{(2)}_{\beta}] \\
&=
[U_{\alpha}, U_{\beta}] + [U_{\sigma(\alpha)}, U_{\sigma(\beta)}] - [U_{\alpha}, U_{\sigma(\beta)}] - [U_{\sigma(\alpha)}, U_{\beta}] \\
&=
N_{\alpha, -\beta}U_{\alpha - \beta} + N_{\sigma(\alpha), -\sigma(\beta)}U_{\sigma(\alpha - \beta)} -N_{\alpha, -\sigma(\beta)}U_{\alpha - \sigma(\beta)} - N_{\sigma(\alpha), - \beta}U_{\sigma(\alpha) -\beta} \\
&=
2N_{\alpha, -\beta}S^{(1)}_{\alpha - \beta} -2N_{\alpha, -\sigma(\beta)}S^{(1)}_{\alpha - \sigma(\beta)}.
\end{split}
\]
If $(\epsilon_{\alpha, - \beta}, \epsilon_{\alpha, -\sigma(\beta)}) = (1,-1)$, then $N_{\alpha, -\beta}^{2} = 1$ and $N^{2}_{\alpha, -\sigma(\beta)} = -1$, and 
\[
\begin{split}
& [T^{(1)}_{\alpha}, T^{(1)}_{\beta}] \\
&=
N_{\alpha, -\beta}U_{\alpha - \beta} + N_{\sigma(\alpha), -\sigma(\beta)}U_{\sigma(\alpha - \beta)} -iN_{\alpha, -\sigma(\beta)}W_{\alpha - \sigma(\beta)} - iN_{\sigma(\alpha), -\beta}W_{\sigma(\alpha) - \beta} \\
&=
2N_{\alpha, -\beta}S^{(1)}_{\alpha - \beta} -2iN_{\alpha, -\sigma(\beta)}S^{(2)}_{\alpha - \sigma(\beta)}, \\
& [T^{(1)}_{\alpha}, T^{(2)}_{\beta}] \\
&=
-N_{\alpha, -\beta}W_{\alpha - \beta} + N_{\sigma(\alpha), -\sigma(\beta)}W_{\sigma(\alpha - \beta)} -iN_{\alpha, -\sigma(\beta)}U_{\alpha - \sigma(\beta)} + iN_{\sigma(\alpha), -\beta}U_{\sigma(\alpha) - \beta} \\
&=
-2N_{\alpha, -\beta}S'_{\alpha - \beta} -2iN_{\alpha, -\sigma(\beta)}S^{(1)}_{\alpha - \sigma(\beta)}, \\
& [T^{(2)}_{\alpha}, T^{(1)}_{\beta}] \\
&=
N_{\alpha, -\beta}W_{\alpha - \beta} - N_{\sigma(\alpha), -\sigma(\beta)}W_{\sigma(\alpha - \beta)} + iN_{\alpha, -\sigma(\beta)}U_{\alpha - \sigma(\beta)} - iN_{\sigma(\alpha), -\beta}U_{\sigma(\alpha) - \beta} \\
&= 
2N_{\alpha, -\beta}S^{(2)}_{\alpha - \beta} + 2iN_{\alpha, -\sigma(\beta)}S^{(1)}_{\alpha - \sigma(\beta)}, \\
& [T^{(2)}_{\alpha}, T^{(2)}_{\beta}] \\
&=
N_{\alpha, -\beta}U_{\alpha - \beta} + N_{\sigma(\alpha), -\sigma(\beta)}U_{\sigma(\alpha - \beta)} + iN_{\alpha, -\sigma(\beta)}W_{\alpha - \sigma(\beta)} + iN_{\sigma(\alpha), -\beta}W_{\sigma(\alpha) - \beta} \\
&=
2N_{\alpha, -\beta}S^{(1)}_{\alpha - \beta} + 2iN_{\alpha, -\sigma(\beta)}S^{(2)}_{\alpha - \sigma(\beta)}. \\
\end{split}
\]
If $(\epsilon_{\alpha, - \beta}, \epsilon_{\alpha, -\sigma(\beta)}) = (-1,1)$, then $N_{\alpha, -\beta}^{2} = -1$ and $N^{2}_{\alpha, -\sigma(\beta)} = 1$, and 
\[
\begin{split}
& [T^{(1)}_{\alpha}, T^{(1)}_{\beta}] \\
&=
-iN_{\alpha, -\beta}W_{\alpha - \beta} - iN_{\sigma(\alpha), -\sigma(\beta)}W_{\sigma(\alpha - \beta)} + N_{\alpha, -\sigma(\beta)}U_{\alpha - \sigma(\beta)} + N_{\sigma(\alpha), - \beta}U_{\sigma(\alpha) -\beta} \\
&=
-2iN_{\alpha, -\beta}S'_{\alpha - \beta} + 2N_{\alpha, -\sigma(\beta)}S^{(1)}_{\alpha - \sigma(\beta)}, \\
& [T^{(1)}_{\alpha}, T^{(2)}_{\beta}] \\
& =
iN_{\alpha, -\beta}U_{\alpha - \beta} - iN_{\sigma(\alpha), -\sigma(\beta)}U_{\sigma(\alpha - \beta)} + N_{\alpha, -\sigma(\beta)}W_{\alpha - \sigma(\beta)} - N_{\sigma(\alpha), - \beta}W_{\sigma(\alpha) -\beta} \\
&=
2iN_{\alpha, -\beta}S^{(1)}_{\alpha - \beta} + 2N_{\alpha, -\sigma(\beta)}S'_{\alpha - \sigma(\beta)}, \\
& [T^{(2)}_{\alpha}, T^{(1)}_{\beta}] \\
&=
iN_{\alpha, -\beta}U_{\alpha - \beta} - iN_{\sigma(\alpha), -\sigma(\beta)}U_{\sigma(\alpha - \beta)} + N_{\alpha, -\sigma(\beta)}W_{\alpha - \sigma(\beta)} - N_{\sigma(\alpha), - \beta}W_{\sigma(\alpha) -\beta} \\
&=
2iN_{\alpha, -\beta}S^{(1)}_{\alpha - \beta} + 2N_{\alpha, -\sigma(\beta)}S'_{\alpha - \sigma(\beta)} , \\
& [T^{(2)}_{\alpha}, T^{(2)}_{\beta}] \\
&=
-iN_{\alpha, -\beta}W_{\alpha - \beta} - iN_{\sigma(\alpha), -\sigma(\beta)}W_{\sigma(\alpha - \beta)} -N_{\alpha, -\sigma(\beta)}U_{\alpha - \sigma(\beta)} - N_{\sigma(\alpha), - \beta}U_{\sigma(\alpha) -\beta} \\
&=
-2iN_{\alpha, -\beta}S'_{\alpha - \beta} + -2N_{\alpha, -\sigma(\beta)}S^{(1)}_{\alpha - \sigma(\beta)}. \\
\end{split}
\]
If $(\epsilon_{\alpha, - \beta}, \epsilon_{\alpha, -\sigma(\beta)}) = (-1,-1)$, then $N_{\alpha, -\beta}^{2} = N^{2}_{\alpha, -\sigma(\beta)} = -1$, and 
\[
\begin{split}
& [T^{(1)}_{\alpha}, T^{(1)}_{\beta}] \\
&=
-iN_{\alpha, -\beta}W_{\alpha - \beta} - iN_{\sigma(\alpha), -\sigma(\beta)}W_{\sigma(\alpha - \beta)} -iN_{\alpha, -\sigma(\beta)}W_{\alpha - \sigma(\beta)} - iN_{\sigma(\alpha), -\beta}W_{\sigma(\alpha) - \beta} \\
&=
-2iN_{\alpha, -\beta}S^{(2)}_{\alpha - \beta} -2iN_{\alpha, -\sigma(\beta)}S^{(2)}_{\alpha - \sigma(\beta)}, \\
& [T^{(1)}_{\alpha}, T^{(2)}_{\beta}]] \\
&=
iN_{\alpha, -\beta}U_{\alpha - \beta} - iN_{\sigma(\alpha), -\sigma(\beta)}U_{\sigma(\alpha - \beta)} -iN_{\alpha, -\sigma(\beta)}U_{\alpha - \sigma(\beta)} + iN_{\sigma(\alpha), -\beta}U_{\sigma(\alpha) - \beta} \\
&=
2iN_{\alpha, -\beta}S^{(1)}_{\alpha - \beta} -2iN_{\alpha, -\sigma(\beta)}S^{(1)}_{\alpha - \sigma(\beta)}, \\
& [T^{(2)}_{\alpha}, T^{(1)}_{\beta}] \\
&=
iN_{\alpha, -\beta}U_{\alpha - \beta} - iN_{\sigma(\alpha), -\sigma(\beta)}U_{\sigma(\alpha - \beta)} + iN_{\alpha, -\sigma(\beta)}U_{\alpha - \sigma(\beta)} - iN_{\sigma(\alpha), -\beta}U_{\sigma(\alpha) - \beta} \\
&=
2iN_{\alpha, -\beta}S^{(1)}_{\alpha - \beta} + 2iN_{\alpha, -\sigma(\beta)}S^{(1)}_{\alpha - \sigma(\beta)}, \\
& [T^{(2)}_{\alpha}, T^{(2)}_{\beta}] \\
&=
-iN_{\alpha, -\beta}W_{\alpha - \beta} - iN_{\sigma(\alpha), -\sigma(\beta)}W_{\sigma(\alpha - \beta)} + iN_{\alpha, -\sigma(\beta)}W_{\alpha - \sigma(\beta)} + iN_{\sigma(\alpha), -\beta}W_{\sigma(\alpha) - \beta} \\
&=
-2iN_{\alpha, -\beta}S^{(2)}_{\alpha - \beta} + 2iN_{\alpha, -\sigma(\beta)}S^{(2)}_{\alpha - \sigma(\beta)}. \\
\end{split}
\]
On the other hand,
\[
\Big\langle S^{(1)}_{\alpha - \beta}, S^{(1)}_{\alpha - \beta} \Big\rangle = \Big\langle U_{\alpha - \beta}, U_{\alpha - \beta} \Big\rangle = \Big\langle iA_{\alpha - \beta}, iA_{\alpha - \beta} \Big\rangle = \Big\langle iA_{\alpha_{1}}, iA_{\alpha_{1}} \Big\rangle = (2e_{1}, 2e_{1}) = d.
\]
Since $\langle S^{(i)}_{\alpha -\beta}, S^{(i)}_{\alpha -\beta} \rangle = \langle S^{(j)}_{\alpha - \sigma(\beta)}, S^{(j)}_{\alpha - \sigma(\beta)} \rangle$ for any $i,j = 1,2$, we have completed the proof.
\end{proof}

\begin{lemm} \label{5-2-9}
Let  $\Sigma$ be of type $C_{2}$, and $\alpha \in \tilde{\Sigma}$ be a longest root and $\beta \in \tilde{\Sigma}$ be a shortest root.
If  $\bar{\alpha} = \bar{\beta} = e_{1} + e_{2}$, then $\alpha - \beta \in \tilde{\Sigma}$, and for any $i = 1,2$,
\[
\Big\langle \big[ T^{(i)}_{\alpha}, T^{(1)}_{\beta} \big], \big[ T^{(i)}_{\alpha}, T^{(1)}_{\beta} \big] \Big\rangle = 8d.
\]
\end{lemm}

\begin{proof}
Since $(\alpha, \beta) \not= 0$ by $\sigma(\beta) = \beta$ and $\alpha + \beta \not\in \tilde{\Sigma}$, we have $\alpha - \beta \in \tilde{\Sigma}$.
If $\epsilon_{\alpha, -\beta} = 1$, then $N^{2}_{\alpha, -\beta} = 1$ and 
\[
\begin{split}
[T_{\alpha}^{(1)}, T^{(1)}_{\beta}] 
&= [W_{\alpha}, W_{\beta}] + [W_{\sigma(\alpha)}, W_{\beta}] 
= N_{\alpha, -\beta}U_{\alpha - \beta} + N_{\sigma(\alpha), -\beta}U_{\sigma(\alpha - \beta)} \\
&= 2N_{\alpha, -\beta}S^{(1)}_{\alpha - \beta}, \\
[T_{\alpha}^{(2)}, T^{(1)}_{\beta}]
&= -[U_{\alpha}, W_{\beta}] + [U_{\sigma(\alpha)}, W_{\beta}] 
= N_{\alpha, -\beta}W_{\alpha - \beta} - N_{\sigma(\alpha), -\beta}W_{\sigma(\alpha - \beta)} \\
&= 2N_{\alpha, -\beta}S^{(2)}_{\alpha - \beta}.
\end{split}
\]
If $\epsilon_{\alpha, -\beta} = -1$, then $N^{2}_{\alpha, -\beta} = -1$ and
\[
\begin{split}
[T_{\alpha}^{(1)}, T^{(1)}_{\beta}] 
&= -iN_{\alpha, -\beta}W_{\alpha - \beta} -iN_{\sigma(\alpha), -\beta}W_{\sigma(\alpha - \beta)}
= -2iN_{\alpha, -\beta}S^{(2)}_{\alpha - \beta}, \\
[T_{\alpha}^{(2)}, T^{(1)}_{\beta}]
&= -iN_{\alpha, -\beta}U_{\alpha - \beta} + iN_{\sigma(\alpha), -\beta}U_{\sigma(\alpha - \beta)} 
= -2iN_{\alpha, -\beta}S^{(1)}_{\alpha - \beta}.
\end{split}
\]
Since $|\alpha - \beta| = |\beta|$ and $\sqrt{2}|\alpha - \beta| = |\alpha|$,
\[
\Big\langle S^{(1)}_{\alpha - \beta}, S^{(1)}_{\alpha - \beta} \Big\rangle = \Big\langle U_{\alpha - \beta}, U_{\alpha - \beta} \Big\rangle = \Big\langle iA_{\alpha - \beta}, iA_{\alpha - \beta} \Big\rangle = 2 \Big\langle iA_{\alpha_{1}}, iA_{\alpha_{1}} \Big\rangle = 2d.
\]
Similarly, $\langle S^{(2)}_{\alpha - \beta}, S^{(2)}_{\alpha - \beta} \rangle = 2d$ and the statement follows.
\end{proof}

\begin{remark} \label{5-2-10}
Let $\alpha \in \tilde{\Sigma}$ be a longest root such that $\bar{\alpha} = e_{1} + e_{k} \ (2 \leq k \leq n)$.
Since $(\alpha, \sigma(\alpha)) = 0$,
\[
\begin{split}
[T^{(1)}_{\alpha}, T^{(2)}_{\alpha}] 
&= [W_{\alpha} + W_{\sigma(\alpha)}, -U_{\alpha} + U_{\sigma(\alpha)}] \\
&= -[W_{\alpha}, U_{\alpha}] + [W_{\sigma(\alpha)}, U_{\sigma(\alpha)}] = -2iA_{\alpha} + 2iA_{\sigma(\alpha)}.
\end{split}
\]
Therefore,
\[
\Big\langle [T^{(1)}_{\alpha}, T^{(2)}_{\alpha}], [T^{(1)}_{\alpha}, T^{(2)}_{\alpha}] \Big\rangle 
= \Big\langle 2iA_{\alpha}, 2iA_{\alpha} \Big\rangle + \Big\langle 2iA_{\sigma(\alpha)}, 2iA_{\sigma(\alpha)} \Big\rangle
= 4d + 4d = 8d.
\]
\end{remark}

\begin{lemm} \label{5-2-11}
Let $\Sigma$ be of type $BC_{n}\ (n \geq 2)$.
Moreover, let $\alpha, \beta \in \tilde{\Sigma}$ be longest roots such that $\bar{\alpha} = e_{1} + e_{k}\ ( 2 \leq k \leq n)$ and  $\bar{\beta} = e_{1}$.
Then, either $\beta$ or $\sigma(\beta)$ is orthogonal to $\alpha$, while the other is not orthogonal to $\alpha$.
Moreover, for any $i,j = 1,2$,
\[
\Big\langle [T^{(i)}_{\alpha}, T^{(j)}_{\beta}], [T^{(i)}_{\alpha}, T^{(j)}_{\beta}] \Big\rangle = 4d.
\]
\end{lemm}

\begin{proof}
The $\frak{a}_{0}$-part and $i\frak{b}$-part of $\alpha$ (resp.\ $\beta$) are denoted by $\lambda_{1}$ and $\mu_{1}$ (resp.\ $\lambda_{2}$ and $\mu_{2}$), respectively.
Assume $(\alpha, \beta) = (\alpha, \sigma(\beta)) = 0$.
Then,
\[
\begin{split}
0 &= (\alpha, \beta) = (\lambda_{1}, \lambda_{2}) + ( \mu_{1}, \mu_{2} ), \quad 
0 = (\alpha, \sigma(\beta)) = (\lambda_{1}, \lambda_{2}) - (\mu_{1}, \mu_{2}).
\end{split}
\]
Hence, $(\lambda_{1}, \lambda_{2}) = 0$, which is a contradiction.
If neither $(\alpha, \beta)$ nor $(\alpha, \sigma(\beta))$ is $0$, then 
\[
\begin{split}
\frac{2(\alpha, \beta)}{(\alpha, \alpha)} = \frac{2(\alpha, \sigma(\beta))}{(\alpha, \alpha)} = 1.
\end{split}
\]
Since
\[
(\alpha, \beta) = (\lambda_{1}, \lambda_{2}) + ( \mu_{1}, \mu_{2} ), \quad (\alpha, \sigma(\beta)) = (\lambda_{1}, \lambda_{2}) - (\mu_{1}, \mu_{2}),
\]
we obtain $(\mu_{1}, \mu_{2}) = 0$.
On the other hand, $\alpha - \beta \in \tilde{\Sigma}$ and 
\[
\alpha - \beta = (\lambda_{1} - \lambda_{2}) + (\mu_{1} - \mu_{2}).
\]
Thus, $|\alpha - \beta| > |\beta| = |\alpha|$, which is a contradiction.
Hence, the former part of the statement follows.
The latter part can be proven by a similar argument to the proof of Lemma \ref{5-2-2}.
\end{proof}

%%%%%%%%%%%%%%%%%%%%%%%%%%%%%%%%%%%%%%%%%%%%%%%%%%%%%%%%%%%%%%%%%%%%%%%%%%%%%%%%%%%%%%%%%%%%%%%%%%%%%%%%%%%%%%%%%%%%%%%%%%%%%%%%%%%%%%%%%%%%%%%%%%%%%%%

\subsection{Sasaki-Einstein $CR$ orbits} \label{s5-2}

In this subsection, we study the Einstein constant of $(M^{+}, (1/(d+1)) \langle \ ,\ \rangle)$ for each irreducible Hermitian symmetric space $M$ of compact type, except for $G_{k}(\mathbb{C}^{n})\ (k<n, \ k \not= 1, \ n \not= 2k)$.
In $\frak{m}^{+}_{1}$, the set
\[
\begin{split}
&\left\{
\frac{1}{\sqrt{d+1}}\bar{T}^{(i)}_{\alpha} \ ;\ \alpha \in \tilde{\Sigma}^{\sigma}_{\lambda}, \ \lambda \in (\Sigma^{+}_{x})_{C}, \ \sigma(\alpha) \not= \alpha, \ i = 1,2 
\right\} \\
& \quad\quad\quad\quad\quad\quad\quad\quad\quad\quad \cup  
\left\{
\frac{1}{\sqrt{d+1}}\bar{T}^{(1)}_{\beta} \ ;\ \beta \in \tilde{\Sigma}^{\sigma}_{\lambda}, \ \lambda \in (\Sigma^{+}_{x})_{C}, \ \sigma(\beta) = \beta 
\right\}
\end{split}
\]
forms an orthonormal basis with respect to $(1/(d+1))\langle \ ,\ \rangle$.
Denote $\mathrm{exp}(C_{0}(1))$ by $g$.

\begin{itemize}

\item[(i)]
$M = G_{n}(\mathbb{C}^{2n})\ (n \geq 2)$

$\tilde{\Sigma}$ is of type $A_{2n-1}$ and $\Sigma$ is of type $C_{n}$.
Then, $M^{+}_{1} = \mathbb{C}P^{n-1} \times \mathbb{C}P^{n-1}$.
Let $\alpha \in \tilde{\Sigma}$ be a longest root such that $\bar{\alpha} = e_{1} + e_{k}\ (2 \leq k \leq n)$.
By Lemma \ref{5-2-2}, Lemma \ref{5-2-5}, Lemma \ref{5-2-8}, and Remark \ref{5-2-10},
\[
\begin{split}
& r^{+}(g_{*}T^{(1)}_{\alpha}, g_{*}T^{(1)}_{\alpha}) \\
&=
\sum_{\lambda \in (\Sigma^{+}_{x})_{C}}\sum_{\beta \in \tilde{\Sigma}^{\sigma}_{\lambda}, i=1,2} \Big\langle \Big[ \bar{T}^{(i)}_{\beta}, T^{(1)}_{\alpha} \Big], \Big[ \bar{T}^{(i)}_{\beta}, T^{(1)}_{\alpha} \Big] \Big\rangle \\
&=
\sum_{\lambda \in (\Sigma^{+}_{x})_{C}}\sum_{\beta \in \tilde{\Sigma}^{\sigma}_{\lambda}, i=1,2} \frac{(\lambda, \lambda)}{4} \Big\langle \Big[ T^{(i)}_{\beta}, T^{(1)}_{\alpha} \Big], \Big[ T^{(i)}_{\beta}, T^{(1)}_{\alpha} \Big] \Big\rangle \\
&=
\frac{1}{4} \Big( \sum_{l = 2, l \not= k}^{n} \Big( \big( e_{1} + e_{l}, e_{1} + e_{l} \big)m(e_{1} + e_{l}) 2d + \big( e_{1} - e_{l}, e_{1} - e_{l} \big)m(e_{1} - e_{l})2d \Big) \\
& \quad\quad\quad + \big( e_{1} - e_{k}, e_{1} - e_{k} \big)8d + \big( e_{1} + e_{k}, e_{1} + e_{k} \big) 8d \Big)  = 4n, \\
& \frac{1}{d+1} \langle g_{*}T^{(1)}_{\alpha}, g_{*}T^{(1)}_{\alpha} \rangle 
=
\frac{1}{d+1} \frac{4}{(e_{1} + e_{k}, e_{1} + e_{k})} 
=
\frac{2d}{d+1}
\end{split}
\]
Therefore, the Einstein constant of $(M^{+}, (1/(d+1)) \langle \ ,\ \rangle)$ is $2n(d+1)/d$.
Since $\dim M^{+}_{1} = 4(n-1)$, a Sasaki $CR$ orbit is Einstein if and only if $d = n/(n-1)$.

\ 

\item[(ii)]
$M = \tilde{G}_{2}(\mathbb{R}^{n})\ (n =2m, m \geq 3)$

$\tilde{\Sigma}$ is of type $D_{m}$ and $\Sigma$ is of type $C_{2}$.
Then, $M^{+}_{1} = \tilde{G}_{2}(\mathbb{R}^{n-2})$.
Let $\alpha \in \tilde{\Sigma}$ be a longest root such that $\bar{\alpha} = e_{1} + e_{2}$.
By Lemma \ref{5-2-5}, Lemma \ref{5-2-8}, and Remark \ref{5-2-10},
\[
\begin{split}
r^{+}(g_{*}T^{1}_{\alpha}, g_{*}T^{(1)}_{\alpha})
&=
\frac{1}{4} \Big( (e_{1} + e_{2}, e_{1} + e_{2})(m(e_{1} + e_{2}) - 1)8d + (e_{1} - e_{2}, e_{1} - e_{2})8d \Big) \\
&=
\frac{1}{4} \Big( \frac{2}{d} 8d (2(m-2) -1 ) + \frac{2}{d} 8d \Big)  = 4(n-4). \\
\end{split}
\]
Therefore, the Einstein constant of $(M^{+}, (1/(d+1)) \langle \ ,\ \rangle)$ is $2(n-4)(d+1)/d$.
Since $\dim M^{+}_{1} = 2(n-4)$, a Sasaki $CR$ orbit is Einstein if and only if $d = n-4$.

\ 

\item[(iii)]
$M = \tilde{G}_{2}(\mathbb{R}^{n})\ (n = 2m+1, m \geq 2)$

$\tilde{\Sigma}$ is of type $B_{m}$ and $\Sigma$ is of type $C_{2}$.
Then, $M^{+}_{1} = \tilde{G}_{2}(\mathbb{R}^{n-2})$.
Let $\alpha \in \tilde{\Sigma}$ be a longest root such that $\bar{\alpha} = e_{1} + e_{2}$.
By Lemma \ref{5-2-5}, Lemma \ref{5-2-7}, Lemma \ref{5-2-9}, and Remark \ref{5-2-10},  
\[
\begin{split}
r(g_{*}T_{\alpha}^{(1)}, g_{*}T_{\alpha}^{(1)})
&=
\frac{1}{4} \Big( (e_{1} + e_{2}, e_{1} + e_{2})(m(e_{1} + e_{2}) - 1)8d + (e_{1} - e_{2}, e_{1} - e_{2})8d \Big) \\
&=
\frac{1}{4} \Big( \frac{2}{d} ( 8d (2(m-2)) ) + \frac{2}{d} 8d \Big) 
= 4 (2m - 3) = 4(n -4). \\
\end{split}
\]
Therefore, the Einstein constant of $(M^{+}, (1/(d+1)) \langle \ ,\ \rangle)$ is $2(n-4)(d+1)/d$.
Since $\dim M^{+}_{1} = 2(n-4)$, a Sasaki $CR$ orbit is Einstein if and only if $d = n-4$.

\ 

\item[(iv)]
$M = SO(2n)/U(n)\ (n =2m, m \geq 2)$

$\tilde{\Sigma}$ is of type $D_{n}$ and $\Sigma$ is of type $C_{m}$.
Then, $M^{+}_{1} = G_{2}(\mathbb{C}^{n})$.
Let $\alpha \in \tilde{\Sigma}$ be a longest root such that $\bar{\alpha} = e_{1} + e_{k}\ ( 2 \leq k \leq m)$.
By Lemma \ref{5-2-2}, Lemma \ref{5-2-5}, Lemma \ref{5-2-8}, and Remark \ref{5-2-10}, 
\[
\begin{split}
& r(g_{*}T_{\alpha}^{(1)}, g_{*}T_{\alpha}^{(1)}) \\
&=
\frac{1}{4} \Big( \sum_{l = 2, l \not= k}^{m}\Big( (e_{1} + e_{l}, e_{1} + e_{l})m(e_{1} + e_{l})2d + (e_{1} - e_{l}, e_{1} - e_{l})m(e_{1} - e_{l})2d \Big) \\
& \quad\quad\quad + (e_{1} + e_{k}, e_{1} + e_{k})(m(e_{1} + e_{k}) - 1)8d + (e_{1} - e_{k}, e_{1} - e_{k})8d \Big) \\
&=
\frac{1}{4} \Big( \Big( (m-2)\frac{2}{d}8d + (m-2)\frac{2}{d}8d \Big) + \frac{2}{d} 24 d + \frac{2}{d} 8d \Big) = 8m = 4n. \\
\end{split}
\]
Therefore, the Einstein constant of $(M^{+}, (1/(d+1)) \langle \ ,\ \rangle)$ is $2n(d+1)/d$.
Since $\dim M^{+}_{1} = 4(n-2)$, a Sasaki $CR$ orbit is Einstein if and only if $d = n/(n-3)$.

\ 

\item[(v)]
$M = SO(2n)/U(n)\ (n =2m+1, m \geq 1)$

$\tilde{\Sigma}$ is of type $D_{n}$ and $\Sigma$ is of type $BC_{m}$.
Then, $M^{+}_{1} = G_{2}(\mathbb{C}^{n})$.
Let $\alpha \in \tilde{\Sigma}$ be a longest root such that $\bar{\alpha} = e_{1} + e_{k}\ ( 2 \leq k \leq m)$.
By Lemma \ref{5-2-2}, Lemma \ref{5-2-5}, Lemma \ref{5-2-8}, Remark \ref{5-2-10}, and Lemma \ref{5-2-11},
\[
\begin{split}
& r(g_{*}T_{\alpha}^{(1)}, g_{*}T_{\alpha}^{(1)}) \\
&=
\frac{1}{4} \Big( \sum_{l = 2, l \not= k}^{m}\Big( (e_{1} + e_{l}, e_{1} + e_{l})m(e_{1} + e_{l})2d + (e_{1} - e_{l}, e_{1} - e_{l})m(e_{1} - e_{l})2d \Big) \\
& \quad + (e_{1} + e_{k}, e_{1} + e_{k})(m(e_{1} + e_{k}) - 1)8d + (e_{1} - e_{k}, e_{1} - e_{k})8d + (e_{1}, e_{1})m(e_{1})4d \Big) \\
&=
\frac{1}{4} \Big( \Big( (m-2)\frac{2}{d}8d + (m-2)\frac{2}{d}8d \Big) + \frac{2}{d} 24 d + \frac{2}{d} 8d + \frac{1}{d}16d \Big)  = 8m + 4 = 4n. \\
\end{split}
\]
Therefore, the Einstein constant of $(M^{+}, (1/(d+1)) \langle \ ,\ \rangle)$ is $2n(d+1)/d$.
Since $\dim M^{+}_{1} = 4(n-2)$, if $n \geq 4$, then a Sasaki $CR$ orbit is Einstein if and only if $d = n/(n-3)$.
If $n = 3$, there are no Sasaki-Einstein $CR$ orbits for any invariant metrics.
Note that $SO(6)/U(3)$ is diffeomorphic to $\mathbb{C}P^{2}$, and there are no Sasaki-Einstein $CR$ orbits in $\mathbb{C}P^{2}$ by Lemma \ref{5-1-2}.

\ 

\item[(vi)]
$M = Sp(n)/U(n)\ (n \geq 2)$

Both of $\tilde{\Sigma}$ and $\Sigma$ are of type $C_{n}$.
Then, $M^{+}_{1} = \mathbb{C}P^{n-1}$.
Let $\alpha \in \tilde{\Sigma}$ be the shortest root such that $\bar{\alpha} = e_{1} + e_{k}\ ( 2 \leq k \leq m)$.
By Lemma \ref{5-2-3} and Lemma \ref{5-2-6},
\[
\begin{split}
& r(g_{*}T_{\alpha}^{(1)}, g_{*}T_{\alpha}^{(1)}) \\
&=
\frac{1}{4} \Big( \sum_{l=2, l \not= k}^{n} \Big( (e_{1} + e_{l}, e_{1} + e_{l})m(e_{1} + e_{l})2d + (e_{1} - e_{l}, e_{1} - e_{l})m(e_{1} - e_{l})2d \Big) \\
& \quad\quad\quad + (e_{1} -  e_{k}, e_{1} - e_{k})m(e_{1} - e_{k})16d \Big) 
=
\frac{1}{4} \Big( (n-2)\Big( \frac{2}{d}2d + \frac{2}{d}2d \Big) + \frac{2}{d} 8d \Big) = 2n. \\
\end{split}
\]
Therefore, the Einstein constant of $(M^{+}, (1/(d+1)) \langle \ ,\ \rangle)$ is $n(d+1)/d$.
Since $\dim M^{+}_{1} = 2(n-1)$, a Sasaki $CR$ orbit is Einstein if and only if $d = 1$.

\ 

\item[(vii)]
$M = EIII$

$\tilde{\Sigma}$ is of type $E_{6}$ and $\Sigma$ is of type $BC_{2}$.
Then, $M^{+} = SO(10)/U(5)$.
Let $\alpha \in \tilde{\Sigma}$ be a longest root such that $\bar{\alpha} = e_{1} + e_{2}$.
By Lemma \ref{5-2-5}, Lemma \ref{5-2-8}, Remark \ref{5-2-10}, and Lemma \ref{5-2-11}, 
\[
\begin{split}
& r(g_{*}T_{\alpha}^{(1)}, g_{*}T_{\alpha}^{(1)}) \\
&=
\frac{1}{4} \Big( (e_{1} + e_{2}, e_{1} + e_{2})(m(e_{1} + e_{2}) - 1)8d + (e_{1} - e_{2}, e_{1} - e_{2})8d \Big) + (e_{1}, e_{1})m(e_{1})4d \Big) \\
&=
\frac{1}{4} \Big( \frac{2}{d} 8d \cdot 5  + \frac{2}{d} 8d + \frac{1}{d} 8 \cdot 4d \Big) 
= 32. \\
\end{split}
\]
Therefore, the Einstein constant of $(M^{+}, (1/(d+1)) \langle \ ,\ \rangle)$ is $16(d+1)/d$.
Since $\dim M^{+}_{1} = 20$, a Sasaki $CR$ orbit is Einstein if and only if $d = 8/3$.

\

\item[(viii)]
$M = EVII$

$\tilde{\Sigma}$ is of type $E_{7}$ and $\Sigma$ is of type $C_{3}$.
Then, $M^{+} = EIII$.
Let $\alpha \in \tilde{\Sigma}$ be a longest root such that $\bar{\alpha} = e_{1} + e_{2}$.
By Lemma \ref{5-2-2}, Lemma \ref{5-2-5}, Lemma \ref{5-2-8}, and Remark \ref{5-2-10},
\[
\begin{split}
&r(g_{*}T_{\alpha}^{(1)}, g_{*}T_{\alpha}^{(1)}) \\
&=
\frac{1}{4} \Big( (e_{1} + e_{2}, e_{1} + e_{2})(m(e_{1} + e_{2}) - 1)8d + (e_{1} - e_{2}, e_{1} - e_{2})8d \\
& \quad + (e_{1} + e_{3}, e_{1} + e_{3})m(e_{1} + e_{3})2d + (e_{1} - e_{3}, e_{1} - e_{3})m(e_{1} - e_{3})2d \Big) \\
&=
\frac{1}{4} \Big( \frac{2}{d} 7 \cdot 8d + \frac{2}{d} 8d + \frac{2}{d} 8 \cdot 2d + \frac{2}{d} 8 \cdot 2d \Big) = 48. \\
\end{split}
\]
Therefore, the Einstein constant of $(M^{+}, (1/(d+1)) \langle \ ,\ \rangle)$ is $24(d+1)/d$.
Since $\dim M^{+}_{1} = 32$, a Sasaki $CR$ orbit is Einstein if and only if $d = 12/5$.

\end{itemize}

Summarizing these arguments, we obtain Theorem \ref{5-3-1}.

\begin{thm} \label{5-3-1}
For the following irreducible Hermitian symmetric space $M$ of compact type, we define $d > 0$ as follows:
\[
\begin{array}{c|ccccccccccccccccccccccccccccc} \hline
M & G_{n}(\mathbb{C}^{2n}) & \tilde{G}_{2}(\mathbb{R}^{n}) & SO(2n)/U(n) & Sp(n)/U(n) & EIII & EVII \\ 
& (n \geq 2) & (n \geq 5) & (n \geq 4) & (n \geq 2) & & \\  \hline
d & n/(n-1) & n-4 & n/(n-3) & 1 & 8/3 & 12/5 \\ \hline
\end{array}
\]
For each $M$ and an invariant metric $\langle \ ,\ \rangle$ on $M$, the Sasaki $CR$ orbit is Einstein with respect to the induced metric if and only if the length of any shortest closed geodesic of $M$ with respect to $\langle \ ,\ \rangle$ is $\sqrt{d} \pi$.
For any irreducible Hermitian symmetric space $M$ of compact type not mentioned above and any invariant metric on $M$, there are no Sasaki-Einstein $CR$ orbits.
\end{thm}

By Corollary \ref{4-3-3}, we immediately obtain Corollary \ref{5-3-2}.

\begin{coro} \label{5-3-2}
If $M = Sp(n)/U(n)\ (n \geq 2)$ and $d = 1$, then the Sasaki-Einstein $CR$ orbit is a ruled $CR$ submanifold.
\end{coro}

\end{document}